 \documentclass[]{article}
\usepackage[utf8]{inputenc}

\usepackage{pdfpages}

\usepackage{euler} 

\usepackage[Conny]{fncychap}

\usepackage[T1]{fontenc}
\usepackage[english]{babel}
\usepackage{mathrsfs}
\usepackage{amsmath} 
\usepackage{amsthm}  

\usepackage{amssymb,amsfonts,amscd}

\usepackage{comment}

\usepackage{xfrac}
\usepackage{faktor}
\usepackage{mathtools}
\usepackage{tikz}
\usetikzlibrary{matrix,arrows,decorations.pathmorphing,decorations.markings, cd, scopes, backgrounds}

\makeatletter
\tikzset{
	open/.code     = {\tikzset{right hook->, circled};},
	closed/.code   = {\tikzset{right hook->, slashed};},
	open'/.code    = {\tikzset{left hook->, circled};},
	closed'/.code  = {\tikzset{left hook->, slashed};},
	circled/.code  = {\tikzset{markwith = {\draw (0,0) circle (.375ex);}};},
	slashed/.code  = {\tikzset{markwith = {\draw[-] (-.4ex,-.4ex) -- (.4ex,.4ex);}};},
	markwith/.code ={
		\pgfutil@ifundefined%
		{tikz@library@decorations.markings@loaded}%
		{\pgfutil@packageerror{tikz}{You need to say %
				\string\usetikzlibrary{decorations.markings} to use arrows with markings}{}}{}%
		\pgfkeysalso{/tikz/postaction = {
				/tikz/decorate,
				/tikz/decoration={markings, mark = at position 0.5 with {#1}}}
		}
	},
}
\makeatother

\usepackage{centernot}
\usepackage{bbding}
\usepackage{indentfirst}
\usepackage{stackrel}
\usepackage{scalerel}
\usepackage{hyperref}
\usepackage{cleveref}
\usepackage{mathdots}
\usepackage{caption}
\usepackage{geometry}
\usepackage{etoolbox}
\usepackage{calligra}
\usepackage{leftidx}
\usepackage{tcolorbox}
\usepackage{bbm}
\usepackage{bm}
\usepackage[autostyle]{csquotes}
\usepackage{epigraph}

\usepackage{diagbox}

\usepackage{graphicx}

\usepackage{stmaryrd}

\usepackage{mathrsfs}
\DeclareMathAlphabet{\mathpzc}{OT1}{pzc}{m}{it}

\usepackage[toc,page]{appendix}

\usepackage[style=alphabetic, backend=bibtex, useprefix, hyperref]{biblatex}

\NewBibliographyString{appendix,by}
\DefineBibliographyStrings{english}{
	appendix = {with an appendix},
	by       = {by},
}

\newbibmacro*{appendixby}{%
	\iffieldundef{related}
	{}
	{\setunit{\addspace}%
		\bibstring{appendix}%
		\setunit{\addspace}%
		\bibstring{by}%
		\setunit{\addspace}%
		\entrydata{\strfield{related}}{\printnames{author}}}}

\AtEveryBibitem{%
	\ifentrytype{article}{%
		\usebibmacro{appendixby}%
	}{}%
}

\hypersetup{colorlinks=false, pdfborder={0 0 0}}
\usepackage[]{enumitem}
\setdescription{font=\normalfont}       
\makeatletter                                                  
\def\cleardoublepage{\clearpage\if@twoside \ifodd\c@page\else  
	\hbox{}                                                        
	\vspace*{\fill}                                                
	\begin{center}                                                 
		\*                                                             
	\end{center}                                                   
	\vspace{\fill}                                                 
	\thispagestyle{empty}                                          
	\newpage                                                       
	\if@twocolumn\hbox{}\newpage\fi\fi\fi}                         
\makeatother                                                   

\makeatletter
\renewcommand\part{%
	\if@openright
	\cleardoublepage
	\else
	\clearpage
	\fi
	\thispagestyle{empty}
	\if@twocolumn
	\onecolumn
	\@tempswatrue
	\else
	\@tempswafalse
	\fi
	\null\vfil
	\secdef\@part\@spart}
\makeatother

\usepackage{fancyhdr}                                   

\usepackage{xcolor}
\usepackage{url}
\usepackage{attachfile} 
\makeatletter                        
\g@addto@macro{\UrlBreaks}{\UrlOrds} 
\makeatother                         
\newsavebox\MBox
\makeatletter                        
\newcommand{\dotminus}{\mathbin{\text{\@dotminus}}}

\newcommand{\@dotminus}{%
	\ooalign{\hidewidth\raise1ex\hbox{.}\hidewidth\cr$\m@th-$\cr}%
}
\makeatother

\newcommand{\into}{\hookrightarrow}

\newcommand{\allora}{\Rightarrow}

\newcommand{\sseq}{\subseteq}

\newcommand{\set}[1]{ \left \{ #1 \right \} }

\newcommand{\mr}{\mathrm}
\newcommand{\mc}[1]{\mathcal{#1}}
\newcommand{\mb}{\mathbb}
\newcommand{\mbbm}{\mathbbm}
\newcommand{\mf}{\mathfrak}

\renewcommand{\ker}{\operatorname{Ker}}

\newcommand{\N}{\mathbb{N}}

\newcommand{\Z}{\mathbb{Z}}

\newcommand{\pro}{\mathbb{P}}
\newcommand{\A}{\mathbb{A}}
\newcommand{\sk}{\mathbbm{k}}
\newcommand{\st}[2]{\mathrel{\raisebox{#1 pt}{$ \mathrel{\stretchto{\mid}{#2 ex}} $}}}

\newcommand{\mo}[1]{\mc O_{#1}}

\newcommand{\catname}[1]{\mathbf{#1}}

\newcommand{\oocatname}[1]{\scaleobj{1.25}{\mathpzc{#1}}}

\newcommand{\colim}[1]{\underset{#1}{\mathrm{colim}}\ }

\newcommand{\epf}{{}_{!}}
\newcommand{\epfs}{{}_{\#}}

\newcommand{\Th}[2]{\mr{Th}_{#1}\left( #2 \right)}

\newcommand{\restrict}[2]{{
		\left.\kern-\nulldelimiterspace 
		#1 
		\vphantom{\big|} 
		\right|_{#2} 
}}

\newcommand{\bigslant}[2]{
	\mathchoice
	{
		{\raisebox{0em}{$#1$}\!\!\!\;\!\!\;\left/\!\!\;\!\!\;\raisebox{-0em}{$#2$}\right.}%
	}
	{
		#1\:\!/\:\!#2
	}
	{
		#1\:\!/\:\!#2
	}
	{
		#1\:\!/\:\!#2
	}
}

\newcommand{\quot}[2]{
	\mathchoice
	{
		{\raisebox{.2em}{$#1$}\!\!\,\left/\!\raisebox{-.2em}{$#2$}\right.}%
	}
	{
		#1\!\:/\!\!\:\:#2
	}
	{
		#1\!\:/\!\!\:\:#2
	}
	{
		#1\!\:/\!\!\:\:#2
	}
}

\makeatletter
\DeclareRobustCommand*{\mfaktor}[3][]
{
	{ \mathpalette{\mfaktor@impl@}{{#1}{#3}{#2}} }
}
\newcommand*{\mfaktor@impl@}[2]{\mfaktor@impl#1#2}
\newcommand*{\mfaktor@impl}[4]{
	\settoheight{\faktor@zaehlerhoehe}{\ensuremath{#1#2{#3}}}%
	\settoheight{\faktor@nennerhoehe}{\ensuremath{#1#2{#4}}}%
	\raisebox{-0.5\faktor@zaehlerhoehe}{\ensuremath{#1#2{#3}}}%
	\mkern-4mu\reflectbox{\,$ / $\,}\mkern-5mu%
	\raisebox{0.5\faktor@nennerhoehe}{\ensuremath{#1#2{#4}}}%
}
\makeatother

\makeatletter
\newcommand{\bigperp}{%
	\mathop{\mathpalette\bigp@rp\relax}%
	\displaylimits
}

\newcommand{\bigp@rp}[2]{%
	\vcenter{
		\m@th\hbox{\scalebox{\ifx#1\displaystyle2.1\else1.5\fi}{$#1\perp$}}
	}%
}
\makeatother

\newcommand\blfootnote[1]{%
	\begingroup
	\renewcommand\thefootnote{}\footnote{#1}%
	\addtocounter{footnote}{-1}%
	\endgroup
}

\DeclareMathOperator{\Hom}{Hom}

\DeclareMathOperator{\spec}{Spec}

\DeclareMathOperator{\sym}{Sym}

\DeclareMathOperator{\Map}{Map}
\DeclareMathOperator{\iMap}{\underline{Map}}

\DeclarePairedDelimiter{\abs}{\lvert}{\rvert}

\makeatletter
\g@addto@macro\bfseries{\boldmath}
\makeatother

\makeatletter
\newcommand{\cbigotimes}{\DOTSB\cbigotimes@\slimits@}
\newcommand{\cbigotimes@}{\mathop{\widehat{\bigotimes}}}
\makeatother

\newtheorem{thm}{Theorem}
\numberwithin{thm}{section} 
\newtheorem{co}[thm]{Corollary}

\newtheorem{lemma}[thm]{Lemma}

\newtheorem{pr}[thm]{Proposition}

\theoremstyle{definition}
\newtheorem{defn}[thm]{Definition}
\newtheorem{claim}[thm]{Claim}
\newtheorem{rmk}[thm]{Remark}
\newtheorem{notation}[thm]{Notation}

\newtheorem{construction}[thm]{Construction}
\newtheorem{exa}[thm]{Example}
\newtheorem{disclaimer}[thm]{Disclaimer}

\newtheorem{ithm}{Theorem}
\newtheorem{ipr}[ithm]{Proposition}

\fancypagestyle{main}{\fancyhf{}
	\fancyhead[LE,RO]{\scriptsize $ SL_{\eta} $-Virtual Localisation}
	\fancyhead[RE,LO]{\scriptsize \rightmark }
	\fancyfoot[CE,CO]{\thepage}}

\title{Virtual Localisation Formula\\
	for $ SL_{\eta} $-Oriented Theories}

\author{Alessandro D'Angelo }
\date{}
 
\bibliography{SL_eta_VLoc.bib}
\begin{document}

		\maketitle{}
	\begin{abstract}
		In this paper, we extend the Virtual Localization Formula of Levine to a wide class of motivic ring spectra, obtaining in particular a localization formula for virtual fundamental classes in Witt theory $ \mr{KW} $.  Applying standard tools of $\A^1$-intersection theory to any $ SL_{\eta} $-oriented spectra $ \mr A $, we obtain an additive presentation of $ \mr A(BN) $, for $ N $ the normaliser of the torus in $ SL_2 $. Then we establish an equivariant Atiyah-Bott localization theorem for $ \mr A_N^{\mr{BM}}(X) $ and we conclude with the $N$-equivariant virtual localisation formula. Of independent interest, we also describe the ring structure of $ \mr{KW}(BN) $.
	\end{abstract}
	
	\tableofcontents
	
	\section*{Introduction}
In the presence of a torus action, many intersection theoretic and enumerative problems get simplified, looking at the fixed point locus of the problem. But to do so, one has to precisely relate the equivariant intersection theory to the intersection theory of the fixed points. In algebraic geometry, using Chow groups and the Bott residue theorem of \cite{Edidin-Graham_Atiyah-Bott}, one can get a complete description of equivariant characteristic classes of a scheme $ X $ in terms of invariants of the fixed locus $ X^T $. For algebraic K-theory analogous computations were made in \cite{Thomason_RR_and_Traces}. Often the spaces we are interested in are not smooth and we need to take into account deformation theoretic information to treat them. Schemes (or even algebraic stacks) equipped with a perfect obstruction theory, in the sense of  \cite{Behrend1997}, give rise to virtual fundamental classes and can basically be studied as if they were (almost) smooth (or quasi-smooth, using the terminology from derived geometry).\\

In \cite{Graber-Pandharipande}, Graber-Pandharipande proved a "Virtual Localization Formula" for virtual classes. This  formula relates the virtual fundamental class of a scheme $ X $, equipped with a torus action and a perfect obstruction theory, to a suitable virtual  class associated to the fixed point locus. Extensions of Atiyah-Bott and virtual localization theorems, in the context of motivic homotopy theory, were recently made in \cite{Levine_Atiyah-Bott}, \cite{VLF_Levine} and, with different techniques (applied to algebraic stacks), in \cite{aranha2022localization}. \\

Relying on the foundational works of Ananyevskiy (cf. \cite{Ananyevskiy_PhD_Thesis, Ananyevskiy_Witt_MSp_Reations, Ananyevskiy_SL_oriented}), we are now able to extend most of the results in \cite{Levine_Atiyah-Bott, VLF_Levine} to any $ SL_{\eta} $-oriented ring spectrum. As already noted in \cite{Levine_Atiyah-Bott}, the localization theorem for $\mb G_m$-actions is not interesting in the case of spectra where $ \eta $ is invertible. In those cases, the Euler classes we need to invert will already be zero (see \cref{ch2:_BGm_corollary}). The closest natural candidate for a localization theorem is then $ N $, the normaliser of the torus in $ SL_2 $. Indeed we obtained a nice presentation of the cohomology of $ \mc BN $:

\begin{ipr}[Proposition\ref{ch2:_5.3_MEC}]
	For any $ SL_{\eta} $-oriented ring spectrum, we get the following isomorphisms of graded $ \mr A^{\bullet}(S) $-modules:
	\[ A^{\bullet}(\mc{B}N)\simeq \mr A^{\bullet}(\mc{B}SL_2)\oplus \mr A^{\bullet}(S)  \]
	\[ A^{\bullet}(\mc{B}N; \gamma_N)\simeq \mr A^{\bullet-2}(\mc{B}SL_2)e(\mc T)\oplus \mr A^{\bullet}(S)   \]
	\noindent where $ \mc T $ is the tangent bundle of $\left[\quot{\pro(\sym^2(F))}{SL_2}\right]$ over $ \mc {B}SL_2 $ and $ \gamma_N $ is the generator of $ \mr{Pic}(\mc{B}N) $.
\end{ipr}

With this description at our disposal, we will be able to extend the Atiyah-Bott localization theorem to our context:
\begin{ithm}[Theorem \ref{ch4:_ABL_8.6}]
	Let $ X \in \catname{Sch}_{\bigslant{}{\sk}}^{N} $ be a scheme with an $ N $-action and let $ \mr A \in \mr{SH}(\sk) $ be an $ SL_{\eta} $-oriented spectrum. Let $ \iota: \abs{X} ^{N} \into X  $ be the closed immersion. Let $ L \in \mr{Pic}(X) $ be an $ N $-linearised line bundle. Suppose the $ N $-action  is semi-strict. Then there is a non-zero integer $ M $ such that:
	\[ \iota_*: \mr{A}^{\mr{BM}}_{\bullet, N}\left( \abs{X}^{N}; \iota^*L \right)\left[ \left( M \cdot e \right)^{-1} \right] \stackrel{}{\longrightarrow} \mr{A}^{\mr{BM}}_{\bullet, N}\left( X ; L \right)\left[ \left( M \cdot e \right)^{-1} \right]   \]
	\noindent is an isomorphism. 
\end{ithm}

Finally, for schemes $ X $ equipped with special $ N $-actions, and an $ N $-linearised obstruction theory, we have:

\begin{ithm}[Theorem \ref{ch4:_VLF_for_general_A}]
	Let $ \mr A\in \mr{SH}(\sk) $ be an $ SL_{\eta} $-oriented ring spectrum. Let $ \iota: X \into Y $ be a closed immersion in $ \catname{Sch}_{\bigslant{}{\sk}}^{N} $, with $ Y $ a smooth $ N $-scheme. Let $ \varphi_{\bullet}: \mc E_{\bullet} \rightarrow \mb L_{\bigslant{X}{\sk}} $ be an $ N $-linearised perfect obstruction theory. Then we have:
	\[ \left[X, \varphi \right]_N^{vir}=\sum_{j=1}^{s} \iota_j{}_*\left( \left[ \abs{X}_{j}^{N}, \varphi^{(j)} \right]_N^{vir} \cap e_N\left( N_{\iota_j}^{vir} \right)^{-1} \right) \in \mr{A}^{\mr{BM}}_N\left( X, E_{\bullet} \right)\left[ \left(M \cdot e\right)^{-1} \right] \]
	\noindent where $ \left[-,-\right]_N^{vir} $ denote the appropriate virtual fundamental class.
\end{ithm}

	\section*{Acknowledgements}

This paper was part of the author's PhD thesis and he would like to thank his supervisor M. Levine for the countless discussions and insights he patiently shared with him: this work could have not been accomplished without his encouragement and inspiration. The author would also like to thank C. Chowdhury for many discussions that helped to improve the paper, and the ESAGA group at the University of Essen for providing a stimulating working environment.

\blfootnote{The author was supported by the ERC through the project QUADAG.  This work is part of a project that has received funding from the European Research Council (ERC) under the European Union's Horizon 2020 research and innovation programme (grant agreement No. 832833).} \blfootnote{\hspace{-2em}\includegraphics[scale=0.08]{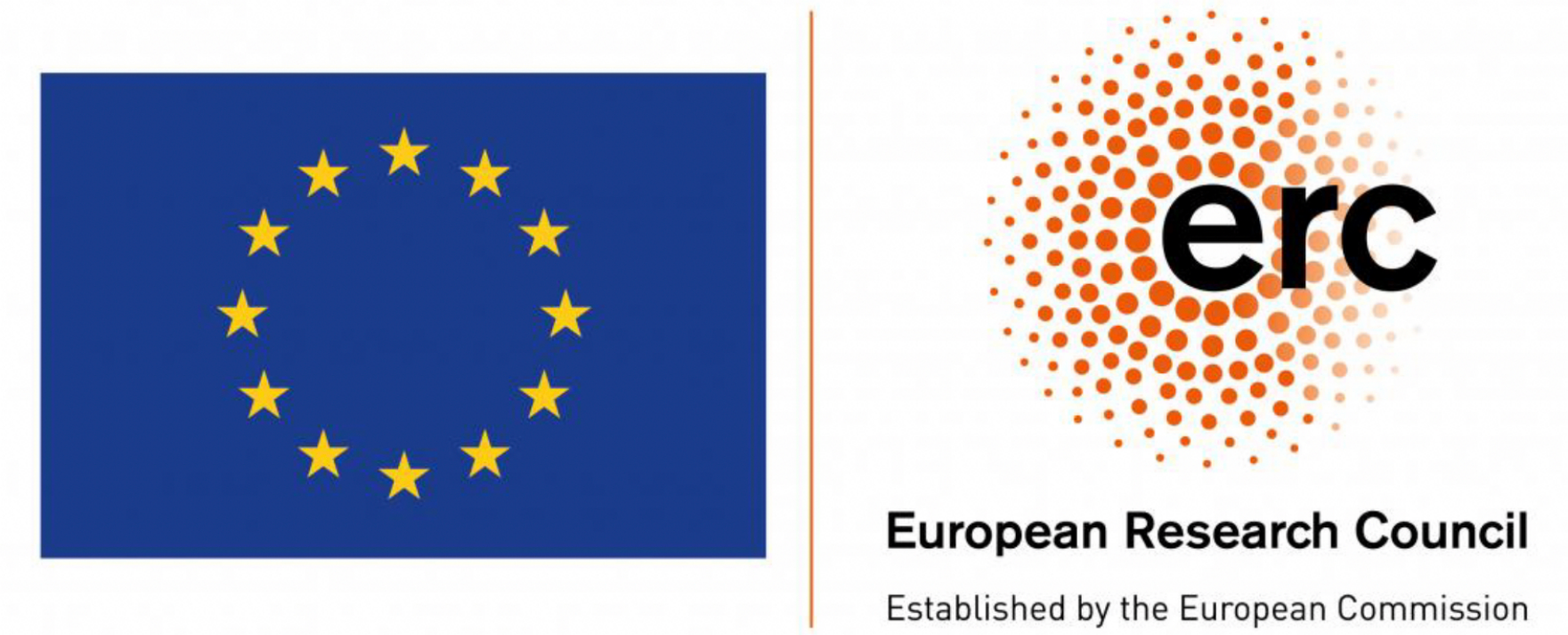}}.\\
\section*{Notations and Conventions}

\begin{enumerate}
	\item The categories $ \catname{Sch}_{\bigslant{}{B}}, \catname{Sch}_{\bigslant{}{B}}^{G} $ will always denote quasi-projective schemes over a base scheme $ B $ of finite Krull dimension, without or with a (left) $ G $-action. If we write $ \catname{Sm}_{\bigslant{}{B}},\catname{Sm}_{\bigslant{}{B}}^{G} $ then we only consider smooth quasi-projective schemes.
	\item We will denote by $ \oocatname{ASt}_{\bigslant{}{S}} $ the $ \infty $-category of algebraic stacks over some base $ S $ (that could be either a scheme or a stack, but for us will always be a scheme). Moreover given $ X \in \catname{Sch}_{\bigslant{}{B}}^{G} $, we will denote the associated quotient stack as $ \left[\bigslant{X}{G}\right] $.
	\item If not specified otherwise, whenever we are working over a field $ \sk $, we will assume it is of characteristic different from $ 2 $. If we work over a general base scheme $ B $, we will always assume $ \dfrac{1}{2} \in \mo{B}^{\times} $.
	\item Recall that a morphism of quasi-projective schemes $ f: X \rightarrow S $ in $ \catname{Sch}_{\bigslant{}{B}} $ is called \textit{lci} (that stands for \textit{local complete intersection}) if there exists a factorization of $ f $ as $ X \stackrel{i}{\into} M \stackrel{p}{\rightarrow} S $ with $ i $ a regular closed immersion and $ p $ a smooth map. In the conventions of \cite{DJK}, these are called \textit{smoothable lci} maps, but we do not need such distinction.
	
	\item If not specified otherwise, $ G $ will always denote a closed sub-group scheme inside $ GL_n $ for some $ n $.
	
	\item Given a group $ S $-scheme $ G $, we will always denote by $ \mf g^{\vee}_S $ the sheaf associated to the co-Lie algebra of $ G $. If the base scheme is clear from the context, we will only write $ \mf g^{\vee} $.
	
	\item In \cite[\S 4]{Morel-Voevodsky}, some ind-schemes were introduced to approximate quotient stacks. For a given algebraic group $ G $, those ind-schemes were denoted in \textit{loc. cit.} as \textit{geometric} classifying spaces $ B_{gm}G $. To distinguish between actual quotient classifying stacks $ [\bigslant{S}{G}]\in \oocatname{ASt}_{\bigslant{}{S}} $, over some base $ S $, and geometric classifying spaces (that are just ind-schemes), we will use a dual notation:
	\[ \mc BG:=\left[\bigslant{S}{G}\right] \in \oocatname{ASt}_{\bigslant{}{S}}\]
	\[  BG:= B_{gm}G \in \mr{Ind}(\catname{Sch}_{\bigslant{}{S}})\]
	We will come back to these notations, with more details, in Chapter 1. In general, we will try to use the calligraphic font $ 	\mc X, \mc Y, \mc BG $, etc., for algebraic stacks.

	\item Given a scheme (or an algebraic stack) $ X $, we will denote its Thomason-Trobaugh K-theory space as $ \mr K(X)=\mr K(\oocatname{Perf}(X)) $, where $ \oocatname{Perf}(X) $ is the infinity category of perfect complexes on $ X $. Given $ X \in \catname{Sch}_{\bigslant{}{B}}^{G} $ with associated quotient stack $ \mc X:= \left[ \bigslant{X}{G} \right] $, the \textit{genuine} equivariant $ \mr K $-theory of $ \mc X $ will be denoted as $ \mr K^G(X):=\mr K(\oocatname{Perf^G}(X))=\mr K(\mc X) $.
	\item Given infinity category $ \oocatname{C} $, we will denote by $ \Map_{\oocatname{C}}(-,-) $ the mapping space of $ \oocatname{C} $. If moreover $ \oocatname{C} $ is closed monoidal, the internal mapping space of $ \oocatname{C} $ we will denoted as $ \iMap_{\oocatname{C}}(-,-) $.
	
	\item When dealing with \textit{bivariant} theories, as known as Borel-Moore homology for us, we will have to take into account \textit{twists} by perfect complexes. For example given a scheme (or an algebraic stack) $ X $ and a perfect complex $ v \in \oocatname{Perf}(X) $, we can define the $ v $-twisted Borel-Moore homology $ \mb E(X,[v] ) $ as in  \cite{DJK}, where $ [v]\in K_0(X)=K_0(\oocatname{Perf}(X)) $. When $ \mc V $ is a locally free sheaf with associated class $ v:=[\mc V]\in K_0(X) $, to the automorphism in $ \mr{SH}(X) $ given by \textit{$ v $-suspension} $ \Sigma^{v} $, it will then correspond the element $ \Th{X}{V}:=\Sigma^{v}\mbbm 1_X $, where $ V:=\mb V_{X}(\mc V):=\spec(\sym^{\bullet}(\mc V^{\vee})) $ is the vector bundle given by $ \mc V $. We will always use calligraphic letters $ \mc V,\mc E $, etc., for perfect complexes and Roman letters for vector bundles.  We are also using calligraphic letters for algebraic stacks, but this should not be source of any confusion since it will be clear from the context to which kind of object we are referring to.

	\item We will stick with the conventions and notation of \cite{DJK} for the general theory of motivic bivariant theories. But from the second section and on, we will use the twisting conventions of \cite{Motivic_Euler_Char}, \cite{VLF_Levine}, \cite{Levine_Atiyah-Bott} that clashes with the one of \cite{DJK} when we  twist by line bundles: what in the former three papers  is denoted $ \mb E^{a,b}\left( X; L \right) $ is actually $ \mb E^{a+2,b+1}\left( \Th{X}{L} \right) $, that in the notation of the latter one is $ \mb E^{a+2,b+1}\left( X; -[\mc L] \right) $ where $ \mc L $ is the invertible sheaf corresponding to the line bundle $ L $. More precisely we have:
	
	\[  \mb E^{a,b}_{\mr{Levine}}\left( X; L \right):=\mb E^{a+2,b+1}\left( \Th{X}{L} \right) =: \mb E^{a+2,b+1}_{\mr{DJK}}\left( X,  -[\mc L] \right)  \]
	
	\noindent Although slightly confusing, both notation have their pros and cons. It should be clear  what we are using from the context and from the fact that we will always use the semi-colon for one and just a comma for the other (and square brackets if we want to stress that we are considering K-theoretic classes). The rule of the thumb should be: if we are working with cohomology theories and if it is a twist by a line bundle (and not by an invertible sheaf), we are using the $ \mb E_{\mr{Levine}}(-;-) $-convention, otherwise we are following \cite{DJK}.
	
\end{enumerate}

\section{Preliminaries}

\subsection{Geometric Approximations}
\label{sec.2:_Factorization_U_j}
We will now work over a base scheme $ S $ and we will denote by $ GL_n $ the group scheme $ GL_{n,S}:=GL_n\times_{\Z} S $ defined over $ S $, where $ GL_n $ is the usual group scheme of invertible $ (n\times n )$-matrices defined over $ \spec(\Z) $. Whenever we will encounter a group scheme $ G $ over $ S $, we will always assume that $ G $ is smooth.\\
We describe a version of the construction of the classifying space of $ G $, found in \cite[\S 4.2]{Morel-Voevodsky}. We can consider  $ V_m=\A^{n(n+m)}_{S}\simeq \Hom(\A^{n+m}_{S}, \A^{n}_{S}) $ equipped with the natural (left) $ GL_n $-action (hence we also have an induced natural $ G $-action).  Once we identify $ V_m $ with the scheme of $ (n,n+m) $-matrices, we can restrict to the open subset $ {E}_mG\sseq V_m $ made of those matrix with rank $ n $. On $ {E}_mG $ we have a free $ GL_n $-action (hence a free $ G $-action too) . We can  define a map $ s_m: {E}_mG \into E_{m+1}G $ sending an element $ B \in  {E}_mG $ to a matrix of the form:

\[ \left(\begin{array}{ccc|c}
	& & &0\\
	& B & &\vdots \\
	& & & 0 \\
\end{array}\right) \]

\begin{notation}\label{ch1:_Notation_quot_ind}
	Since we will closely follow  \cite{Levine_Atiyah-Bott} and \cite{VLF_Levine}, we will adapt the same notation too.  Choose as a base point   $ x_0=(I_n, 0_n, \ldots, 0_n) $ where $ I_n $ is the $ (n\times n) $ identity matrix and $ 0_n $ is the zero vector column of length $ n $.  We will denote by  $ {E}SL_n={E}GL_n $ the presheaf on $ \catname{Sm}_{S} $ given by $ \colim{m} \ {E}_mSL_n $. For any closed subgroup $ G $ of $ GL_n $ or $ SL_n $, we will denote the quotient $ {B}_mG:=\bigslant{{E}_mGL_n}{G} $ whose limit gives us the approximation $ {B}G:=\colim{m} {B}_mG $ for the quotient stack $ \mc BG:=\left[ \bigslant{S}{G} \right] $. For $ G=GL_n $, the spaces $ {B}_mGL_n\simeq \mr{Gr}_S(n, n+j) $ are represented by the Grassmannians. 
\end{notation}

\subsection{BM Motives and Operations on Algebraic Stacks}

We will freely use the six functor formalism developed for algebraic stacks in \cite{ChoDA24}. A summary is given by the following:

\begin{thm}[{\cite[Theorem 4.26]{ChoDA24}}]
	We have a functor:
	\[ \mr{SH}^*_!: Corr\left( \oocatname{ASt} \right)_{lft,all} \longrightarrow \oocatname{Pr}_{stb}^{\mr L} \]
	\noindent extending the analogous functor defined on schemes. This functor encodes the following data:
	
	\begin{enumerate}
		\item For every $ \mc X $ NL-stack, we have the tensor- hom adjunction in $ \mr{SH}(\mc X) $:
		\[ -\otimes - \dashv \iMap_{\mr{SH}(\mc X)}(-,-) \]
		\item For any map $ f: \mc X \longrightarrow \mc Y $ in $ \oocatname{ASt} $, we have a pair of adjoint functors:
		\[ f^*\dashv f^* \]
		\noindent and if $ f $ is smooth we also have:
		\[ f_{\#}\dashv f^* \dashv f_* \]
		\item For $ f: \mc X \longrightarrow \mc Y $ a lft map in $ \oocatname{ASt} $, we have another pair of adjoint functors:
		\[ f_!\dashv f^! \]
		\item Given a cartesian diagram in $ \oocatname{ASt}$: 
		\begin{center}
			\begin{tikzpicture}[baseline={(0,-1)}, scale=1.25]
				\node (a) at (0,1) {$ \mc W $};
				\node (b) at (1, 1) {$ \mc Y $};
				\node (c)  at (0,0) {$  \mc Z$};
				\node (d) at (1,0) {$ \mc X $};
				\node (e) at (0.2,0.75) {$ \ulcorner $};
				\node (f) at (0.5,0.5) {$ \Delta $};

				\path[font=\scriptsize,>= angle 90]
				
				(a) edge [->] node [above ] {$ g $} (b)
				(a) edge [->] node [left] {$ q $} (c)
				(b) edge[->] node [right] {$ p $} (d)
				(c) edge [->] node [below] {$ f $} (d);
			\end{tikzpicture}
		\end{center}
		\noindent where $p$ is lft, we have the base change equivalence:
		\[ p_!f^*\overset{Ex^*_!}{\simeq} g^*q_! \]
		\item For any lft map $f$, the projection formula holds:
		\[ f_!(-)\otimes - \simeq f_!( - \otimes f^*(-))  \]
		\item For any smooth map $f$, the smooth projection formula holds:
		\[ f\epfs(-)\otimes - \simeq f\epfs( - \otimes f^*(-))  \]
	\end{enumerate}
	These functors satisfy the usual compatibilities given by exchange transformations:
	$$ Ex^*_!,Ex^!_*, Ex^*_{\#}, Ex^*_*,Ex_{\# *}, Ex_{!*}, Ex_{!\#}, Ex^{*!} $$
	\noindent and they give rise to localisation fiber sequences like in \cite[Proposition 4.2.1]{Chowdhury24}. Moreover, for a given smooth map $f$, we have the following natural purity equivalence:
	\[ f\epfs \longrightarrow f_!\Sigma^{\mb L_f} \]
\end{thm}

\begin{defn}
	The adjoint equivalences:
	\[ \Sigma^{\mc E}:=p\epfs s_*: \mr{SH}(\mc X)\leftrightarrows \mr{SH}(\mc X): s^!p^*=:\Sigma^{-\mc E} \]
	\noindent are called \textit{Thom transformations}. We will denote by:
	\[ \Th{\mc X}{E}:=\Sigma^{\mc E}\mbbm 1_{\mc X} \in \mr{Pic}(\mr{SH}(\mc X)) \]
	\noindent the \textit{Thom space} of $ E $, with inverse $ \Sigma^{-\mc E}\mbbm 1_{\mc X} $.
\end{defn}

\begin{notation}
	Let:
	\[ J_{Bor}: \mr K \longrightarrow \mr{Pic}(\mr{SH}^{\triangleleft}) \]
	\noindent be the \textit{Borel J-homomorphism} as defined in \cite[\S 4.1]{ChoDA24}. For a given NL-stack $ \mc X $ and a given $ v \in \mr K_0(\mc X) $, we will denote the associated automorphism of $ \mr{SH}(\mc X) $ as $ \Sigma^{v} $, with inverse $ \Sigma^{-v} $. 
\end{notation}

\begin{defn}\label{ch1:_BM_Mot_and_thy_lci_ind_schemes}
	Let $ g: \mc X \rightarrow \mc B $ be a lft map of NL-stacks and $ w \in \mr K_0(\mc X) $ and let $ \mb E \in \mr{SH}(\mc B) $.
	
	\begin{enumerate}
		\item [{$\left( \begin{array}{c}
				\mr{\textit{BM}}\\
				\mr{\textit{Motive}}
			\end{array} \right)$}]  The twisted Borel-Moore motive over $ \mc B $ is defined as:
		\[ \left( \bigslant{\mc X}{\mc B} \right)^{\mr{BM}}(w):= g_! \Sigma^{w}  \mbbm 1_{\mc X} \]

		\item [{$\left( \begin{array}{c}
				\mr{\textit{BM}}\\
				\mr{\textit{Homology}}
			\end{array} \right)$}]  The twisted Borel-Moore homology over $ \mc B $ is defined as:
		\[ \mb E^{\mr{BM}}\left(\bigslant{\mc X}{\mc B}, w\right):=\iMap_{\mr{SH}(\mc B)}\left( \left( \bigslant{\mc X}{\mc B} \right)^{\mr{BM}}(w), \mb E \right) \]
		\noindent The BM-homology groups will be then defined as:
		\[  \mb E^{\mr{BM}}_{a,b}\left(\bigslant{\mc X}{\mc B}, w\right):=\pi_0\left( \Sigma^{-a,-b} \mb E^{\mr{BM}}\left(\bigslant{\mc X}{\mc B}, w\right) \right) \]
		
		\item [{$\left( \begin{array}{c}
				\mr{\textit{Generalised}}\\
				\mr{\textit{Cohomology}}
			\end{array} \right)$}]  The twisted  cohomology of $ \mc X $ is defined as:
		\[ \mb E\left( \mc X, w \right):=\iMap_{\mr{SH}(\mc B)}\left( \mbbm 1_{\mc B}, g_*\Sigma^{w}g^*\mb E  \right)\]
		\noindent and its twisted cohomology groups as:
		\[  \mb E^{a,b}\left( \mc X, w \right):=\pi_0\left( \Sigma^{a,b}  \mb E\left( \mc X, w \right) \right) \]
		
	\end{enumerate}
\end{defn}

As for the case of schemes (cf. \cite{DJK}), we can talk about smooth pushforwards, proper pullbacks and Gysin maps for maps of algebraic stacks, and associated operations in BM-homology and cohomology. 

\begin{defn}\label{ch1:_def_equiv_op}
	Let $ \pi_{\mc X}: \mc X \rightarrow \mc B $ and $ \pi_{\mc Y}: \mc Y \rightarrow \mc B $ be lft maps in  $ \oocatname{ASt}_{\bigslant{}{B}} $ and let $ f: \mc X \rightarrow \mc Y $ be a map between them. 
	
	\begin{enumerate}
		\item [{$\left( \begin{array}{c}
				\mr{\textit{SPf}}\\
			\end{array} \right)$}] If $ f $ is smooth and $ v \in \mr K_0(\mc Y) $, then we have a smooth pushforward map between BM-motives:
		\[ f_!: \left(  \bigslant{\mc X}{\mc B}\right)^{\mr{BM}}(v+ \mb L_f) \longrightarrow \left(  \bigslant{\mc Y}{\mc B}\right)^{\mr{BM}}(v)  \]
		\noindent induced by $ f_!\Sigma^{\mb L_f}f^*\stackrel{\mf p_f}{\simeq}  f_{\#}f^* \stackrel{\eta_{\#}^{*}(f)}{\longrightarrow} Id $.
		
		\item [{$\left( \begin{array}{c}
				\mr{\textit{PPb}}\\
			\end{array} \right)$}] If $ f $ is representable and proper, then we have a proper pullback map between BM-motives:
		\[ f^*:\left(  \bigslant{\mc Y}{\mc B}\right)^{\mr{BM}}(v) \longrightarrow \left(  \bigslant{\mc X}{\mc B}\right)^{\mr{BM}}(v)  \]

		\item [{$\left( \begin{array}{c}
				\mr{\textit{GPf}}\\
			\end{array} \right)$}] If $ f $ is smooth and it admits a section $ s: \mc Y \rightarrow \mc X $, then we have a natural transformation:
		\[ Id\simeq (f\circ s)_! (f\circ s)^!=f_!s_!s^!f^! \stackrel{\eta_!^!(s)}{\longrightarrow} f_! f^!\simeq f_!\Sigma^{\mb L_f} f^* \]
		This natural transformation induces then a \textit{Gysin pushforward} on BM-motives:
		\[ s_!: \left(  \bigslant{\mc Y}{\mc B}\right)^{\mr{BM}}(v) \longrightarrow \left(  \bigslant{\mc X}{\mc B}\right)^{\mr{BM}}(v+\mb L_f)   \]
		
		When $ f $ is a vector bundle, then by homotopy invariance the smooth pushforward $ f_! $ is an isomorphism on BM-motives with inverse given by the Gysin  pushforward $ s_! $ (the same argument in \cite[Lemma 2.2]{levine2017intrinsic} works verbatim).
	\end{enumerate}
	
	The operations we just defined on BM-motives will respectively induce smooth pullbacks, proper pushforwards and Gysin pullbacks on BM-homology as in the case of schemes. A more detailed discussion can be found in \cite{Motivic_Vistoli}. 
\end{defn}

\section{$SL_{\eta}$-Oriented Theories}
\subsection{$ SL $- and $ SL_{\eta}$-Orientations }\label{ch2:_SL_orientations}
Recall that  $ S $ is our general base scheme. We have a very special element in $ \mr H^{-1,-1}(S) $, the algebraic Hopf map:
\[ \eta: \A^2_S\setminus \set{0} \longrightarrow \pro^1_S \]
\noindent sending $ (x,y)\mapsto [x:y] $, giving us an element $ \eta: \Sigma_{\mb G_m}\mbbm 1_{S} \rightarrow \mbbm 1_{S} \in \mr{H}^{-1,-1}(S)  $. For any motivic ring spectrum $ \mb E $, via the unit map $ \mu: \mbbm 1_{S} \rightarrow \mb E $, we get a corresponding element $ \eta_{\mb E} \in \mb E^{-1,-1}(S) $. 

\begin{defn}
	A motivic ring spectrum $ \mb E \in \mr{SH}(S) $ is said to be $ \eta $-invertible if  multiplication by $ \eta_{\mb E} $, $ -\times \eta_{\mb E}: \mb E^{0,0}(S)\rightarrow \mb E^{-1,-1}(S) $, is an isomorphism.
\end{defn}

\begin{defn}[{\cite[Def.1]{Ananyevskiy_PhD_Thesis}}]
	An $ SL_n $-vector bundle $ (E, \theta) $ on some $ X \in \catname{Sch}_{\bigslant{}{S}} $ is the data given by a vector bundle $ E $ of rank $ n $, together with a trivialization of the determinant, $ \theta: \det(E)\stackrel{\sim}{\rightarrow} \mo{X} $. We will often denote the $ SL_n $-vector bundle by just the underlying vector bundle $ E $. If there is no need to specify the rank of the bundle, we will say that $ E $ is a $ SL $-vector bundle.
\end{defn}

After Panin-Walter, we will use the following definition:
\begin{defn}[{cf. \cite{Ananyevskiy_SL_oriented}}]\label{ch2:_SL_orientation_def}
	Let $ \catname{C} $ be a full subcategory of $ \catname{Sch}_{\bigslant{}{S}} $. Given a ring spectrum $ \mb E \in \mr{SH}(S) $, an \textit{SL-orientation with respect to} $ \catname{C} $ for $ \mb E $ is a rule which assigns to each $ SL_n $-vector bundle $ V $, over $ X \in \catname{C} $, an element:
	\[ \mr{th}(V) \in \mb A^{2n,n}(\Th{X}{V}) \]
	\noindent with the following properties:
	\begin{enumerate}
		\item For any isomorphism $ \varphi: V_1\rightarrow V_2 $ of $ SL $-vector bundles over $ X \in \catname{C} $, we have:
		\[ \varphi^*\mr{th}(V_2)=\mr{th}(V_1) \]
		\noindent where $ \varphi^* $ is the pullback map induced by $ \varphi $.
		\item For any morphism $ f: X \rightarrow Y $ in $ \catname{C} $, and  $ V $ an $ SL $-vector bundle over $ Y $, we have:
		\[ f^*\mr{th}(V) =\mr{th}(f^*V)\]
		\item For $ \mc V_1, V_2 $  $ SL $-vector bundles on some $ X\in \catname{C} $, we have:
		\[ \mr{th}(V_1\oplus V_2)=p_1^*\mr{th}(V_1)\cup p_2^*\mr{th}(V_2) \]
		\noindent where $ p_i: V_1 \oplus V_2 \rightarrow V_i $ are the projection maps.
		\item We have:
		\[ \mr{th}(\A^1_S)=\Sigma_{T}1\simeq[ \Sigma_{T} \mbbm 1_S \stackrel{\Sigma_T u_{\mb E}}{\longrightarrow } \Sigma_{T} \mb E] \in \mb E^{2,1}(\pro^1_S) \]
		\noindent where $ u_{\mb E}: \mbbm 1_S \rightarrow \mb E $ is the unit map of the ring spectrum.
	\end{enumerate}
	We refer to the elements $ \mr{th}(V) $ as \textit{Thom classes}. If a ring spectrum $ \mb E $ has a normalised $ SL $-orientation with respect to $ \catname{C}:=\catname{Sm}_{\bigslant{}{S}} $, we simply say that $ \mb E $ has an $ SL $-orientation, and we will say that $ \mb E $ is $ SL $-oriented. If $ \catname{C}:=\catname{Sch}_{\bigslant{}{S}} $, then a normalised $ SL $-orientation with respect to $ \catname{C} $ will be called an \textit{absolute $ SL $-orientation}, and $ \mb E $ will be said to be absolutely $ SL $-oriented (following the conventions in \cite{Déglise_Fasel_Borel_Char}).
	
\end{defn}

We do not really need to distinguish between orientations and absolute orientations, thanks to the following proposition proved in \cite[Corollary 1.24]{KW-Euler}:
\begin{pr}\label{ch2:_abs_SL_or_on_eta_inv}
	Let $ \mr A \in \mr{SH}(\sk) $ be an $ \eta $-invertible motivic ring spectrum. Then we have a one to one correspondence between the following data:
	\begin{enumerate}
		\item $ SL $-orientations on $ \mr A $;
		\item maps of ring spectra $ \varphi: \mr{MSL}\longrightarrow \mr A $ such that $ \varphi(\mr{th}_{\mr{MSL}}(V))=\mr{th}_{\mr A}(V)   $ for any $ V $ vector $ SL $-bundle over $ X\in \catname{Sm}_{\bigslant{}{\sk}} $;
		\item absolute $ SL $-orientations.
	\end{enumerate}
\end{pr}

Recall from \cite{Panin-Walter} that there exists a spectrum $ \mr{BO}_{S} \in SH(S) $,  whenever $ \dfrac{1}{2} \in \mo{S}^{\times} $, that represents Hermitian K-theory\footnote{There are recent works towards possible extension to more general schemes where $ 2 $ is not invertible in the ring of regular functions. It is worth mentioning for example \cite{Kumar_PhD}.}. 
\begin{defn}
	Let  $ \mr{KW}_{S}:=\mr{BO}_{S,\eta} $ be the Witt theory (absolute) spectrum defined by inverting the element $ \eta \in \mr{BO}^{-1,-1}_{S}(S) $ as done in detail in \cite[\S6 and Theorem 6.5]{Ananyevskiy_Witt_MSp_Reations}.  
\end{defn}
\begin{rmk}
	The spectrum $ \mr{BO}_S $ is $ Sp $-oriented (cf. \cite{Panin-Walter}) and hence $ SL $-oriented. This induces an $ SL $-orientation on $ \mr{KW} $, and indeed $ \mr{KW} $ will be our main example and focus point as an $ SL_{\eta} $-oriented theory.
\end{rmk}

\subsection*{Thom Isomorphism and Euler Classes}


For any $ SL $-oriented theory, we can then talk about Euler classes $ e(E,\theta) $ for $ SL $-bundles $ E $.

\begin{notation}
	We will adopt the convention of \cite{Motivic_Euler_Char} for twisted cohomology theories. That means that given $ \mb E $ an $ SL $-oriented theory and $ L \rightarrow X $ a line bundle over some $ X \in \catname{Sch}_{\bigslant{}{S}} $, we denote the $ L $-twisted $ \mb E $- cohomology by:
	\[ \mb E^{a,b}\left( X; L \right):=\mb E^{a+2,b+1}\left( \Th{X}{L} \right) \]
	Similarly, given a vector bundle $ V\rightarrow X $, we will denote the $ L $-twisted $ \mb E $-cohomology on $ \Th{X}{V} $ as:
	\[ \mb E^{a,b}\left( \Th{X}{V}; L \right):=\mb E^{a+2,b+1}\left( \Th{X}{V}\otimes \Th{X}{L} \right)\simeq \mb E^{a+2,b+1}\left( \Th{X}{V\oplus L} \right) \]
\end{notation}

Given any vector bundle $ V $ of rank $ r $ on $ X \in \catname{Sch}_{\bigslant{}{S}} $, if $ L:=\det(V) $, we can construct the associated $ SL $-vector bundle given by $ V \oplus L^{-1} $ with its canonical trivialization of the determinant $ \omega_{can}: V\oplus L^{-1}\rightarrow \mo{X} $.

\begin{defn}\label{ch2:_def_twisted_Thom_class}
	Let $ \catname{C} $ be a full subcategory of $ \catname{Sch}_{\bigslant{}{S}} $ and let $ \mb E \in \mr{SH}(S) $ be a ring spectrum with an $ SL $-orientation with respect to $ \catname{C} $. Let $ p: V \rightarrow X $ be a rank $ r $ vector bundle on $ X \in \catname{C} $ with determinant $ L:=\det(V) $.
	\begin{enumerate}
		\item 	We define the Thom class in $ L^{-1} $-twisted cohomology by:
		\[ \mr{th}(V):=\mr{th}_{V\oplus L^{-1}}\in \mb E^{2r,r}(\Th{X}{V}; L^{-1}):=\mb E^{2r+2,r+1}(\Th{X}{V\oplus L^{-1}}) \]

		\item Let $ s_{0,L}: X \oplus L^{-1} \stackrel{s_0\oplus Id}{\longrightarrow} V\oplus L^{-1} $ be the map induced by the zero section $ s_0 $ of $ V $, then we define the (twisted) Euler class as:
		\[ e(E):=s_{0,L}^*\mr{th}^{\varphi}(V) \in \mb E^{2n+2,n+1}(\Th{X}{L^{-1}})=\mb E^{2n,n}(X; L^{-1}) \]
		
	\end{enumerate}
\end{defn}

\begin{pr}[Twisted Thom Isomorphism]\label{ch2:_Twisted_Thom_SL}
	Let $ p:V\rightarrow X $ be a rank $ r $ vector bundle over a scheme $ X $, and let $ \mb E\in \mr{SH}(S) $ be a $ SL $-oriented ring spectrum together with an $ SL $-orientation map $ \varphi $. Then we have an isomorphism:
	\[ \vartheta_{V}^{\varphi}:= p^*(-) \cup \mr{th}^{\varphi}(V): \mb E^{*,*}(X; \det(V)) \longrightarrow \mb E^{*+2r,*+r}(\mr{Th}(V)) \]
\end{pr}

\begin{construction}
	\label{ch2:_tautological_symbol}
	We will now construct a \textit{symbol} element associated to a section of a line bundle, using the construction of a symbol associated to an invertible function on a scheme $ X $ as done in \cite[Definition 6.1]{Ananyevskiy_SL_oriented}. For simplicity we will restrict to the case of a scheme, but the same procedure will work for any NL-stack without changing a word. Recall from \textit{loc. cit.} that given $ u \in \Gamma(X, \mo{X}^{\times}) $, for $ X \in \catname{Sm}_{\bigslant{}{S}} $ and $ \mb E \in \mr{SH}(X) $, we have a well defined element $ \langle u\rangle \in \mb E^{0,0}(X) $ induced by the multiplication by $ u $ on $ T=\bigslant{\A^1_X}{\mb G_{m,X}} $. This element $ \langle u \rangle $ is  called the \textit{symbol} associated to $ u $. Consider now a line bundle $ p:L\rightarrow X $, and consider $ \lambda: X \rightarrow L $ a section. Denote by $ \mc Z(\lambda) $ the vanishing locus of $ \lambda $:
	\begin{center}
		\begin{tikzpicture}[baseline={(0,1)}, scale=1.5]
			\node (a) at (0,1) {$ \mc Z(\lambda) $};
			\node (b) at (1, 1) {$ X $};
			\node (c)  at (0,0) {$  X $};
			\node (d) at (1,0) {$ L $};
			\node (e) at (0.25,0.75) {$ \ulcorner $};
			\node (f) at (0.5,0.5) {$  $};

			\path[font=\scriptsize,>= angle 90]
			
			(a) edge [closed] node [above ] {$ \iota_{\lambda} $} (b)
			(a) edge [closed] node [left] {$  $} (c)
			(b) edge[closed] node [right] {$ s_0 $} (d)
			(c) edge [closed] node [below] {$ \lambda $} (d);
		\end{tikzpicture}
	\end{center}
	\noindent Let $ j_{\lambda}:U(\lambda)\into X $ be the open complement in $ X $ of $ \mc Z(\lambda) $. Then $ \lambda $ induces a non vanishing section $ j^*\lambda: U \rightarrow j^*_{\lambda}L^{\times} $ of $ j^*_{\lambda}L $. But this means we can trivialise $ j_{\lambda}^*L $, i.e. we have:
	\[ \tau_{j^*_{\lambda}}: \mb A^{1}_{U(\lambda)}\stackrel{\sim}{\longrightarrow} j^*_{\lambda}L \]
	\noindent with associated inverse:
	\[ \left( \tau_{j^*_{\lambda}} \right)^{-1}: j^*_{\lambda}L \stackrel{\sim}{\longrightarrow} \A^1_{U(\lambda)} \]
	Taking the associated Thom spaces, we get:
	
	\[ \mr{Th}(\tau_{j^*_{\lambda}}^{-1}):\Th{U(\lambda)}{j^*_{\lambda}L}\simeq\Sigma^{j^*_{\lambda}\mc  L} \mbbm 1_{U(\lambda)} \longrightarrow \Th{U(\lambda)}{\A^1}\simeq \Sigma^{\mo{}}\mbbm 1_{U(\lambda)} \]
	Twisting by $ \Sigma^{-\mo{}} $, we have:
	
	\[  \Sigma^{-\mo{}}\mr{Th}(\tau_{j^*_{\lambda}}^{-1}): \Sigma^{-\mo{}}\Sigma^{j^*_{\lambda}\mc  L} \mbbm 1_{U(\lambda)} \longrightarrow \mbbm 1_{U(\lambda)} \]
	\begin{defn}\label{ch2:_symbol_of_line_bundle_section}
		With the same notation above, let $ \mb E\in \mr{SH}(S) $ be a ring spectrum with unit $ u: \mbbm 1 \rightarrow \mb E $, then we define the $ \mb E $-\textit{symbol} associated to $ \lambda: X \rightarrow L $ to be:
		\[ \langle \lambda \rangle_{\mb E}:=u\circ \Sigma^{-\mo{}}\mr{Th}(\tau_{j^*_{\lambda}}^{-1}): \Sigma^{-\mo{}}\Th{U(\lambda)}{j^*_{\lambda}L} \rightarrow \mb E \in \mb E^{0,0}(U(\lambda); j^*_{\lambda}L) \]
		
	\end{defn}

	\begin{exa}
		Consider $ p:L\rightarrow X $ a line bundle over $ X $. Let:
		\[ t_{can}: L\rightarrow p^*L \]
		\noindent be the tautological section. Then $ U(t_{can})=L^{\times}=L\setminus 0 $ and for any $ \mb E\in \mr{SH}(S) $ we get:
		\[ \langle t_{can}\rangle \in \mb E^{0,0}(L^{\times}; L) \]
		Consider  $ X=BGL_n $ and $ L=\mo{}(1) $. Then $ L^{\times}\simeq BSL_n $ and we get:
		\[ \langle t_{can}\rangle \in \mb E^{0,0}(BSL_n; \mo{}(1)) \]
		Sometimes we will refer to $ \langle t_{can}\rangle $ as the \textit{tautological symbol} associated to $ L $.
	\end{exa}

\end{construction}

\subsection{$ SL $-Orientations for Algebraic Stacks}\label{ch2:_orientations_Stacks}
We will now present an easy way to get Thom classes and Euler classes on algebraic stacks. The methods used here can be adapted to most of the common $ G $-orientations used in the literature, but since we will need to specialise to $ SL $-oriented spectra anyway, we will only talk about those.\\

Consider $ \oocatname{U_n}\rightarrow \mc BSL_n $ the universal bundle over $ \mc BSL_n $ (the one corresponding, under Yoneda, to the identity map of $ \mc BSL_n $.  In \cite[Proposition 1.38]{KW-Euler}, we proved the following:

\begin{pr}\label{ch2:_Univ_Thom_Iso}
	Let $ \mb E \in \mr{SH}(S) $ be an $ SL $-oriented ring spectrum. Then we have a natural equivalence of mapping spectra:
	\[ \tau: \mb E(\mc BSL_n)\longrightarrow \Sigma^{2n,n}\mb E(\Th{\mc BSL_n}{\oocatname{U_n}}) \]
\end{pr}

\begin{defn}
	We define the \textit{canonical} Thom class of $ \oocatname{U_n} \rightarrow \mc BSL_n $ as the element:
	\[ \mr{th}(\oocatname{U_n}):=\tau(1_{\mc BSL_n}) \in \mb E^{2n,n}(\Th{\mc BSL_n}{\mc U_n}) \]
	\noindent where $ 1_{\mc BSL_n} \in\mb E^{0,0}(\mc BSL_n) $ is the identity element in the $ \mb E $-cohomology of $ \mc BSL_n $.
\end{defn}

Now, let $ \mc X \in \oocatname{ASt}_{\bigslant{}{S}} $ be an algebraic stack. Let $ v:  V \rightarrow \mc X $ be a vector bundle of rank $ n $ with trivialised determinant. The vector bundle $  V $ is classified by a map $ f_V $ such that:
\begin{center}
	\begin{tikzpicture}[baseline={(0,0)}, scale=1.5]
		\node (a) at (0,1) {$  V $};
		\node (b) at (2, 1) {$ \oocatname{U_n} $};
		\node (c)  at (0,0) {$  \mc X $};
		\node (d) at (2,0) {$ \mc BSL_n $};
		\node (e) at (0.2,0.8) {$ \ulcorner $};
		\node (f) at (0.5,0.5) {$  $};

		\path[font=\scriptsize,>= angle 90]
		
		(a) edge [->] node [above ] {$  $} (b)
		(a) edge [->] node [left] {$  $} (c)
		(b) edge[->] node [right] {$  $} (d)
		(c) edge [->] node [below] {$ f_{ V} $} (d);
	\end{tikzpicture}
\end{center}

\begin{defn}
	We define the Thom class of the special linear vector bundle $  V \rightarrow \mc X $ with values in a $ SL $-oriented ring spectrum $ \mb E\in \mr{SH}(S) $ as:
	\[ \mr{th}({ V}):= f_{ V}^*\mr{th}({\oocatname{U_n}})\in \mb E^{2n,n}\left(  \Th{\mc X}{ V} \right) \]
	\noindent where $ f_{ V}: \mc X \rightarrow \mc BSL_n $ is the map classifying the special linear bundle $  V $.
\end{defn}

We can hence define:

\begin{defn}
	The Thom class for $  V $ in $  L^{-1} $-twisted $ \mb E $-cohomology is the element:
	\[ \mr{th}_{ V}:=\mr{th}_{ V\oplus  L^{-1}}\in \mb E^{2n,n}(\Th{ \mc X}{ V};  L^{-1}):=\mb E^{2n+2,n+1}\left( \Th{\mc X}{ V\oplus  L^{-1}} \right) \]
\end{defn}

\begin{defn}
	In the same notation as above, we defined the $ \mb E $-valued Euler class of a vector bundle $  V \rightarrow \mc X $ of rank $ n $ as:
	\[ e( V):=s^*\mr{th}_{ V} \in \mb E^{2n, n}(X;  L^{-1}):=\mb E^{2n+2,n+1}(\Th{\mc X}{ L^{-1}}) \]
	\noindent where $ s^* $ is the pullback map induced by the zero section $ s_0: \mc X \rightarrow  V $.
\end{defn}

\section{$ SL_{\eta} $-Oriented Theories on $BN$}
\subsection{The Additive Structure of $ SL_{\eta}$-Oriented Theories on $BN$}


Closely following  \cite[\S 5]{Motivic_Euler_Char}, the aim of this section is to compute $ \mr A(\mc BN) $ for $ \mr A $ an $ SL_{\eta} $-oriented theory and $N$ the normaliser of the standard torus in $SL_2$.

Let $ T $ be the standard diagonal torus $ T\sseq SL_2 $. Note that $ T\simeq \mb G_m $, where for $ R $ a ring, we map $ t \in \mb G_m(R)= R^{\times} $ to the diagonal matrix:
\[ \left( \begin{matrix}
	t & 0 \\
	0 & t^{-1}
\end{matrix} \right) \]
We often simply write this matrix as $ t $, when there is no cause of confusion.
Notice that $ N $ is generated by $ T $ plus the element:
\[  \sigma:=\left(\begin{matrix*}[c]
	0 &1 \\
	-1 & 0
\end{matrix*}\right) \] 
As showed by the computation in \cite[\S 2]{Motivic_Euler_Char}, we have:
\[ \quot{SL_2}{N}\simeq \quot{GL_2}{N'}\simeq \pro^2 \setminus C \]
\noindent where $ N' $ is the normaliser of the  diagonal torus in $ GL_2 $ and $ C $ is the conic defined by the equation $ Q:=T_1^2-4T_0T_2 $. \\


For the reader convenience, let us recall how to view the left $ SL_2 $ action on $ \bigslant{SL_2}{N} $ under the above identification. Consider $ F=\A^2 $ with the standard left $ SL_2 $ action. We get a map:
\[ sq: \pro(F)\longrightarrow \pro\left( \sym^2(F) \right) \]
\noindent induced by the squaring map:
\[ \begin{array}{cccc}
	sq: & F & \longrightarrow & \sym^2(F) \\
	& v & \mapsto & v^2
\end{array} \]
This map is constructed in the following standard way. The map $ \varphi: F \rightarrow F \otimes F $ sending $ v \mapsto v \otimes v $ is $ SL_2 $-equivariant (where the $ SL_2 $-action on $ F\otimes F $ is the diagonal one given by $ g \cdot(v\otimes w):= g \cdot v  \otimes g \cdot w  $). Post-composing $ \varphi $ with the quotient map $ F \otimes F \rightarrow \sym^2(F) $ that sends $ a \otimes b$ to $ ab $,  we get the $ SL_2 $-equivariant map $ sq $ we wanted.

Using $ sq $ we can identify $ C\sseq \pro^2 $ with $ sq(\pro(F))\sseq \pro\left( \sym^2(F) \right) $. Since $ sq $ is $ SL_2 $-invariant, this means that $ C $ is $ SL_2 $-invariant. In particular $ Q $ is $ SL_2 $-invariant up to a scalar, so considering the multiplication morphism $  SL_2\ni g\cdot : C \rightarrow C $ with associated map on the global section denoted by $ (g\cdot)^* $, we get:

\[ \begin{array}{ccc}
	SL_2 & \longrightarrow & \mb G_m \\
	g & \mapsto & \frac{(g\cdot)^*Q}{Q}
\end{array} \]
\noindent But we know $ SL_2 $ is a simple algebraic group, so it only admits a trivial  character and hence we get that $ Q $ is actually $ SL_2 $-invariant. As an ind-scheme approximation for $ \mc BN $ we choose as in \cite[\S2, \S 5]{Motivic_Euler_Char}:
\[ BN:=\quot{SL_2}{N}\times^{SL_2} EGL_2\simeq \bigslant{ EGL_2}{N}\simeq \bigslant{ ESL_2}{N} \]



We have that $ SL_2 $ is special, so $ {E}SL_2 \longrightarrow BSL_2 $ is a Zariski locally trivial bundle and so it is $ BN \longrightarrow BSL_2 $ (cf. \cite[\S 2]{Motivic_Euler_Char}). From the description of $ \bigslant{SL_2}{N} $ we recalled above, we can realise $ BN $ as an open subscheme of the $ \pro^2 $-bundle $ \pro\left( \sym^2(F) \right)\times^{SL_2} {E}SL_2\rightarrow  BSL_2 $, with closed complement $ \pro(F)\times^{SL_2}{E}SL_2 $:
\begin{center}
	\begin{equation}
		\begin{tikzpicture}[baseline={(0,0)}, scale=2.5]
			\node (a) at (0,0.5) {$ BN $};
			\node (b) at (2, 0.5) {$ \pro(\sym^2(F))\times^{SL_2} {E}SL_2 $};
			\node (c)  at (4,0.5) {$ \pro(F)\times^{SL_2} {E}SL_2 $};
			\node (d) at (2,0) {$ BSL_2 $};

			\path[font=\scriptsize,>= angle 90]
			
			(a) edge [open] node [above ] {$  $} (b)
			(c) edge [closed'] node [above] {$  $} (b)
			(b) edge[->] node [right] {$  $} (d)
			(a) edge [->] node [below] {$  $} (d)
			(c) edge [->] node [below] {$ $} (d);
		\end{tikzpicture}
	\end{equation}
\end{center}

Remember that we have a short exact sequence:
\[ 1 \rightarrow T \longrightarrow N \longrightarrow \set{\pm 1} \rightarrow 1 \]
\noindent where $ T $ is our torus in $ SL_2 $. The normaliser as we already said is generated by $ T $ and the element $ \sigma $, so sending $ T $ to $ 1 $ and $ \sigma $ to  $ -1 $ give us a representation $ \rho^{-}: N \longrightarrow \mb G_m $. \\

The description we just gave for the approximating ind-scheme $ BN $ is also useful to give a different presentation of the quotient stack $ \mc BN $. Indeed, since $ \bigslant{SL_2}{N}\simeq \pro(\sym^2(F))\setminus \pro(F) $ we have:

\begin{equation}\label{chap.2_fig_loc_seq_BN}
	\begin{tikzpicture}[baseline={(0,0)}, scale=2.5]
		\node (a) at (0,0.5) {$ \mc BN=\left[ \bigslant{\left(\bigslant{SL_2}{N}\right) }{SL_2}\right] $};
		\node (b) at (2, 0.5) {$ \left[ \bigslant{\pro(\sym^2(F))}{SL_2} \right] $};
		\node (c)  at (4,0.5) {$\left[ \bigslant{\pro(F)}{SL_2} \right] $};
		\node (d) at (2,0) {$ \mc BSL_2 $};

		\path[font=\scriptsize,>= angle 90]
		
		(a) edge [open] node [above ] {$ j $} (b)
		(c) edge [closed'] node [above] {$ \iota $} (b)
		(b) edge[->] node [right] {$ p_2 $} (d)
		(a) edge [->] node [below] {$ p $} (d)
		(c) edge [->] node [below] {$ \bar{p}$} (d);
	\end{tikzpicture}
\end{equation}

Since $ \quot{SL_2}{N}\simeq \pro^2\setminus C $, a line bundle over $ \mc BN $ can be described as a line bundle over $ \pro^2\setminus C $ together with a $ SL_2 $-linearisation. The line bundle $ \mo{\pro^2}(1) $ is equipped with a natural $ SL_2 $-linearisation (induced by the natural action of $ SL_2 $ on $ \pro^2 $), hence its restriction to $ \pro^2\setminus C $ gives us a well defined element $ \gamma_N \in \mr{Pic}^{SL_2}(\pro^2\setminus C)\simeq \mr{Pic}(\mc BN) $.
\begin{lemma}
	Let $ \sk $ be a field. The Picard group $ \mr{Pic}(\mc BN) $ of $ \mc BN \in \oocatname{ASt}_{\bigslant{}{\sk}} $ is generated by the line bundle $ \gamma_N \in \mr{Pic}^{SL_2}(\pro^2\setminus C)\simeq \mr{Pic}(\mc BN)  $ coming from $ \mo{\pro^2}(1) $, wioth its natural $ SL_2 $-linearisation. Moreover $ \mr{Pic}(\mc BN)\simeq \bigslant{\Z}{2\Z} $.
\end{lemma}
\begin{proof}
	Since we work over a field, by \cite[Proposition 2.10]{Brion_lin_pic}, we can identify $ \mr{Pic}(\mc BN)=\mr{Pic}^N(\sk)\simeq \oocatname{X}(N)=\Hom(N, \mb G_m) $, where $ \oocatname{X}(N) $ denotes the character group. Let $ \chi \in \oocatname{X}(N)$ and let $ t $ be the parameter of the diagonal torus $ T\sseq N $ and $ \sigma:=\left(\begin{matrix*}[c]
		0 &1 \\
		-1 & 0
	\end{matrix*}\right)\in N $. Then we must have that $ \chi(t)=t^n $ for some integer $ n $ while $ \chi(\sigma)\in \set{\pm 1} $. But we also have:
	\[ \chi\left( \sigma\left(\begin{matrix*}[c]
		t &0 \\
		0 & t^{-1}
	\end{matrix*}\right) \right)=\chi\left( \left(\begin{matrix*}[c]
		t &0 \\
		0 & t^{-1}
	\end{matrix*}\right)\sigma \right) \] 
	This implies:
	\[ \chi(\sigma)\cdot t^{n}=t^{-n}\cdot \chi(\sigma) \]
	and hence $ n=0 $. Therefore there can be only two characters for the group $ N $: the trivial one sending $ \sigma $ to the identity and the non trivial one sending $ \sigma $ to $ -1 $. This proves that $ \mr{Pic}(\mc BN)\simeq \bigslant{\Z}{2\Z} $. To see that $ \gamma_N $ is the generator, it is enough to notice that the pullback of $ \mo{\pro^2}(1) $ on $ \pro^2\setminus C $, with its $ SL_2 $-trivialization, cannot be trivial since there are no non-vanishing global sections of $ \mo{}(1) $ on $ \pro^2\setminus C $. Thus $ \gamma_N $ must be the generator. Moreover we do have a non-vanishing global section for $ \mo{}(2) $, given by $ Q=T_1^2-4T_0T_2 $, and hence $ \gamma_N $ is indeed a $ 2 $-torsion element as expected.
\end{proof}

%
%
%
\begin{rmk}
	The representation $ \rho^{-}: N \rightarrow \mb G_m $ sending $ \sigma $ to $ -1 $ corresponds exactly to the line bundle $ \gamma_N$ generating $ Pic(BN) $. 
\end{rmk}

\begin{notation}
	To make the notation more compact, we will write:
	\[ \widetilde{-}:=  \left[ \bigslant{-}{G} \right]  \]
	\noindent to denote the quotient stack of $ G $-equivariant objects, whenever the group $ G $ is clear from the context.
\end{notation}

\begin{lemma}[{\cite[Lemma 5.1]{Motivic_Euler_Char}}]\label{ch2:_Lemma_5.1_MEC}
	Let $ F $ be the tautological rank two representation of $ SL_2 $ and let $ \mr A \in \mr{SH}(S) $ be an $ \eta $-invertible spectrum. Then the structure map $ \pi_{\widetilde{\pro(F)}}: \widetilde{\pro(F)}\rightarrow S $ induces an isomorphism:
	\[ \pi_{\widetilde{\pro(F)}}^*:\mr{A}^{\bullet}(S)  \stackrel{\sim}{\longrightarrow} \mr{A}_{SL_2}^{\bullet}\left( \pro(F)\right) \]
	\end{lemma}

	\begin{proof}
		We will denote by $ \pi_{-}: - \rightarrow S $ the structure maps of our schemes and stacks. Give to $ F=\A^2 $  the standard  left $ SL_2 $-action, and equip $ \A^2\setminus \set{0} $ with the induced (left) action.  Equip $ SL_2 $ with the left $ SL_2 $-action coming from matrix multiplication, then we have a $ SL_2 $-equivariant map:
		\[ \begin{array}{cccc}
			r: &  SL_2 & \longrightarrow & \A^2 \setminus \set{0} \\
			& \left(\begin{matrix}
				a & b \\
				c & d
			\end{matrix}\right) & \mapsto & (a,b)
		\end{array} \]
		\noindent that realises $ SL_2 $ as a ($ SL_2 $-equivariant) $ \A^1 $-bundle over $ \A^2\setminus \set{0} $, with zero section map $ s_0: \A^2\setminus \set{0}\rightarrow SL_2 $. This means that once we pass to the quotient stacks we get an $ \A^1 $-bundle map:
		\[ \tilde{r}: \left[ \quot{SL_2}{SL_2} \right]\simeq S \longrightarrow \left[ \quot{\A^2\setminus \set{0}}{SL_2} \right]=:\widetilde{\A^2\setminus\set{0}} \]
		\noindent with zero section $ \tilde{s_0}: \widetilde{\A^2\setminus\set{0}} \rightarrow S $. By homotopy invariance, we get an equivalence of mapping spectra:
		\[ \tilde{s_0}^*: \mr A(S)=\iMap_{\mr{SH}(S)}(\mbbm 1_S, \mr A) \stackrel{\sim}{\longrightarrow} \mr A\left( \widetilde{\A^2\setminus \set{0}} \right)=\iMap_{\mr{SH}(S)}\left(\mbbm 1_S, \pi_{\widetilde{\A^2\setminus \set{0}}}{}_*\pi_{\widetilde{\A^2\setminus \set{0}}}^*\mr A\right) \]
		On the other hand, the algebraic Hopf map:
		\[ \eta: \A^2\setminus \set{0} \rightarrow \pro(F) \]
		\noindent induces a map on the quotients stacks:
		\[ \tilde{\eta}: \widetilde{\A^2\setminus \set{0}} \rightarrow \widetilde{\pro(F)}:=\left[ \quot{\pro(F)}{SL_2} \right] \]
		\noindent  Hence we get a well defined map:
		\[ \tilde{\eta}^*: \mr A\left(\widetilde{\pro(F)}\right)\longrightarrow \mr A\left( \widetilde{\A^2\setminus \set{0}} \right) \]
		By \cite[Remark 3.38]{Motivic_Vistoli}, for any NL-stack $ \mc X \in \oocatname{ASt}_{\bigslant{}{S}}^{NL} $ and any NL-atlas $ X \rightarrow \mc X $, we can compute the cohomology of $ \mc X $ using its \v Cech nerve $ X_{\mc X}^{n}:=\check{C}_n\left( \bigslant{X}{\mc X} \right) $:
		\begin{align*}
			\mr A(\mc X)&=\Map_{\mr{SH}(S)}(\mbbm 1_S, \pi_{\mc X}{}_*\pi_{\mc X}^*\mr A)\simeq\\ 
			& \simeq \Map_{\mr{SH}(S)}(\mbbm 1_S, \lim_{n \in \Delta}\pi_{X_{\mc X}^n}{}_*\pi_{X_{\mc X}^n}^*\mr A)\simeq \\
			& \simeq \lim_{n \in \Delta} \Map_{\mr{SH}(S)}(\mbbm 1_S, \pi_{X_{\mc X}^n}{}_*\pi_{X_{\mc X}^n}^*\mr A)\simeq \\
			& \simeq \lim_{n \in \Delta} \mr A(X_{\mc X}^n)
		\end{align*}
		\noindent For $ \widetilde{\pro(F)}=\left[ \bigslant{\pro(F)}{SL_2} \right] $ a NL-atlas is given by $ \pro(F)\rightarrow \widetilde{\pro(F)} $ with \v Cech nerve given levelwise by $ \check{C}_n\left( \bigslant{\pro(F)}{\widetilde{\pro(F)} } \right)=\pro(F)\times SL_n^{n} $, hence:
		\[ \mr A\left( \widetilde{\pro(F)} \right)\simeq \lim_{n \in \Delta} \mr A\left( \pro(F)\times SL_2^{n} \right) \]
		
		\noindent Similarly, for $ \widetilde{\A^2\setminus \set{0}}  $ we get:
		\[ \mr A\left( \widetilde{\A^2\setminus \set{0}}\right)\simeq \lim_{n \in \Delta} \mr A\left( (\A^2\setminus\set{0})\times SL_2^{n} \right) \]
		
		\noindent Therefore the map $ \tilde{\eta}^* $ restricted levelwise on the \v Cech nerves gives us a map:
		\[ \tilde{\eta}_n^*: \mr A\left( \pro(F)\times SL_2^n \right) \longrightarrow \mr A\left( (\A^2\setminus\set{0})\times SL_2^n \right) \]
		\noindent that is just the pullback map along $ \eta_n=\eta\times Id: (\A^2\setminus\set{0})\times SL_2^n \rightarrow \pro(F)\times SL_2^n $. But since $ \eta $ is invertible in $ \mr A $, all the maps $ \tilde\eta_n^* $ are invertible and hence: 
		\[ \tilde{\eta}^*: \mr A\left(\widetilde{\pro(F)}\right)\longrightarrow \mr A\left( \widetilde{\A^2\setminus \set{0}} \right)  \]
		\noindent is an equivalence. \\

		Now consider the following commutative diagram of quotient stacks:
		
		\begin{center}
			\begin{tikzpicture}[baseline={(0,0)}, scale=1.5]
				\node (a) at (0,1) {$ \widetilde{\A^2\setminus \set{0}} $};
				\node (b) at (0, 0) {$ S $};
				\node (c)  at (1.5,1) {$ \widetilde{\pro(F)} $};

				\path[font=\scriptsize,>= angle 90]
				
				(a) edge [->] node [left ] {$  \tilde{s}_0 $} (b)
				(a) edge [->] node [above] {$ \tilde{\eta} $} (c)
				(c) edge[->] node [below right] {$ \pi_{\widetilde{\pro(F)}} $} (b);
			\end{tikzpicture}
		\end{center}
		This means that we have:
		\[ (\tilde{s}_0)^*\simeq \tilde{\eta}^*\pi_{\widetilde{\pro(F)}}^* \]
		And hence, since both $ (\tilde{s}_0)^* $ and $ \tilde{\eta}^* $ are equivalences, the pullback map:
		
		\[ (\pi_{\widetilde{\pro(F)}})^*\simeq (\tilde{\eta}^*)^{-1}(\tilde{s}_0)^*: \mr A(S)\stackrel{}{\longrightarrow} \mr A\left( \widetilde{\pro(F)} \right)  \]
		\noindent is an equivalence too.

		\end{proof}

		To prove the full statement of the additive presentation of $ \mr A^{\bullet}(BN) $, for an $ SL_{\eta} $-oriented $ \mr A $, we will need various intermediate steps.

		\begin{lemma}[{\cite{Marc_Notes_Gen}}]
			\label{ch2:_SL_eta_Lemma_1}
			Given $ A\in \mr{SH}(S) $ and $ \eta $-invertible spectrum, then the pullback of the structure map $ \pi_{\pro^2_S}: \pro^2_S \rightarrow S $ induces an isomorphism: 
			\[ \mr A^{\bullet}(S)  \longrightarrow \mr A^{\bullet}(\pro_{S}^2) \]
		\end{lemma}
		\begin{proof}
			We can cover $ \pro^2_{S} $ by two opens $ U,V $ where $ U=\pro^2_{S}\setminus \set{p} $ with $ p=[1:0:0] $, and $ V=\pro^2_{S}\setminus\set{x_0=0}\simeq \A^2_{S} $. Then we have a Mayer-Vietoris diagram:
			\begin{center}
				\begin{tikzpicture}[baseline={(0,-1)}, scale=1.5]
					\node (a) at (0,1) {$ U\cap V\simeq\A^2\setminus \set{0} $};
					\node (b) at (2, 1 ) {$ U $};
					\node (c)  at (0,0) {$  V\simeq \A^2_{S} $};
					\node (d) at (2,0) {$ \pro^2_{S} $};
					\node (e) at (0.2,0.8) {$  $};
					\node (f) at (0.5,0.5) {$  $};

					\path[font=\scriptsize,>= angle 90]
					
					(a) edge [->] node [above ] {$  $} (b)
					(a) edge [->] node [left] {$  $} (c)
					(b) edge[->] node [right] {$ j_U $} (d)
					(c) edge [->] node [below] {$ j_V $} (d);
				\end{tikzpicture}
			\end{center}
			We can identify $ U $ with the bundle $ \mo{\pro^1}(1)\rightarrow \pro^1_{S} $ that sends $ [x_0:x_1:x_2]\mapsto [0:x_1:x_2] $, hence $ U $ is $ \A^1 $-equivalent to $ \pro^1_{S} $. Then up to $ \A^1 $-invariance, the Mayer-Vietoris diagram becomes:
			\begin{center}
				\begin{tikzpicture}[baseline={(0,-1)}, scale=1.5]
					\node (a) at (0,1) {$ \A^2\setminus \set{0} $};
					\node (b) at (2, 1) {$ \pro^1_{S} $};
					\node (c)  at (0,0) {$  \A^2_S $};
					\node (d) at (2,0) {$ \pro^2_{S} $};
					\node (e) at (0.2,0.8) {$  $};
					\node (f) at (0.5,0.5) {$  $};

					\path[font=\scriptsize,>= angle 90]
					
					(a) edge [->] node [above ] {$ \eta $} (b)
					(a) edge [->] node [left] {$  $} (c)
					(b) edge[->] node [right] {$  $} (d)
					(c) edge [->] node [below] {$ j_V $} (d);
				\end{tikzpicture}
			\end{center}
			Hence $ \pro^2_{S} $ is $ \A^1 $-equivalent to $ \pro^1_{S}\amalg_{\A^2\setminus \set{0}} \A^2_S $. But if we invert the (unstable) Hopf map, we can replace $ \A^2_S\setminus \set{0} $ with $ \pro^1_{S} $ and hence $ \pro^2_{S} $ becomes equivalent to the $ \A^2_S $ in $ \mr{SH}(S)[\eta^{-1}] $. Consider the following commutative diagram:
			
			\begin{center}
				\begin{tikzpicture}[baseline={(0,0)}, scale=1]
					\node (a) at (0,1) {$ \A^2 $};
					\node (b) at (2, 1) {$ \pro^2_{S} $};
					
					\node (d) at (2,0) {$ S $};
					\node (e) at (0.2,0.8) {$  $};
					\node (f) at (0.5,0.5) {$  $};

					\path[font=\scriptsize,>= angle 90]
					
					(a) edge [->] node [above ] {$ j_V $} (b)
					(a) edge [->] node [below left] {$ \pi_{\A^2} $} (d)
					(b) edge[->] node [right] {$ \pi_{\pro^2} $} (d);
				\end{tikzpicture}
			\end{center}
			
			Then we have that $ \pi_{\A^2}^*\simeq j_V^*\pi_{\pro^2_S}^* $, but $ \pi_{\A^2}^* $ is an equivalence by homotopy invariance and $ j_V^* $ becomes an equivalence once we invert $ \eta $, therefore $ \pi_{\pro^2_S}^* $ is an equivalence too in $ \mr{SH}(S)[\eta^{-1}] $. This means that we have an isomorphism:
			\[ \pi_{\pro^2_S}^*: \mr A^{\bullet}(S) \stackrel{\sim}{\longrightarrow} \mr A^{\bullet}(\pro^2_S) \]
			\noindent as claimed.
		\end{proof}

		\begin{lemma}[{\cite{Marc_Notes_Gen}}]
			\label{ch2:_SL_eta_Lemma_2}
			Let $ \mr A\in \mr{SH}(S) $ be an $SL [\eta^{-1}] $-oriented spectrum. Let $ T \rightarrow \pro^2_{S} $ be the tangent bundle of $ \pro^2_S $ and let $ e(T)\in \mr A^{2}(\pro^2_S; \mo{\pro^2}(3))\simeq  \mr A^{2}(\pro^2_S; \mo{\pro^2}(1))   $ be its Euler class. The cup product with $ e(T) $ induces an isomorphism:
			\[ \mr A^{n-2}(S)\simeq \mr A^{n}(\pro^2_S;  \mo{\pro^2}(1))\simeq   \mr A^{n}(\pro^2_S; \mo{\pro^2}(3))\] 
		\end{lemma}
		\begin{proof}
			Consider the localization sequence associated to a point $ \set{p} \into \pro^2 $:
			\[ \mr A^{\bullet-2}(p) \stackrel{\iota_*}{\longrightarrow }\mr A^{\bullet}(\pro^2_S; \mo{}(1))\stackrel{j^*}{\longrightarrow} \mr A^{\bullet}(\pro^2_S\setminus \set{p}; \mo{}(1)) \]
			\noindent where we identified the twist by $ \mo{}(3) $ with the one by $ \mo{}(1) $ since $ \mr A $ is $ SL $-oriented. As in \cref{ch2:_SL_eta_Lemma_1},homotopy invariance gives us an  isomorphism:
			\[  \mr A^{\bullet}(\pro^2_S\setminus \set{p}; \mo{}(1))\simeq   \mr A^{\bullet}(\pro^1_S; \mo{}(1)) \]
			By definition we have $  \mr A^{\bullet}(\pro^1_S; \mo{}(1)) =\mr A^{\bullet+1}(\Th{\pro^1_S}{\mo{}(1)})  $. But by homotopy invariance, $ \mo{}(1) $ can be identified with $ \pro^1_S $ and after inverting $ \eta $, we can also identify $ \A^2_S \setminus \set{0}\simeq \mo{}(1)\setminus \set{0}\sseq \mo{}(1) $ with $ \pro^1_S $. Hence $ \Th{\pro^1}{\mo{}(1)} $ becomes just a (pointed) point, so $ A^{\bullet}(\Th{\pro^1}{\mo{}(1)}) =0$. Then from the localization sequence we get an isomorphism:
			\[ \iota_*:\mr A^{\bullet-2}(S)\stackrel{\sim}{\longrightarrow} \mr A^{\bullet}(\pro^2_{S}; \mo{}(1)) \]
			The pushforward map $ \pi_{\pro^2}{}_*: \mr A^{\bullet}(\pro^2_S; \mo{}(1) ) \longrightarrow \mr A^{\bullet}(S) $, with $ \pi_{\pro^2_S}: \pro^2_S \rightarrow S $ the structure map,  is an isomorphism too since $ \pi_{\pro^2_S}{}_*\iota_*=Id $. 
			
			Now we make the following claim: if $ T $ is the tangent bundle of $ \pro^2_S $, then $ \pi_{\pro^2_S}e(T) $ is a unit in $ \mr A^0(S) $. To see this, we use the Leray spectral sequences for $ \mr A^{\bullet}(S) $ and $ \mr A^{\bullet}(\pro^2_S;\mo{}(1)) $ (cf. \cite[\S 4]{Asok-Deglise-Nagel}) to reduce to computations along the fibers $ \mr A^{\bullet}(\kappa(x)) $ and $ \mr A^{\bullet}(\pro^2_{\kappa(x)}) $. By the motivic Gau\ss-Bonnet theorem (cf. \cite{Levine_Raksit_Gauss_Bonnet} or \cite[4.6.1]{DJK}), we have that $ \chi\left( \bigslant{\pro^2_{\kappa(x)}}{\kappa(x)} \right)=\pi_{\kappa(x)}{}_*(e^{\mb S}(T_{\pro^2_{\kappa(x)}})) $, where $ e^{\mb S}(\cdot) $ denotes the Euler characteristic relative to the sphere spectrum. We know that the Euler-characteristic of $ \pro^2_{\kappa(x)} $ is equal to $ \langle2\rangle+\langle-1\rangle \in \mr{GW(\kappa(x))} $ by the computation in \cite[Ex.1.8]{Hoyois_Quadratic_Lefschetz}, but this means that  $ \pi_{\kappa(x)}{}_*(e(T_{\pro^2_{\kappa(x)}}))$ is the image of $ \langle1\rangle \in \mr W(\kappa(x))  $  under the unit map $ \mr W(\kappa(x))\rightarrow \mr A^0(\kappa(x)) $. So we have that $ \pi_{\pro^2_{\kappa(x)}}{}_*(\iota_{\kappa(x)}{}_* 1_{\pro^2})=\pi_{\pro^2}{}_*(e(T_{\pro^2})) $ defines an isomorphism and is indeed a unit. Therefore:
			\[ \left(\cdot\cup e(T_{\pro^2_{\kappa(x)}}) \right) \circ \pi_{\pro^2_{\kappa(x)}}^*: \mr A^{\bullet-2}(\kappa(x)) \longrightarrow \mr A^{\bullet}(\pro^2_{\kappa(x)}; \mo{}(1)) \]
			\noindent defines an isomorphism too, as we wanted.
			
		\end{proof}
		
		We will now go on a brief excursus computing $ SL_{\eta} $-oriented theories of $ \pro^n_S $, just for the sake of completeness. The following lemma (as well as the last two lemmas before) is already known (cf. \cite[Theorem 2]{Ananyevskiy_Pushforwards_Eta_Inverted} for a more general statement):
		
		\begin{lemma}\label{ch2:_computation_pro_spaces}
			Given $ \mr A\in \mr{SH}(S) $ with $ \mr A $ an $ SL_{\eta} $-oriented spectrum. Then we have:
			\[ \mr A^{\bullet}(\pro^{2k+1})\simeq \mr A^{\bullet}(S) \oplus \mr A^{\bullet-2k-1}(S) \]
			\[ \mr A^{\bullet}(\pro^{2k+1}; \mo{}(1))\simeq 0 \]
			\[ \mr A^{\bullet}(\pro^{2n})\simeq \mr A^{\bullet}(S) \]
			\[   \mr A^{\bullet}(\pro^{2n}; \mo{}(1))\simeq \mr A^{\bullet-2n}(S)\]
		\end{lemma}
		\begin{proof}
			We will proceed by induction. Let us start considering $ n=1 $, then take the localization sequence associated to $ p \into \pro^1 $ with open complement given by $ \A^1 $:
			\[ \ldots \rightarrow \mr A^{\bullet-1}(S)\stackrel{(\iota_1)_*}{\rightarrow} \mr A^{\bullet}(\pro^1_S) \stackrel{j_1^*}{\rightarrow} \mr A^{\bullet}(S) \stackrel{\partial}{\rightarrow} \mr A^{\bullet}(S) \rightarrow \ldots\]
			\noindent Then the projection map $ \pi_{\pro^1}: \pro^1_S \rightarrow S $ induces the pullback map $ \pi^* $ that exhibits $ j_1^* $ as a split surjection, telling us that $ \partial=0 $ and that:
			\[ \mr A^{\bullet}(\pro^{1})\simeq \mr A^{\bullet}(S) \oplus \mr A^{\bullet-1}(S) \]
			For the twisted theory of $ \pro^1 $, we already saw in the proof of \cref{ch2:_SL_eta_Lemma_2} that:
			\[ \mr A^{\bullet}(\pro^{1}; \mo{}(1))\simeq 0 \]
			Moreover \cref{ch2:_SL_eta_Lemma_1} and \cref{ch2:_SL_eta_Lemma_2} already took care of the case $ n=2 $. Let us assume we know the result for any $ m<2n+1 $, then we just need to show it also holds for $ 2n+1 $ and $ 2n+2 $. Let us start with $ \pro^{2n+1}_S $, and again consider $ p \into \pro^{2n+1} $. Similarly to what we did in \cref{ch2:_SL_eta_Lemma_1}, we can identify $ \pro^{2n+1}\setminus\set{p} $ with $ \mo{\pro^{2n}}(1) $ and hence, up to $ \A^1 $-equivalence, with $ \pro^{2n} $. Now the localization sequence associated to $ p \into \pro^{2n+1} $ reads as:
			\[ \ldots \rightarrow\mr A^{\bullet-2n-1}(S)\stackrel{(\iota_{2n+1})_*}{\rightarrow} \mr A^{\bullet}(\pro^{2n+1}_S) \stackrel{j_{2n+1}^*}{\rightarrow} \mr A^{\bullet}(\pro^{2n}) \stackrel{\partial_{2n+1}}{\rightarrow} \mr A^{\bullet-2n}(S) \rightarrow \ldots \]
			By induction hypothesis, we have $ \mr A^{\bullet}(\pro^{2n}) =\mr A^{\bullet}(S)  $, and once again the projection map $ \pi_{\pro^{2n+1}}: \pro^{2n+1}_S \rightarrow S $ induces a map $ \pi_{\pro^{2n+1}}^* $ that makes $ j_{2n+1}^* $ into a split surjection, giving us:
			\[ \mr A^{\bullet}(\pro^{2n+1})\simeq \mr A^{\bullet}(S) \oplus \mr A^{\bullet-2n-1}(S) \]
			Considering again the localization sequence associated to $ p \into \pro^{2n+1} $, but this time with the twist by $ \mo{}(1) $, we get:
			\[ \ldots \mr A^{\bullet-2n-1}(S)\stackrel{(\iota(1)_{2n+1})_*}{\rightarrow} \mr A^{\bullet}(\pro^{2n+1}_S;\mo{}(1)) \stackrel{j(1)_{2n+1}^*}{\rightarrow} \mr A^{\bullet}(\pro^{2n};\mo{}(1)) \stackrel{\partial_{2n+1}(1)}{\rightarrow} \mr A^{\bullet-2n}(S) \rightarrow \ldots  \]
			\noindent By induction hypothesis $ \partial_{2n+1}(1) $ must be an isomorphism and thus $ \mr A^{\bullet}(\pro^{2n+1}_S;\mo{}(1)) \simeq 0 $. The case $ 2n+2 $ is completely similar, and we will leave it to the reader.
		\end{proof}
		
		\begin{co}\label{ch2:_BGm_corollary}
			For any $ \eta $-inverted spectrum $ \mr A \in \mr{SH}(S) $, we have:
			\[ \mr A^{\bullet}\left( {B}\mb G_{m,S} \right)\simeq 	\mr A^{\bullet}(S) \]
			\[  \mr A^{\bullet}\left( {B}\mb G_{m,S}; \mo{}(1) \right)\simeq 	0 \]
		\end{co}
		\begin{proof}
			Recall from \cite{Morel-Voevodsky}, that the model for $ B\mb G_m $ is given by $ \pro^{\infty} $. We have that:
			\[ {B}\mb G_{m}\simeq\pro^{\infty}=\colim{m} \ \pro^m \]
			For the untwisted case, we can take the colimit over the even dimensional projective spaces (since this is a cofinal system). Then the structure map $ \pi_{\mc B\mb G_m}: \mc B\mb G_m\rightarrow S $, under the identification in \cite[Proposition 3.33]{Motivic_Vistoli}, induces a map:
			\[ \pi_{\mc B\mb G_m}^*: \mr A^{\bullet}(S) \rightarrow \mr A^{\bullet}(\mc B\mb G_m)\simeq \mr A^{\bullet}(B\mb G_m)\simeq \lim_k \mr A^{\bullet}(\pro^{2k}) \]
			If we consider $ \mr A^{\bullet}(S) $ as a limit spectrum over a constant pro-system, then the map $ \pi_{\mc B\mb G_m}^* $ levelwise becomes $ \pi_{\pro^{2k}}^*: \mr A(S)\rightarrow \mr A(\pro^{2k}) $. By  \cref{ch2:_computation_pro_spaces}, $ \pi_{\pro^{2k}}^* $ is an equivalence for every $ k $ and therefore $  \pi_{\mc B\mb G_m}^* $ is an equivalence too, proving the first claim of our corollary. For the second claim, we just write $ B\mb G_m $ as a colimit of odd dimensional projective spaces and again by \cref{ch2:_computation_pro_spaces} we get:
			\[ \mr A^{\bullet}(B\mb G_m; \mo{}(1))\simeq \lim_{k} \mr A^{\bullet}(\pro^{2k+1}; \mo{}(1))=0 \]
		\end{proof}
		The previous corollary can also be deduced from the following stronger result proved in \cite[Theorem 6.1.3]{Haution_Odd_VB}:
		\begin{thm}[Haution]\label{ch2:_Haution_BGm}
			The map $ B\mb G_m \rightarrow S $ is induces an isomorphism in $ \mr{SH}(S)[\eta^{-1}] $. 
		\end{thm}

		Let $ \theta: T \longrightarrow \pro(\sym^2(F)) $ be the tangent bundle of $ 	\pro(\sym^2(F)) $. 
		We have an induced map:
		\[ \oocatname{T}:=\left[\bigslant{T}{SL_2}\right] \rightarrow \widetilde{\pro(\sym^2(F))}:=\left[ \bigslant{\pro(\sym^2(F))}{SL_2} \right] \]
		\noindent where $ \oocatname{T} $ is the relative tangent bundle of $ \widetilde{\pro(\sym^2(F))} $ over $ \mc BSL_2 $.

		\begin{lemma}[{\cite[Lemma 5.2]{Motivic_Euler_Char}}]\label{ch2:_Lemma_5.2_MEC}
			Let $ \mr A \in \mr{SH}(S) $ be an $ SL_{\eta} $-oriented motivic spectrum. Then:
			\begin{enumerate}
				\item [$ (a) $]  The map $ \tilde p: \widetilde{\pro(\sym^2(F))} \longrightarrow \mc BSL_2 $ induces an isomorphism:
				\[ \tilde p^*: \mr{A}^{\bullet}\left( \mc BSL_2 \right) \stackrel{\sim}{\longrightarrow} \mr{A}^{\bullet}_{SL_2}\left( \pro(\sym^2(F)) \right) \]
				\item [$ (b) $]  $ \mr{A}^{\bullet}_{SL_2}\left( \pro(\sym^2(F)); \mo{}(1) \right) $ is a free $ \mr{A}^{\bullet}(\mc BSL_2) $-module generated by $ e(\oocatname{T}) $.
			\end{enumerate}
		\end{lemma}
		
		\begin{proof}
			\begin{enumerate}
				\item[$ (a) $] Let us consider the cohomological motivic Leray spectral sequence as in \cite[\S 4]{Asok-Deglise-Nagel}, applied to $ p: \pro(\sym^2(F))\times Z \rightarrow Z $ for any scheme $ Z \in \catname{Sch}_{\bigslant{}{S}} $. Then we have:
				\[ E_1^{p,q}=\bigoplus_{x \in Z^{(p)}} \mr{A}^{p+q}\left( p^{-1}(x); N_{x} \right) \allora \mr{A}^{p+q}\left( \pro(\sym^2(F))\times Z\right) \]
				\noindent where $ p^*N_{x} $ is the normal bundle of the inclusion $ \iota_x: p^{-1}(x) \into {\pro_m(\sym^2(F))} $, that is, the pullback of the normal bundle $ N_{x} $ of $ x: \spec(\kappa(x)) \into Z $. Since $ \mr{A} $ is $ SL $-oriented, we actually have a spectral sequence of the form:
				\[ E_1^{p,q}=\bigoplus_{x \in Z^{(p)}} \mr{A}^{q}\left( p^{-1}(x); \det(p^*N_{x}) \right) \allora \mr{A}^{p+q}\left( \pro(\sym^2(F))\times Z\right) \]
				\noindent The bundle $ p: {\pro(\sym^2(F))}\times Z \rightarrow Z $ has fibers $ p^{-1}(x)\simeq \pro^2_{\kappa(x)} $ for all $ x \in Z $.

				By \cref{ch2:_SL_eta_Lemma_2}, we have:
				\[ p_{x}^*: \mr{A}^{q}(x; \det(N_{x})) \stackrel{\sim}{\longrightarrow} \mr{A}^{q}\left( \pro^2_{\kappa(x)}; \det(p^*N_{m,x}) \right)  \]
				But we also have the Gersten spectral sequence for $ Z $:
				\[ E_1^{p,q}=\bigoplus_{x \in Z^{(p)}} \mr{A}^{q}(\kappa(x); \det(N_{x})) \allora \mr{A}^{p+q}(Z) \]
				\noindent and by the functoriality of the Leray spectral sequences\footnote{For $ Z $, the Leray spectral sequence is just the Gersten one.} \cite[Proposition 4.2.10]{Asok-Deglise-Nagel} we get that  the pullback map:
				\[ p^*: \mr{A}^{\bullet}(Z)  \stackrel{\sim}{\longrightarrow} \mr{A}^{\bullet}({\pro(\sym^2(F))\times Z}) \]
				
				\noindent is an isomorphism, since the induced map $ p_{x}^* $ on the $ E_1 $-terms is so. This implies that:
				\begin{equation}\label{ch2:_eq_A(P(Sym2F))_Computation}
					\begin{array}{ccc}
						\mr A( \pro(\sym^2(F))\times Z)&\hspace{-0.5em}=&\hspace{-.5em}\Map_{\mr{SH}(S)}(\mbbm 1_S, (\pi_{ \pro(\sym^2(F))\times Z}){}_*(\pi_{\ \pro(\sym^2(F))\times Z})^*\mr A)\simeq \\
						&\hspace{-.5em}\simeq &\hspace{-13.75em}\Map_{\mr{SH}(S)}(\mbbm 1_Z, \pi_{Z}^*\mr A)=\\
						&\hspace{-0.5em}=& \hspace{-20.75em}\mr A(Z)
					\end{array}
				\end{equation}
				
				The pullback map:
				\[ \tilde p^*:\mr A(\mc BSL_2)  \longrightarrow\mr A(\widetilde{\pro(\sym^2(F))}) \]
				\noindent by \cite[Remark 3.38]{Motivic_Vistoli} gives us a map of limits of mapping spectra over the \v Cech nerves:
				\[ \tilde p^*:  \lim_{n \in \Delta} \mr A\left( SL_2^{n} \right)\longrightarrow\lim_{n \in \Delta} \mr A\left( \pro(\sym^2(F))\times SL_2^{n} \right)  \]
				\noindent But by the computation \eqref{ch2:_eq_A(P(Sym2F))_Computation} we made above, this is a levelwise equivalence and thus we get an isomorphism:
				
				\[ \tilde p^*: \mr{A}^{\bullet}\left( \mc BSL_2 \right) \stackrel{\sim}{\longrightarrow} \mr{A}^{\bullet}_{SL_2}\left( {\pro(\sym^2(F))} \right)  \]
				
				\noindent where we identified $ \mr A( \widetilde{\pro(\sym^2(F)) }) $ with the $ SL_2 $-equivariant cohomology by \cite[Proposition 3.33]{Motivic_Vistoli}.

				\item[$ (b) $]  By a similar argument, for any scheme $ Z \in \catname{Sch}_{\bigslant{}{S}} $, we can use again the Leray spectral sequence converging to $ \mr{A}^{p+q}\left( {\pro(\sym^2(F))}\times Z; \mo{}(1) \right) $ and the isomorphisms:
				\[ \mr{A} ^{\bullet}\left( \pro_{\kappa(x)}^2; \mo{}(1) \right)\simeq \mr A^{\bullet-2}(\kappa(x))\cdot e(T_{\pro^2_{\kappa(x)}}) \]
				\noindent from \cref{ch2:_SL_eta_Lemma_2}. So this time the spectral sequence is telling us that $ e(\oocatname{T}_Z) \in \mr{A}^2\left( {\pro(\sym^2(F))} \times Z \right) $, where $ \oocatname{T}_Z $ is the tangent bundle of $ {\pro(\sym^2(F))}\times Z $ over $ Z $, is a generator for $ \mr{A}^{\bullet}\left( {\pro(\sym^2(F))}\times Z; \mo{}(1) \right) $ as a free $ \mr{A}^{\bullet}\left( Z \right) $-module. Therefore we get an equivalence of mapping spectra:
				\begin{equation}\label{ch2:_eq_P(Sym2F)_Twisted}
					\mr A(Z) \stackrel{\sim}{\longrightarrow} \Sigma^{3}\mr A(\Th{\pro(\sym^2(F))}{\mo{}(1)}\times Z)
				\end{equation}
				\noindent where $ \Sigma^{3}\mr A(\Th{\pro(\sym^2(F))}{\mo{}(1)}\times Z)= \Sigma^2\mr A(\pro(\sym^2(F))\times Z; \mo{}(1)) $ by definition. \\
				By \cite[Remark 3.38]{Motivic_Vistoli} , using the \v Cech nerve of $ \widetilde{\pro(\sym^2(F))}  $, we can rewrite:
				\[ \mr A(\widetilde{\pro(\sym^2(F))})\simeq \lim_{n \in \Delta} \mr A(\pro(\sym^2(F))\times SL_2^{n}) \]
				\noindent The isomorphism of \eqref{ch2:_eq_P(Sym2F)_Twisted} tells us that:
				\[ \mr A(\widetilde{\pro(\sym^2(F))})\simeq \lim_{n \in \Delta} \Sigma^{-2} \mr A(SL_2^{n}) \cdot e(\oocatname{T}_{SL_2^{n}}) \]
				\noindent But for each $ n $, the bundles $ \oocatname{T}_{SL_2^{n}} $ is the pullback of the relative tangent bundle $ \oocatname{T} $ of $ \widetilde{\pro(\sym^2(F))}) $ over $ \mc BSL_2 $, hence $ e(\oocatname{T}) $ gets pulled back to $ e(\oocatname{T}_{SL_2^{n}}) $ along the maps of the \v Cech nerve.  Since we also have that  $ \mr A(\mc BSL_2)\simeq \lim_{n \in \Delta} \mr A(SL_2^{n})  $, we get:
				\[ \mr A(\widetilde{\pro(\sym^2(F))})\simeq \Sigma^{-2}\left(\lim_{n \in \Delta}  \mr A(SL_2^{n}) \right)\cdot e(\oocatname{T})\simeq \Sigma^{-2}\mr A(\mc BSL_2)\cdot e(\oocatname{T})  \]
				After identifying $ \mr A(\widetilde{\pro(\sym^2(F))}) $ with the equivariant cohomology, we get our claim.
			\end{enumerate}

		\end{proof}
		
		Now we want to use the localization sequence relative to:
		
		\begin{equation}\label{ch2:_Loc_Seq_BN}
			\begin{tikzpicture}[baseline={(0,0.5)}, scale=2.5]
				\node (a) at (0,0.5) {$ \mc BN $};
				\node (b) at (2, 0.5) {$  \widetilde{\pro(\sym^2(F))} $};
				\node (c)  at (4,0.5) {$ \widetilde{\pro(F)}  $};
				\node (d) at (2,0) {$ \mc BSL_2 $};

				\path[font=\scriptsize,>= angle 90]
				
				(a) edge [open] node [above ] {$ j $} (b)
				(c) edge [closed'] node [above] {$ \iota $} (b)
				(b) edge[->] node [right] {$ p_2 $} (d)
				(a) edge [->] node [below] {$ p $} (d)
				(c) edge [->] node [below] {$ \bar{p}$} (d);
			\end{tikzpicture}
		\end{equation}

		\noindent to compute  $ \mr A^{\bullet}(BN) $. But first we will need the following:
		\begin{defn}
			Let $ \mb E \in \mr{SH}(S) $, we say that $ \mb E $ is even if $ \mb E^{a,b}(\sk)=0 $ for all  fields $ \sk $ and for all integers $ a,b $ such that $ a-b $ is odd.
		\end{defn}
		\begin{rmk}\label{ch2:_MSL_eta_is_even}
			If $ \mr A \in \mr{SH}(S) $ is $ \eta $-invertible, since we follow the same convention $ \mr A^{n}(\cdot):=\mr A^{n,0}(\cdot)\simeq \mr A^{n+i,i}(\cdot) $ as in \cite{Ananyevskiy_PhD_Thesis}, then $ \mr A $ is even if and only if $ \mr A^{n}(\sk)=0 $ for all fields $ \sk $ and all odd $ n $. In particular by \cite[8.11]{Bachmann_Hopkins}, $ \mr{MSL}_{\eta} $ is even. 
		\end{rmk}
		
		
		\begin{exa}
			Witt theory $ \mr{KW} $ and Witt cohomology $ H\mc W $ are also examples of even spectra.
		\end{exa}
		\begin{rmk}\label{ch2:_rmk_even_BLS2_theories}
			Let us assume that our base scheme $ S=\sk $ is a field. It follows from \cite[Theorem 10]{Ananyevskiy_PhD_Thesis} that, for $ \mr A $ an $ SL_{\eta} $-oriented even spectrum, we have $ \mr A^{n}({B}SL_k)=0 $ for any integer $ k $ and any odd integer $ n $. Hence using \cref{ch2:_Lemma_5.1_MEC},\cref{ch2:_Lemma_5.2_MEC}, for those kind of spectra we have that:
			\[ \mr A^{n}_{SL_2}\left( \pro(F) \right)\simeq0 \]
			\[ \mr A^{n}_{SL_2}\left( \pro(\sym^2(F)) \right)\simeq0 \]
			\[ \mr A^{n}_{SL_2}\left( \pro(\sym^2(F)); \mo{}(1) \right)\simeq0 \]
			\noindent whenever $ n $ is odd (and always working over a field $ \sk $).
		\end{rmk}

		\begin{rmk}
			In the following proposition we are considering quotient stacks over a base field $ \sk $. We should write $ \mc BG_\sk $ to stress this and to avoid confusion with the classifying stack $ \mc BG_S $ over some more general base $ S $, but since it will be clear from the context we will just denote them as $ \mc BG $ (and the same applies to the ind-schemes $ BG $).
		\end{rmk}
		
		\begin{pr}[{\cite[Proposition 5.3]{Motivic_Euler_Char}}]\label{ch2:_5.3_MEC_field}
			Let $ \sk $ be a field and consider $ \mc BN$, $ \widetilde{\pro(\sym^2(F))}$, $ \widetilde{\pro(F)} \in \oocatname{ASt}_{\bigslant{}{\sk}}^{NL} $. Let $ \mc T $ be the tangent bundle of $\widetilde{\pro(\sym^2(F))}$ over $ \mc {B}SL_2 $ and let $ \gamma_N $ be the generator of $ \mr{Pic}(\mc {B}N) $. For any $ SL_{\eta} $-oriented ring spectrum $ \mr A\in \mr{SH}(\sk) $ and any integers $ n, k $, we get split exact sequences:
			\[ 0 \rightarrow A^{2n}_{SL_2}(\pro(\sym^2(F)); \mo{}(k)) \stackrel{j^*}{\longrightarrow} \mr A^{2n}(\mc {B}N; \mo{}(k)) \stackrel{\partial}{\longrightarrow} \mr A^{2n}(\sk) \rightarrow 0 \]
			\noindent yielding the following isomorphisms of graded $ \mr A^{\bullet}(\sk) $-modules:
			\[ A^{\bullet}(\mc {B}N)\simeq \mr A^{\bullet}(\mc {B}SL_2)\oplus \mr A^{\bullet}(\sk)  \]
			\[ A^{\bullet}(\mc {B}N; \gamma_N)\simeq \mr A^{\bullet-2}(\mc {B}SL_2)e(\oocatname{T})\oplus \mr A^{\bullet}(\sk)   \]
		\end{pr}

		\begin{proof}
			Let us begin noticing via the inclusion $ \iota $, the line bundle $ \mo{\pro(\sym^2(F))}(k) \in \mr{Pic}^{SL_2}(\pro(\sym^2(F))) $ pulls back to $ \iota^*\mo{\pro(\sym^2(F))}(k)\simeq\mo{\pro(F)}(2k) $ over $ \pro(F) $ and that the normal bundle of $ \pro(F) \into \pro(\sym^2(F)) $ is isomorphic to $ \mo{\pro(F)}(2) $, so that all the twists will get trivialised over $ \pro(F) $.\\
			Using the localization sequence associated to \cref{ch2:_Loc_Seq_BN}, we get:
			\begin{equation}\label{ch2:_additive_LES}
				\ldots \rightarrow \mr A^{2n-1}_{SL_2}(\pro(F)) \stackrel{\iota_k{}_*}{\rightarrow} A^{2n}_{SL_2}(\pro(\sym^2(F)); \mo{}(k)) \stackrel{j^*_k}{\longrightarrow} \mr A^{2n}(\mc {B}N; \mo{}(k)) \stackrel{\partial_k^{\mr A}}{\longrightarrow} \mr A^{2n}_{SL_2}(\pro(F)) \rightarrow \ldots 
			\end{equation}
			
			Now let us for a moment work with the universal $ SL_{\eta} $-oriented ring spectrum (cf. \cite[Theorem 4.7]{Ananyevskiy_Witt_MSp_Reations}), that is, $ \mr{MSL}_{\eta} $; we will drop the super-script on the boundary maps for the moment and we will use it again at the end for the general case. By \cref{ch2:_MSL_eta_is_even}, we know $ \mr{MSL}_{\eta} $ is even and hence, by \cref{ch2:_rmk_even_BLS2_theories}, our long exact sequence above gives rise to short exact sequences involving only even terms:
			\[ 0\rightarrow \mr{MSL}_{\eta,SL_2}^{2n}(\pro(\sym^2(F)); \mo{}(k)) \stackrel{j^*_k}{\longrightarrow} \mr{MSL}_{\eta} ^{2n}(\mc {B}N; \mo{}(k)) \stackrel{\partial_k}{\longrightarrow} \mr{MSL}_{\eta, SL_2} ^{2n}(\pro(F)) \rightarrow 0 \]
			We have that $ \mr{Pic}^{SL_2}(\pro(\sym^2(F)))=\Z $ is generated by $ \mo{}(1) $, while $ j^*\mo{}(1)\simeq \gamma_N $ is the generator of $ \mr{Pic}(\mc {B}N) $, so, by lemma \cref{ch2:_Lemma_5.2_MEC} we can further reduce to the following two kinds of exact sequences:
			\begin{equation}\label{ch2:_additive_MSL_eq}
				0\rightarrow \mr{MSL}_{\eta}^{2n}(\mc {B}SL_2) \stackrel{j^*}{\longrightarrow} \mr{MSL}_{\eta}^{2n}(\mc {B}N) \stackrel{\partial_0}{\longrightarrow} \mr{MSL}_{\eta}^{2n}(\sk) \rightarrow 0
			\end{equation}
			\begin{equation}\label{ch2:_additive_twisted_MSL_eq}
				0\rightarrow \mr{MSL}_{\eta}^{2n-2}(\mc {B}SL_2)\cdot e(\oocatname{T}) \stackrel{j^*}{\longrightarrow}  \mr{MSL}_{\eta}^{2n}(\mc {B}N;\gamma_N) \stackrel{\partial_1}{\longrightarrow} \mr{MSL}_{\eta}^{2n}(\sk) \rightarrow 0
			\end{equation}
			
			
			Recall we wrote $ \mc BN $ as $ \left[ \bigslant{(\pro^2 \setminus C)}{SL_2} \right]  $, where $ C $ was the conic given by the zero locus of the section $ Q=T_1^2-4T_0T_2 $ of $ \mo{\pro^2}(2) $. Applying our \cref{ch2:_tautological_symbol} to the section $ \lambda_{Q}: \left[ \bigslant{\pro^2}{SL_2} \right] \rightarrow \mo{\left[ \bigslant{\pro^2}{SL_2} \right]}(2) $ induced by $ Q $, we get a well defined element $  \langle {\lambda_Q} \rangle \in \mr{MSL}_{\eta}^{0}(\mc BN) $. To get an $ \mr A^{\bullet}(\sk) $-module splitting for \cref{ch2:_additive_MSL_eq} and \cref{ch2:_additive_twisted_MSL_eq}, it is enough to send $ 1 \in  \mr{MSL}_{\eta}^{0}(\sk) $ to some elements $ q_0^{\mr{MSL_{\eta}}} , q_1^{\mr{MSL_{\eta}}}  $ such that their boundary will be $ 1 $ again. 
			For  \cref{ch2:_additive_twisted_MSL_eq}, we are just content to choose any $ q_1^{\mr{MSL_{\eta}}}  $ such that $ \partial_1(q_0^{\mr{MSL_{\eta}}} )=1 $ (this will exists by surjectivity of $ \partial_1 $). For \cref{ch2:_additive_MSL_eq}, we can make a clever choice: set $ q_0^{\mr{MSL_{\eta}}} :=\langle {\lambda_Q}\rangle $ constructed before. First let us show that $ \partial_0(q_0^{\mr{MSL_{\eta}}} ) $ is invertible. By a Mayer-Vietoris argument we can reduce to the case of a trivial vector bundle, hence to the case where $ \lambda_Q $ is just the standard coordinate section $ t $ of $ \A^1 $. But in this case $ \partial(\langle t \rangle)=\eta $ by \cite[Lemma 6.4]{Ananyevskiy_SL_oriented} applied levelwise  to the \v Cech nerves. Via the isomorphism $ \mr{MSL}_{\eta}^{-1,-1}(\sk)\simeq \mr{MSL}_{\eta}^{0}(\sk) $ given by multiplication with $ \eta^{-1} $, the boundary $ \partial(\langle t \rangle) $  is really sent to $ 1 $ and hence $ \partial_0(q_0^{\mr{MSL_{\eta}}} )  $ is invertible. Let us prove that this boundary is not only invertible, but it is indeed 1. 
			Let $ U $ be some dense open of $ \pro^2 $ where $ \mo{}(2) $ gets trivialised, let $ V $ be the open dense subset of $ \pro(F) $ corresponding to $ U\cap C $ under the identification $ C\simeq\pro(F) $. Let $ j_V: V\into \pro(F)\rightarrow \widetilde{\pro(F)} $ be the map defined by the composition of the open immersion $ V \into \pro(F) $ together with the standard quotient map given by the atlas of $ \widetilde{\pro(F)} $. Since $ \mr{MSL}_{\eta}^{\bullet}(\sk)\simeq \mr{MSL}_{\eta}^{\bullet}(\widetilde{\pro(F)}) $ via the pullback along the structure map $ \pi_{\widetilde{\pro(F)}} $, we can identify $ j_V^* $ with $ \pi_{V}^* $, where $ \pi_V: V \rightarrow \spec(\sk) $ is the structure map of $ V $. Since on $ V $ we can trivialise $ \mo{}(2) $, we get that $ j_V^*(\partial_0(q_0^{\mr{MSL_{\eta}}} ) -1)=0 $. But $ \pi_V^* $ (and hence $ j_V^* $) is injective, so this implies that $ \partial_0(q_0^{\mr{MSL_{\eta}}} ) =1 $ as claimed.
			\newline

			For a general $ SL_{\eta} $-oriented ring spectrum $ \mr A $, by \cite[Theorem 4.7, Lemma 4.9]{Ananyevskiy_Witt_MSp_Reations}, we have a map $ \varphi^{SL}_{\eta}: \mr{MSL}_{\eta} \rightarrow \mr  A $ of $ SL $-oriented ring spectra. Consider again the long exact sequence \cref{ch2:_additive_LES} and consider elements $ q_0^{\mr A}:=\varphi^{SL}_{\eta}(q_0^{\mr{MSL_{\eta}}} ), q_1^{\mr A}:=\varphi^{SL}_{\eta}(q_1^{\mr{MSL_{\eta}}} ) $. Set $ i \in \set{0,1} $. Since $ \partial_i^{\mr A}(q_i^{\mr A})=\partial_i^{\mr A}(\varphi^{SL}_{\eta}(q_i^{\mr{MSL}_{\eta}} ))=\varphi^{Sl}_{\eta}(\partial_i^{\mr{MSL}_{\eta}}(q_i^{\mr{MSL_{\eta}}} )) $ and since $ \varphi^{SL}_{\eta} $ is a map of ring spectra, sending $ 1_{\mr{MSL}_{\eta}}\in \mr{MSL}_{\eta}^0(\sk) $ to $ 1_{\mr A} \in \mr A^0(\sk) $, we have that the boundaries of $ q_i^{\mr A} $ are both 1. Hence $ \partial_i^{\mr A} $ are split surjective maps of $ \mr A^{\bullet}(\sk) $-modules as in the $ \mr{MSL}_{\eta} $ case and we are done.

		\end{proof}
		Let us generalise the previous proposition to the case of a general smooth $ \sk $-scheme $ S $. But before doing that we will need to prove the following lemma:
		
		\begin{lemma}\label{ch2:_Repr_Stack-Coh}
			Let $ \mc X\in \oocatname{ASt}_{\bigslant{}{B}}^{NL} $ be a smooth NL-stack and let $ v \in \mr{K}_0(\mc X) $.  Then there exists a spectrum $ \mb E_{\mc X,v} \in \mr{SH}(B) $ such that there is a natural equivalence of functors:
			\[ \theta:  \mb E_{\mc X,v}(-)\simeq\mb E(\mc X \times_B -, v): \catname{Sm}_{\bigslant{}{B}}\longrightarrow \mr{SH}(B)  \]
		\end{lemma}
		
		\begin{proof}
			Let $ Y\in \catname{Sm}_{\bigslant{}{B}} $. Denote by $ \pi_{\mc X}: \mc X \rightarrow B $ and by $ \pi_{Y}: Y \rightarrow B $ the structure maps of $ \mc X $ and $ Y $. We claim that:
			\[  \mb E_{\mc X,v}:=\iMap_{\mr{SH}(B)}(\mbbm 1_B, \pi_{\mc X}{}_*\Sigma^{v}\pi_{\mc X}^*\mb E) \in \mr{SH}(B)  \]
			\noindent  is the spectrum we are looking for, where $ \iMap $ denotes the internal mapping space in $ \mr{SH}(S) $. Consider the following Tor-independent cartesian square:
			\begin{center}
				\begin{tikzpicture}[baseline={(0,1)}, scale=1.5]
					\node (a) at (0,1) {$ \mc X \times Y$};
					\node (b) at (1, 1) {$ Y $};
					\node (c)  at (0,0) {$  \mc X $};
					\node (d) at (1,0) {$ S $};
					\node (e) at (0.2,0.8) {$ \ulcorner  $};
					\node (f) at (0.5,0.5) {$  $};

					\path[font=\scriptsize,>= angle 90]
					
					(a) edge [->] node [above ] {$ p_2 $} (b)
					(a) edge [->] node [left] {$ p_1 $} (c)
					(b) edge[->] node [right] {$ \pi_{Y} $} (d)
					(c) edge [->] node [below] {$ \pi_{\mc X} $} (d);
				\end{tikzpicture}
			\end{center}
			Notice that since the square is Tor-independent we have $ p_2^*\Omega_X\simeq \mb L_{\bigslant{\mc X\times Y}{\mc X}} $, moreover $ p_1 $ is clearly a smooth representable map. By \cite[Theorem 4.26]{ChoDA24} we have $ Ex_!^*:\pi_{\mc X}^*\pi_Y\epf \stackrel{}{\simeq} p_1\epf p_2^* $ and hence we get:
			\begin{align}\label{ch2:_eq_Smooth_BC}
				\begin{split}
					\pi_{\mc X}\epfs \Sigma^{-v}\pi_{\mc X}^* \pi_{Y}\epfs \mbbm 1_Y & \simeq \pi_{\mc X}\epfs\Sigma^{-v} \pi_{\mc X}^* \pi_{Y}\epf\Sigma^{\Omega_Y} \mbbm 1_Y\simeq\\
					&\stackrel{Ex_!^*}{\simeq} \pi_{\mc X}\epfs p_1\Sigma^{-p_1^*v}\epf p_2^*\Sigma^{\Omega_Y} \mbbm 1_Y\simeq \\
					& \simeq \pi_{\mc X}\epfs p_1\epf\Sigma^{-p_1^*v}\Sigma^{p_2^*\Omega_Y} p_2^* \mbbm 1_Y\simeq \\
					& \simeq \pi_{\mc X}\epfs p_1\epf \Sigma^{-p_1^*v}\Sigma^{\mb L_{\bigslant{\mc X\times Y}{\mc X}}} p_2^* \mbbm 1_Y\simeq \\
					& \simeq \pi_{\mc X}\epfs p_1\epfs \Sigma^{-p_1^*v}p_2^* \mbbm 1_Y\simeq\\
					& 	\simeq \pi_{\mc X \times Y}\epfs \Sigma^{-p_1^*v}\mbbm 1_{\mc X\times Y}
				\end{split}
			\end{align}
			\noindent where we used purity (for representable maps) twice, once for $ \pi_Y $ and once for $ p_1 $. But this means that:
			
			\begin{align*}
				\begin{split}
					\mb E_{\mc X,v}(Y)&= \Map_{\mr{SH}(B)}\left( \mbbm 1_B, \pi_{Y}{}_*\pi_{Y}^*\mb E_{\mc X,v} \right)\simeq \\
					&\simeq \Map_{\mr{SH}(B)}\left( \pi_{Y}\epfs \mbbm 1_Y, \mb E_{\mc X,v} \right)=\\
					& \simeq \Map_{\mr{SH}(B)}\left( \pi_{Y}\epfs \mbbm 1_Y, \iMap_{\mr{SH}(B)}(\mbbm 1_B, \pi_{\mc X}{}_*\Sigma^{v}\pi_{\mc X}^*\mb E) \right)\simeq\\
					& \simeq \Map_{\mr{SH}(B)}\left( \pi_{Y}\epfs \mbbm 1_Y, \pi_{\mc X}{}_*\Sigma^{v}\pi_{\mc X}^*\mb E \right)\simeq\\
					& \simeq \Map_{\mr{SH}(B)}\left(  \pi_{\mc X}\epfs\Sigma^{-v}\pi_{\mc X}^*\pi_{Y}\epfs \mbbm 1_Y,\mb E \right)\simeq\\
					& \stackrel{\text{\eqref{ch2:_eq_Smooth_BC}}}{\simeq}  \Map_{\mr{SH}(B)}\left(  \pi_{\mc X \times Y}\epfs \Sigma^{-p_1^*v} \mbbm 1_{\mc X\times Y} ,\mb E \right)\simeq\\
					&\simeq \mb E(\mc X \times Y, v)
				\end{split}
			\end{align*}
			This identification $ \mb E_{\mc X,v}(Y)\simeq\mb E(\mc X \times_B Y, v) $ is moreover functorial in the $ Y $, where $ f: Y_1\rightarrow Y_2 $ is sent to the pullback map:
			\[ f^*:   \mb E_{\mc X,v}(Y_2)\simeq\mb E(\mc X \times_B Y_2, v) \longrightarrow \mb E_{\mc X,v}(Y_1)\simeq\mb E(\mc X \times_B Y_1, v)  \]
			Hence we have a natural equivalence of functors:
			\[  \theta:  \mb E_{\mc X,v}(-)\simeq\mb E(\mc X \times_B -, v): \catname{Sm}_{\bigslant{}{B}}\longrightarrow \mr{SH}(B)  \]
		\end{proof}
		
		\begin{pr}\label{ch2:_5.3_MEC}
			Let $ S \in \catname{Sm}_{\bigslant{}{\sk}} $. For any $ SL_{\eta} $-oriented ring spectrum $ \mr A \in \mr{SH}(S) $ and any integers $ n, k $, we get split exact sequences:
			\[ 0 \rightarrow A^{2n}_{SL_2}(\pro(\sym^2(F)); \mo{}(k)) \stackrel{j^*}{\longrightarrow} \mr A^{2n}(\mc {B}N; \mo{}(k)) \stackrel{\partial}{\longrightarrow} \mr A^{2n}(S) \rightarrow 0 \]
			\noindent yielding the following isomorphisms of graded $ \mr A^{\bullet}(S) $-modules:
			\[ A^{\bullet}(\mc {B}N_S)\simeq \mr A^{\bullet}(\mc {B}SL_{2,S})\oplus \mr A^{\bullet}(S)  \]
			\[ A^{\bullet}(\mc {B}N_S; \gamma_N)\simeq \mr A^{\bullet-2}(\mc {B}SL_{2,S})e(\oocatname{T})\oplus \mr A^{\bullet}(S)   \]
			\noindent where $ \oocatname{T} $ is the tangent bundle of $\widetilde{\pro(\sym^2(F))}$ over $ \mc {B}SL_{2,S} $ and $ \gamma_N $ is the generator of $ \mr{Pic}(\mc {B}N) $.
		\end{pr}
		
		\begin{proof}
			Since $ S $ is smooth, using \cite[Corollary 3.9]{Motivic_Vistoli}, from \eqref{ch2:_Loc_Seq_BN} we get a localization sequence for $ \mc BN_S, \widetilde{\pro_S(\sym^2(F))}:= \widetilde{\pro(\sym^2(F))}\times S $ and $ \widetilde{\pro_S(F)}:=\widetilde{\pro(F)}\times S $.  For any integer $ k $ we get:
			\[ 	\ldots \rightarrow  A^{2n}_{SL_2}(\pro_S(\sym^2(F)); \mo{}(k)) \stackrel{j^*_k}{\longrightarrow} \mr A^{2n}(\mc {B}N_S; \mo{}(k)) \stackrel{\partial_k^{\mr A}}{\longrightarrow} \mr A^{2n}_{SL_2}(\pro_S(F)) \rightarrow \ldots  \]
			Using the $ SL $-orientation, we can just consider $ k=0,1 $. Recall that $ \mc BN = \left[ \bigslant{\pro^2 \setminus C}{SL_2} \right]  $, where $ C $ is the conic given by the zero locus of the section $ Q=T_1^2-4T_0T_2 $ of $ \mo{\pro^2}(2) $. Applying our \cref{ch2:_tautological_symbol} to the section $ \lambda_{Q}: \left[ \bigslant{\pro^2}{SL_2} \right] \rightarrow \mo{\left[ \bigslant{\pro^2}{SL_2} \right]}(2) $ induced by $ Q $, we get a well defined element $   q_0  \in \mr{A}_{}^{0}(\mc BN_S) $ (cf. proof \cref{ch2:_5.3_MEC_field}). The structure map $ \pi_S: S \rightarrow \sk $ induces a map $ g: \mc BN_S\rightarrow \mc BN_\sk $ and hence a pullback map:
			\[ g^*: \mr A^{\bullet}(\mc BN_{\sk}; \gamma_N)  \longrightarrow \mr A^{\bullet}(\mc BN_S; \gamma_N)\]
			Consider $ q_{1,\sk} \in \mr A^{0}(\mc BN_{\sk}; \gamma_N)  $ constructed in the proof of \cref{ch2:_5.3_MEC_field}, and set:
			\[ q_1:=g^* q_{1,\sk} \mr A^{0}(\mc BN_{S}; \gamma_N) \]
			We then have two maps:
			
			\begin{equation}\label{ch2:_eq_BN_S}
				\begin{array}{cccc}
					(j^*, \sigma_0): & \mr A^{\bullet}(\mc BSL_{2,S}) \oplus \mr A^{\bullet}(S)   & \longrightarrow & \mr A^{\bullet}(\mc BN_S)
				\end{array}
			\end{equation}
			
			\begin{equation}\label{ch2:_eq_BN_S_twisted}
				\begin{array}{cccc}
					(j^*(-) e(\oocatname{T}), \sigma_1): & \mr A^{\bullet-2}(\mc BSL_{2,S}) \oplus \mr A^{\bullet}(S) & \longrightarrow & \mr A^{\bullet}(\mc BN_S; \gamma_N) 
				\end{array}
			\end{equation}
			\noindent where $ \sigma_0: \mr A(S) \rightarrow \mr A(\mc BN) $ sends $ 1 \mapsto q_0 $ and $ \sigma_1: \mr A(S)\rightarrow \mr A(\mc BN; \gamma_N) $ sends $ 1 \mapsto q_1 $.\\
			
			We want now to apply the homotopy Leray (or Gersten) spectral sequence of \cite{Asok-Deglise-Nagel} to both sides of \eqref{ch2:_eq_BN_S} and \eqref{ch2:_eq_BN_S_twisted}, but the cautious reader might object that we are dealing with algebraic stacks and not schemes any more. But by \cref{ch2:_Repr_Stack-Coh}, we can then apply the results in \cite{Asok-Deglise-Nagel} to the motivic spectra $ \mr A_{\mc BN} $ and $ \mr A_{\mc BSL_{2}} $ representing $ \mr A(\mc BN_S) $ and $ \mr A(\mc BSL_{2,S}) $ for $ S \in \catname{Sm}_{\bigslant{}{\sk}} $. Namely, we have that:
			\[ \mr A_{\mc BSL_2}(S)=\mr A(BSL_{2,S}) \]
			\[ \mr A_{\mc BN}(S)=\mr A(\mc BN_S) \]
			\[ \mr A_{\mc BN, \gamma_N}(S)=\mr A(\mc BN_S; \gamma_N) \]
			\noindent and we can apply the results in \cite{Asok-Deglise-Nagel} to these spectra. By \cite[Theorem 4.2.9]{Asok-Deglise-Nagel} (with $ f=Id $ in \textit{loc. cit.}), we have spectral sequences:
			\[ E_1^{p,q}:=\bigoplus_{s \in S^{(p)}} \mr A^{q}_{\mc BN}(\kappa(s)) \allora \mr A^{p+q}_{\mc BN}(S)=\mr A(\mc BN_S)   \]
			\[ 'E_1^{p,q}:=\bigoplus_{s \in S^{(p)}} \left( \mr A^{q}_{\mc BSL_2}(\kappa(s)) \oplus \mr A^{q}(\kappa(s)) \right) \allora \mr A^{p+q}(\mc BSL_{2,S})\oplus \mr A^{p+q}(S) \]
			
			The map $ (j^*, \sigma_0): \mr A^{\bullet}(\mc BSL_{2,S}) \oplus \mr A^{\bullet}(S)    \longrightarrow \mr A^{\bullet}(\mc BN_S) $ induces a map between spectral sequences $ E_1 $ and $ 'E_1 $, but the latter is an isomorphism by \cref{ch2:_5.3_MEC_field} and hence $ (j^*, \sigma_0) $ is an isomorphism too. By a similar argument, applying again \cref{ch2:_5.3_MEC_field} at the level of spectral sequences, we get that:
			\[ (j^*(-) e(\oocatname{T}), \sigma_1):  \mr A^{\bullet-2}(\mc BSL_{2,S}) \oplus \mr A^{\bullet}(S)  \longrightarrow  \mr A^{\bullet}(\mc BN_S; \gamma_N)  \]
			\noindent is an isomorphism too and we are done.
			
		\end{proof}
		
		\begin{rmk}\label{ch2:_rmk_mult_struct}
			\begin{enumerate}
				\item From \cite[Theorem 10]{Ananyevskiy_PhD_Thesis} (plus \cite[Proposition 3.32, Proposition 3.33]{Motivic_Vistoli}), we have that $ \mr A^{\bullet}(\mc BSL_2)\simeq \mr A^{\bullet}(S)\llbracket e \rrbracket $, where $ e=e(E_2) $ is the Euler class of the bundle associated to the tautological  rank two bundle on $ \mc BSL_2 $. Denoting by $ p: \mc BN \rightarrow \mc BSL_2 $, from the computations we just made, we then have that $ \mr A^{\bullet}(\mc BN)\simeq \mr A^{\bullet}(S)\llbracket p^*e \rrbracket \oplus q_0^{\mr A}\cdot \mr A^{\bullet}(S) $. Then if $ \xi: S \rightarrow \mc BN $ is the base-point of $ \mc BN $, the projection of $ \mr A^{\bullet}(\mc BN; \cdot) $ onto the second factor $ \mr A^{\bullet}(S) $ in our previous theorem, is just $ \xi^* $. Hence computing:
				\[ \partial_1((1+q_0^{\mr A})\cdot e(\oocatname{T}) )=\partial_0(q_0^{\mr A})\cdot \xi^*e(\oocatname{T}) =1 \cdot e(\xi^*\oocatname{T})=0 \]
				\noindent since $ \oocatname{T} $ gets trivialised once pulled back to the base-point. Since $  \partial_1((1+q_0^{\mr A})\cdot e(\oocatname{T}) )=0 $, $  (1+q_0^{\mr A})\cdot e(\oocatname{T})  $ must live in the factor $  \mr A^{-2}(\mc BSL_2)\cdot e(\oocatname{T})  $ of $ \mr A^{2}(\mc BN; \gamma_N) $.  Thus we have that $ (1+q_0^{\mr A})\cdot e(\oocatname{T}) = \lambda_1^{\mr A}\cdot e(\oocatname{T}) $ for some $ \lambda_1^{A} \in \mr A^{0}(\mc BSL_2) $. This is particular true for $ \mr{MSL}_{\eta} $, and hence, by universality of $ \mr{MSL} $ as $ SL $-oriented theory,  we have $ \lambda_1^A=\varphi^{SL}_{\eta}(\lambda_1^{\mr{MSL}_{\eta}}) $. 
				\item For $ A=H\mc W $, our $ q_0^{H\mc W} $ is exactly $ \langle \bar{q} \rangle $ constructed in \cite[\S 5]{Motivic_Euler_Char} (under the identification of \cite[Proposition 3.33]{Motivic_Vistoli}).
			\end{enumerate}
			
		\end{rmk}

		\subsection{The Multiplicative Structure of $ \mr{KW}^{\bullet}(BN) $}
		The results in the previous subsection gave us the additive description of  $ \mr{A}^{\bullet}(\mc BN_S) $ for an $ SL_{\eta} $-oriented ring spectrum $ \mr A $ and $ S $ a smooth $ \sk $-scheme. From now on, up to the end of this chapter, if not otherwise specified, we will always work over a smooth $ \sk $-scheme $ S $, so we will omit the subscript from the notation.\\
		
		Let us now proceed with the computation of the multiplicative structure of $ \mr{KW}^{\bullet}\left( \mc BN \right) $ and the $ \mr{KW}^{\bullet}\left( \mc BSL_2 \right) $-module structure of $ \mr{KW}^{\bullet}\left(  \mc BN; \gamma_N \right) $. Let $ q_0:=q_0^{\mr{KW}} $, $ q_1:=q_1^{\mr{KW}} $ be the elements constructed in the proof of \cref{ch2:_5.3_MEC}.  We have that $ q_0^2=1 $ (cf. \cite[6.2]{Ananyevskiy_SL_oriented}). 
		
		
		Recall from \cref{ch2:_rmk_mult_struct} (for $ \mr A=\mr{KW} $), we have that $ \mr{KW}^{\bullet}(\mc BN)\simeq \mr{KW}^{\bullet}(S)\llbracket p^*e \rrbracket \oplus q_0\cdot \mr{KW}^{\bullet}(S) $.\\
		
		
		The only relations  left to compute are: \vspace{-0.75em}\newline
		\begin{minipage}{0.23\textwidth}
			\[  q_0\cdot p^*e \in \mr{KW}^{2}\left( BN \right)    \]
		\end{minipage}
		\hfill
		\begin{minipage}{0.28\textwidth}
			\[  q_0\cdot j^*e(\oocatname{T}) \in \mr{KW}^{2}\left( \mc BN ; \gamma_N \right)   \]
		\end{minipage}
		\hfill
		\begin{minipage}{0.23\textwidth}
			\[ q_0\cdot q_1 \in q_1 \cdot \mr{KW}^{0}(S)  \]
		\end{minipage}\vspace{0.25em}\\
		
		To do this we can consider the inclusion $ \mb G_m \into N $ and the corresponding map $ \mc B\mb G_m \longrightarrow \mc BN $, where $ \mc B\mb G_m\simeq \left[\quot{\left(\quot{SL_2}{\mb G_m}\right)}{SL_2}\right] $. The description of $ \mc BN $ as the quotient  $ \left[\bigslant{ \left(\pro\left( \sym^2(F) \right) \setminus \pro(F)\right)}{SL_2}\right] $, gives us a section $ j^*Q $ of $ j^*\mo{}(2) $ coming from the section $ Q:=T_1^2-4T_0T_2 $ of $ \mo{\pro(\sym^2F)}(2) $, via the restriction along $ j:\mc BN \into \widetilde{\pro(\sym^2)(F)}  $.\\
		
		Given an algebraic stack $ X $, a  line bundle $ \mb V(\mc L) $ on $ X $ and a section $ s $ of $ \mc L^{\otimes 2} $, we can construct another algebraic stack $ X(\sqrt{s}) $ as the fiber product:
		
		\begin{center}
			\begin{tikzpicture}[baseline={(0,2)}, scale=2]
				\node (a) at (0,1) {$ X(\sqrt{s}) $};
				\node (b) at (1, 1) {$L $};
				\node (c)  at (0,0) {$ X $};
				\node (d) at (1,0) {$ L^{\otimes 2} $};
				\node (e) at (0.2,0.8) {$ \ulcorner $};

				\path[font=\scriptsize,>= angle 90]
				
				(a) edge [->] node [above ] {$  $} (b)
				(a) edge [->] node [above] {$  $} (c)
				(b) edge[->] node [right] {$ sq $} (d)
				(c) edge [closed] node [below] {$ \sigma $} (d);
			\end{tikzpicture}
		\end{center}
		\noindent where $ sq: L \longrightarrow L^{\otimes 2} $ is the squaring map and $ \sigma: X \longrightarrow L^{\otimes 2} $ is the map induced by the section $ s $. 
		
		
				%
				%
				%
				%

		In particular, on $ \mc BN $ we have a section $ j^*Q $ of $ j^*\mo{}(2)= j^*\mo{}(1)^{\otimes 2} $ and we can consider the stack $ \mc BN(\sqrt{j^*Q}) $. 
		Similarly to what was showed in \cite[End of Sect. \S2 and \S5]{Motivic_Euler_Char}, we can identify $ \mc B\mb G_m \simeq \mc BN(\sqrt{Q}) \rightarrow \mc BN $ as a double cover. 
		
		\begin{lemma}\label{ch2:_Lemma_5.4_MEC_field_pt1}
			Consider $ S=\spec(\sk) $. Then we have $ (1+q_0)\cdot p^*e=0 $ in $ \mr{KW}^{2}\left( \mc BN \right) $.
		\end{lemma}
		
		\begin{proof}
			We will freely use the notation employed in the proof of  \cref{ch2:_5.3_MEC} and we will closely follow the proof in \cite[Lemma 5.4]{Motivic_Euler_Char}. Let us recall here the underlying geometry of our objects. We have a rank two tautological bundle $ {E}_2=\left[\bigslant{F}{SL_2}\right]\longrightarrow \mc BSL_2 $ and its pullback $ \hat{E}_2:=p^*{E}_2 \longrightarrow \mc BN $. The class $ e \in \mr{KW}^{2}\left( \mc BSL_2 \right) $ was the Euler class of $ {E_2} $, and so we have $ p^*e=e(\hat{E}_2) $. We have the double cover $ \varphi: \mc B\mb G_m \longrightarrow \mc BN $. The pullback $ \varphi^*\hat{E}_2 $ splits as $ \varphi^*\hat{E}_2\simeq \mo{}(1)\oplus \mo{}(-1) $ corresponding to the decomposition of $ F $ into the eigenspaces relative to $ t $ and $ t^{-1} $ under the $ \mb G_m $-action. The $ \bigslant{\Z}{2\Z}\simeq \bigslant{N}{\mb G_m} $-action sends $ t \mapsto t^{-1} $ and thus swaps the two factors $ \mo{}(1) $ and $ \mo{}(-1) $. Considering the cone $  \left(\mo{}(1)\times 0\right) \cup \left( 0\times \mo{}(-1) \right) \sseq \varphi^*\hat{E}_2 $, we get by descent under the $ \bigslant{\Z}{2\Z} $-action the corresponding cone $ \mf C \sseq \hat{E}_2 $. We have a map $ \mf \nu: \mf C^{N}:= \mo{}(1) \longrightarrow \mf C $, induced by the normalization of the atlas of $ \mf C $\footnote{All the geometry we have done so far is the geometry of $ SL_2 $-quotient stacks, we can just work with their standard atlas and then by descent pass to the $ SL_2 $-quotients.}. So we have the following commutative diagram:
			\begin{center}
				\begin{tikzpicture}[baseline={(0,0)}, scale=2]
					\node (a) at (0,1) {$ \mf C^{N}:= \mo{}(1) $};
					\node (b) at (1, 1) {$ \varphi^*\hat{E}_2 $};
					\node (c)  at (0,0) {$ \mf C $};
					\node (d) at (1,0) {$ \hat{E}_2 $};
					\node (e) at (2,1) {$ \mc B\mb G_m $};
					\node (f) at (2,0) {$ \mc BN $};

					\path[font=\scriptsize,>= angle 90]
					
					(a) edge [->] node [above ] {$ {\iota}_{\mf C^{N}} $} (b)
					(a) edge [->] node [left] {$ \nu $} (c)
					(b) edge[->] node [left] {$ \hat{\varphi} $} (d)
					(c) edge [->] node [below] {$ \iota_{\mf C} $} (d)
					(b) edge [->] node [above] {$ \rho $} (e)
					(d) edge [->] node [below] {$ r $} (f)
					(e) edge [->] node [right] {$ \varphi $} (f);
				\end{tikzpicture}
			\end{center}
			\noindent Considering the localization sequences associated to: \\

			\begin{minipage}[h]{0.35\textwidth}
				\begin{center}
					\begin{tikzpicture}[baseline={(0,0)}, scale=1]
						\node (c) at (0,0) {$ 0_{\mf C}=\mc{B}N $};
						\node (b) at (2, 0) {$ \mf C  $};
						\node (a)  at (4,0) {$ \mf C \setminus 0=: \ \mf C^{\circ} $};

						\path[font=\scriptsize,>= angle 90]
						
						(a) edge [open'] node [above ] {$  $} (b)
						(c) edge [closed] node [above] {$  $} (b);
					\end{tikzpicture}\\
					\begin{tikzpicture}[baseline={(0,0)}, scale=1]
						\node (c) at (0,0) {$ 0_{\mf C^{N}}=\mc{B}\mb G_m $};
						\node (b) at (2, 0) {$ \mf C^{N}  $};
						\node (a)  at (5,0) {$ \mf C^{N} \setminus 0=:\ (\mf C^{N})^{\circ}\simeq \mf C^{\circ} $};

						\path[font=\scriptsize,>= angle 90]
						
						(a) edge [open'] node [above ] {$  $} (b)
						(c) edge [closed] node [above] {$  $} (b);
					\end{tikzpicture}
				\end{center}
			\end{minipage}
			\hfill
			\begin{minipage}[h]{0.37\textwidth}
				\begin{center}
					\begin{tikzpicture}[baseline={(0,0)}, scale=1]
						\node (c) at (0,0) {$ 0_{\hat{E}_2}=\mc{B}N $};
						\node (b) at (2.5, 0) {$ \hat{E}_2 $};
						\node (a)  at (4.5,0) {$ \hat{E}_2\setminus 0=:\ \hat{E}_2^{\circ} $};

						\path[font=\scriptsize,>= angle 90]
						
						(a) edge [open'] node [above ] {$  $} (b)
						(c) edge [closed] node [above] {$ s_0  $} (b);
					\end{tikzpicture}\\
					\begin{tikzpicture}[baseline={(0,0)}, scale=1]
						\node (c) at (0,0) {$ 0_{\varphi^*\hat{E}_2}=\mc{B}\mb G_m $};
						\node (b) at (2.5, 0) {$ \varphi^*\hat{E}_2  $};
						\node (a)  at (4.5,0) {$ \varphi^*\hat{E}_2^{\circ} $};

						\path[font=\scriptsize,>= angle 90]
						
						(a) edge [open'] node [above ] {$  $} (b)
						(c) edge [closed] node [above] {$ \sigma_0 $} (b);
					\end{tikzpicture}
				\end{center}
			\end{minipage}\\
			
			\noindent we get the following diagram:

			\begin{center}
				\begin{equation}\label{ch2:_big_diagram_loc_seq_BG_BN}
					\begin{tikzpicture}[baseline=(current  bounding  box.south), scale=0.5]
						
						\matrix (m) [matrix of math nodes, row sep=2 em,
						column sep=2 em]{
							\mr{KW}^{0,0}\left( \mf C^{\circ}; \mo{}(-1) \right)  & &  \mr{KW}^{-1,-1}\left( \mc B\mb G_m\right) &\\
							& \mr{KW}^{2,1}\left( \varphi^*\hat{E}_2^{\circ}\right) & & \mr{KW}^{-1,-1}\left( \mc B \mb G_m\right) & \\
							\mr{KW}^{0,0}\left( \mf C^{\circ}; \mo{}(-1) \right)  & &   \mr{KW}^{-1,-1}\left( \mc BN\right)   &\\ 
							&\mr{KW}^{2,1}\left( \hat{E}_2^{\circ} \right)& &  \phantom{c} \mr{KW}^{-1,-1}\left( \mc BN\right)   \\};
						\path[-stealth]
						(m-1-1) edge [->] node [above] {$ \partial_{\mf C^{N}} $} (m-1-3) edge[->] node [left] {$ \iota_{\bigslant{\mf C^N}{\varphi^*E}} {}_* $} (m-2-2)
						edge [->] node [left] {$ \hat \varphi_* $}  (m-3-1)
						(m-1-3) edge [->] node [below] {\phantom{ciao}$ \varphi_* $} (m-3-3) edge  [->]  node [right] {$ Id $} (m-2-4)
						(m-2-2) edge [-,line width=6pt,draw=white] (m-2-4) edge [->] node [above] {$ \partial_{\varphi^*\hat E_2}$}  (m-2-4) edge [right hook->] (m-4-2)
						(m-3-1) edge [->] node [above] {\phantom{ciao} $ \partial_{\mf C} $} (m-3-3)
						edge [->] node [left] {$ \iota_{\bigslant{\mf C}{\hat E_2}} $}  (m-4-2)
						(m-4-2) edge [->] node [below] {$ \partial_{\hat{E}_2} $}  (m-4-4)
						(m-3-3) edge [->] node [right] {$ Id $}  (m-4-4)
						(m-2-2) edge [-,line width=6pt,draw=white] (m-4-2)
						(m-2-2) edge [->] (m-4-2)
						(m-2-4)edge [->] (m-4-4);
					\end{tikzpicture}
				\end{equation}
			\end{center}
			
			\noindent where we used the fact that the normal bundle of $ \mf C^{N} $ inside $ \varphi^*\hat{E}_2 $ is $ N_{\bigslant{\mf C^{N}}{\varphi^*\hat{E}_2}}\simeq \restrict{\rho^*\mo{\mc B\mb G_m}(-1)}{\mf C^{N}} $. 
			
			Let $ \pi: \mf C^{N}\simeq \mo{\mc B\mb G_m}(1) \rightarrow\mc  B\mb G_m $ be the line bundle map, then we can consider on $ \mo{\mf C^{N}}(1)\simeq \pi^*\mo{}(1) $ the tautological section $ can: \mf C^{N}\simeq \mo{ \mc B\mb G_m}(1) \longrightarrow \mo{\mf C^{N}}(1) $\footnote{For a line bundle $ p:L \rightarrow X $, we get a tautological section on $ p^*L \rightarrow L $ from the fact that for every $ y=p(l) $ we have $ (p^*L)_l\simeq L_y \ni l $, so the section in this case is just $ l \mapsto (l,l) $.}. By \cref{ch2:_symbol_of_line_bundle_section}, we get a well defined element $ \langle t_{can} \rangle \in \mr{KW}^{0,0}(\mf C^{\circ}; \mo{}(1)) $. \\
			By a Mayer-Vietoris argument, to compute $ \partial_{\mf C^{N}}(\langle t_{can}\rangle) $ we can restrict to trivialising open subsets, and we have that $ \partial_{\mf C^{N}}(\langle t_{can}\rangle) =\eta \in \mr{KW}^{-1,-1}(\mo{\mc B\mb G_m}(1)) $. But by homotopy invariance we get $ \mr{KW}^{-1,-1}(\mo{\mc B\mb G_m}(1)) \simeq \mr{KW}^{-1,-1}(\mc B\mb G_m) $. Using the isomorphism induced by $ \eta $, we can identify $ \mr{KW}^{-1,-1}(\mc B\mb G_m)\simeq \mr{KW}^0(\mc B\mb G_m) $: under this isomorphism $ \eta $ is sent to $ 1 $, so we have $ \partial_{\mf C^{N}}(\langle t_{can}\rangle) =1 \in \mr{KW}^{0}(\mc B\mb G_m) $.\\

			If we push forward through $ \varphi_* $ the boundary of $ \langle t_{can} \rangle  $, we get $ \varphi_*\langle 1 \rangle= \langle 2 \rangle (1+ q_0) $ by \cref{ch2:_Pushforward_1} (see below).


			Using the commutativity of the diagram \cref{ch2:_big_diagram_loc_seq_BG_BN}, we see that  $ \langle 2 \rangle (1+ q_0 )=\varphi_*\left( \partial_{\mf C^{N}}\left( \langle t_{can }\rangle  \right) \right)\simeq \partial_{\hat{E}_2}\left( \hat \varphi_*(\iota_{\mf C^{N}})_* \langle t_{can} \rangle  \right) $ is a boundary in $ \mr{KW}^{0}\left( BN \right) $, so it is sent to zero via $ (s_0)_* $, where $ s_0: BN \rightarrow \hat E_2 $ is the zero section of the bundle. But if we post-compose with the pullback map $ s_0^* $, we also get $ (s_0)^*(s_0)_*\varphi_*\langle 1 \rangle=0 $, and spelling out the element on the left we have $ \langle 2 \rangle (1+ q_0 )\cdot e(\hat{E}_2)=0 $ as we wanted.
			
		\end{proof}
		
		We need to complete the previous proof, but before doing that we need another technical lemma.  Recall that by \cite[Corollary 4]{Ananyevskiy_PhD_Thesis}, for $ m=2k+1 $ and for any $ SL_{\eta} $-oriented ring spectrum $ \mr A\in \mr{SH}(\sk) $, we have:
		\[ \mr{A}^{\bullet}(B_mSL_2)=\bigslant{\mr{A}^{\bullet}(\sk) [e]}{(e^{m-1})}\]
		\noindent  where $ e $ is the Euler class of the tautological bundle. 
		\begin{lemma}\label{ch2:_Pushforward_pt0}
			Let $ \mr A \in \mr{SH}(\sk) $ be a $ SL_{\eta} $-oriented ring spectrum.  Denote by $ \nu_m: B_mN\rightarrow \mc BN $ the natural map to the quotient stack and by $ p: \mc BN\rightarrow \mc BSL_2 $ the structure map over $ \mc BSL_2 $. Let $ e $ be the Euler class of the tautological bundle over $ \mc BSL_2 $. Then,  for each odd integer $ m $,the pullback map:
			\[ \nu_m^*: \mr A^{\bullet}(\mc BN)\simeq \mr A^{\bullet}(\sk)\llbracket p^*e\rrbracket \oplus  q_0\cdot \mr A^{\bullet}(\sk) \longrightarrow \mr A^{\bullet}(B_mN) \]
			\noindent is surjective with kernel $ \ker(\nu_m^*)=(p^*e^{m-1})\cdot \mr A^{\bullet}(\sk)\llbracket p^*e \rrbracket $.
		\end{lemma}
		
		\begin{proof}
			Denote for short:
			\[ \pro_k(F):=\pro(F)\times^{SL_2}E_mSL_2 \]
			\[ \pro_m(\sym^2(F)):=\pro(\sym^2(F))\times^{SL_2} E_mSL_2 \]
			\noindent and let $ j_m: B_mN \into \pro_m(\sym^2(F)) $ and $ \iota_m: \pro_m(F)\into \pro_m(\sym^2(F)) $ be the associated open and closed immersions. Let:
			\[ \pi_m: \pro_m(\sym^2(F))\longrightarrow B_mSL_2 \]
			\noindent be the structure map over $ B_mSL_2 $.
			
			We will start proving the following:
			\begin{claim}\label{ch2:_claim_iota_m_is_zero}
				The map induced by proper pushforward along $ \iota_m $:
				\[ 	(\iota_m)_*: \mr A^{\bullet}(\pro_m(F))\longrightarrow \mr A^{\bullet+1}(\pro_m(\sym^2(F))) \]
				\noindent is the zero-map.
			\end{claim}
			\textit{Proof of the Claim.} Consider $ T_{\pi_m} $ the relative tangent bundle of $ \pi_m: \pro_m(\sym^2(F)) \rightarrow B_mSL_2 $. The fibers of $ \pi_m $ are $ \pi_m^{-1}(x)\simeq \pro^2_{\kappa(x)} $, hence comparing the Leray spectral sequences it is easy to see that:
			\[ \pi_m^*: \mr{A}^{\bullet}(B_mSL_2)\stackrel{\sim}{\longrightarrow} \mr{A}^{\bullet}\left( \pro(\sym^2(F)) \times^{SL_2} E_mSL_2 \right) \]
			\noindent is an equivalence. Indeed, the map between spectral sequences, induced by $ \pi_m^* $ above, is an isomorphism  by \cref{ch2:_SL_eta_Lemma_1}. We also have an isomorphism given by:
			\[ (-\cup e(T_{\pi_m}))\circ \pi_m^*: \mr{KW}^{\bullet-2}(B_mSL_2) \stackrel{\sim}{\longrightarrow} \mr{A}^{\bullet}\left( \pro(\sym^2(F)) \times^{SL_2} E_mSL_2 \right) \]
			This follows again from the comparison of Leray spectral sequences, using the fact that on the fibers the map induced by $ (-\cup e(T_{\pi_m}))\circ \pi_m^* $ is an isomorphism (cf. end of proof of \cref{ch2:_SL_eta_Lemma_2}). In particular we get that the map:
			\[ (-\cup e(T_{\pi_m})): \mr{A}^{\bullet}(\pro_m(\sym^2(F))) \longrightarrow \mr{A}^{\bullet+2}(\pro_m(\sym^2(F)); \mo{}(1)) \]
			\noindent is an isomorphism. But $ i_m^*T_{\pi_m} $ fits in a short exact sequence:
			\[ 0\rightarrow T_{\pi_m\circ i_m} \rightarrow i_m^*T_{\pi_m} \rightarrow N_{i_m}\rightarrow 0 \]
			\noindent where $ N_{i_m} $ is the normal bundle of the closed immersion. This implies that $ i_m^*e(T_{\pi_m})=0 $, indeed $ T_{\pi_m\circ i_m} $ and $ N_{i_m} $ are line bundles and Euler classes of line bundles are trivial for $ SL_{\eta} $-oriented spectra (see \cite[Lemma 4.3]{Motivic_Euler_Char}). Then by the projection formula, for any $ x \in \mr{KW}^{\bullet-1} $ we have:
			\[ (i_m)_*(x)\cup e(T_{\pi_m})=(i_m)_*(x\cup i_m^*e(T_{\pi_m}))=0 \]
			But $ (-\cup e(T_{\pi_m})) $ is an isomorphism, in particular it is injective and this implies that we must have $ i_m^*(x)=0 $ as claimed.
			\begin{flushright}
				$ \stackrel{\text{(Claim 1)}}{\blacksquare} $
			\end{flushright}

			From \cref{ch2:_claim_iota_m_is_zero}, it follows that for each integer $ k $ the localization sequence: \begin{equation}\label{ch2:_eq_Loc_Sec_BmN}
				\ldots \stackrel{(\iota_m)_*}{\rightarrow} \mr A^{k}(\pro_m(\sym^2(F)))\stackrel{j_m^*}{\longrightarrow} \mr A^{k}(B_mN)\stackrel{\partial_m}{\longrightarrow} \mr A^{k}(\pro_m(F))\stackrel{(\iota_m)_*}{\rightarrow} \ldots  
			\end{equation}
			\noindent  splits into short exact sequences:
			\[ 0 \rightarrow \mr A^{k}(\pro_m(\sym^2(F)))\stackrel{j_m^*}{\longrightarrow} \mr A^{k}(B_mN)\stackrel{\partial_m}{\longrightarrow} \mr A^{k}(\pro_m(F))\rightarrow 0 \]
			Let $ x_0 $ be the base point of $ B_mSL_2 $ and consider $ y_0 \in \pi_m^{-1}(x_0) $ the $ \sk $-point $ y_0: \spec(\sk) \rightarrow \pro_m(\sym^2(F)) $ given by $ [1:0] $ in the fiber of $ x_0 $. Then the map:
			\begin{equation}\label{ch2:_eq_split_inj_BmN}
				\pi_{\pro_m(\sym^2(F))}^*: \mr A^{\bullet}(\sk) \longrightarrow \mr A^{\bullet}(\pro_m(\sym^2(F))) 
			\end{equation}
			\noindent is injective, and splits via $ y_0^* $. Let $ q_0:=q_0^{\mr A}\in \mr A^0(\mc BN) $ be the element we used to get the splitting of \cref{ch2:_5.3_MEC_field} and denote by $ q_{0,m} $ its pullback to $ B_mN $. Recall that the map $ p_m:  B_mN \rightarrow  B_mSL_2 $ is given by the composition of $ j_m $ and $ \pi_m $ (and similarly for $ p: \mc BN\rightarrow \mc BSL_2 $). Then by \eqref{ch2:_eq_Loc_Sec_BmN} and \eqref{ch2:_eq_split_inj_BmN}, we deduce that the map:
			\[ \psi_m:=(p_m^*, q_{0,m}\cdot\pi_{B_mN}^*):  \mr A^{\bullet}(B_mSL_2)\oplus \mr A^{\bullet}(\sk) \longrightarrow \mr A^{\bullet}(B_mN) \]  
			\noindent is injective. Therefore, denoting by  $ \sigma_m: B_mSL_2 \rightarrow \mc BSL_2 $ the natural map to the quotient stack,  we get a commutative diagram:
			
			\begin{equation}\label{ch2:_eq_PF_diagram}
				\begin{tikzpicture}[baseline={(0,0)}, scale=2]
					\node (a) at (0,1) {$ \mr{A}^{\bullet}(\mc BSL_2)\oplus \mr A^{\bullet}(\sk)$};
					\node (b) at (2, 1) {$ \mr{A}^{\bullet}(\mc BN) $};
					\node (c)  at (0,0) {$  \mr{A}^{\bullet}(B_mSL_2)\oplus \mr A^{\bullet}(\sk)$};
					\node (d) at (2,0) {$ \mr{A}^{\bullet}(B_mN) $};
					\node (e) at (0.2,0.8) {$   $};
					\node (f) at (0.5,0.5) {$  $};

					\path[font=\scriptsize,>= angle 90]
					
					(a) edge [->] node [above ] {$ \sim $} (b)
					(a) edge [->] node [left] {$ (\sigma_m^*, Id) $} (c)
					(b) edge[->] node [right] {$ \nu_m^* $} (d)
					(c) edge [right hook->] node [below] {$ \psi_m $} (d);
				\end{tikzpicture}
			\end{equation}
			But this implies that $ \ker(\nu_m^*)\simeq \ker(\sigma_m^*) $. For $ m $ odd, by \cite[Corollary 4, Theorem 10]{Ananyevskiy_PhD_Thesis}, we have:
			\[ \ker(\sigma_m^*)\simeq (e^{m-1})\cdot \mr A^{\bullet}(\sk)\llbracket e \rrbracket  \]
			\noindent and we are done.
			
		\end{proof}
		
		As promised, let us now complete the proof of \cref{ch2:_Lemma_5.4_MEC_field_pt1} where we used the following computation:
		\begin{lemma}\label{ch2:_Pushforward_1}
			Let $ S=\spec(\sk) $ and let $ \varphi: \mc B\mb G_m \longrightarrow \mc BN $ be the double cover we already introduced. Then we have:
			\[ \varphi_*\langle 1 \rangle=\langle 2\rangle (1+q_0) \in \mr{KW}^0(\mc BN) \]
		\end{lemma}
		
		\begin{proof}
			As already mentioned, we have $ \mc B\mb G_m\simeq \mc BN(\sqrt{Q}) $ and  we can use Grothendieck-Serre Duality to compute $ \varphi_*1 $ for $ \mr{KW} $ (cf. \cite[\S 8D]{Levine_Raksit_Gauss_Bonnet}). We will reduce the computation on $ \mc BN $ to a computation on the finite level approximations given by $ B_mN $. From \cref{ch2:_5.3_MEC_field}, we know that:
			\[ \mr{KW}^{0}(\mc BN)\simeq\left( \mr{KW}^{\bullet}(\sk)\llbracket p^*e\rrbracket \right)^{0} \oplus \mr{KW}^{0}(\sk)\cdot q_0 \]
			Hence:
			\[ \varphi_*\langle 1\rangle=\sum_{i=0}^{\infty} a_i e^{2i} + b \cdot q_0 \]
			\noindent where $ a_i \in \mr{KW}^{-4i}(\sk) $ and $ b \in \mr{KW}^0(\sk) $. Once we determine the coefficients $ a_i $'s and $ b $ we are done. We have a natural map:
			\[ \alpha: \mr{KW} ^{0}(\mc BN) \longrightarrow H\mc W^{0}(\mc BN) \]
			\noindent from the $ 0^{th} $ Witt theory to the $ 0^{th} $ Witt sheaf cohomology, induced by sheafification. Under this map, we have:
			\[ \alpha(\varphi_*\langle 1\rangle)=a_0+ b \cdot q_0^{H\mc W} \]
			But we also have:
			\[ \alpha(\varphi_*\langle 1\rangle)= \varphi_*^{H\mc W}\langle 1\rangle \]
			\noindent where $ \varphi_*^{H\mc W} $ is the pushforward map on Witt sheaf cohomology. By \cite[Proof of Proposition 5.3]{Motivic_Euler_Char}, we know that $ \varphi_*^{H\mc W}\langle 1\rangle=\langle 2\rangle (1+ q_0^{H\mc W}) $ and hence:
			\[ \alpha(\varphi_*\langle 1\rangle)=a_0+b\cdot q_0^{H\mc W}= \langle 2\rangle(1+q_0^{H\mc W}) \]
			\noindent implying that:
			\begin{equation}\label{ch2:_eq_PF_<1>}
				a_0=b=\langle 2\rangle 
			\end{equation}
			
			It remains to determine all the remaining $ a_i $'s for $ i>0 $. By \cref{ch2:_Pushforward_pt0}, these coefficients are determined via the pullback map to $ B_mN $ for $ m $ going to infinity. 
			By construction $ B_m \mb G_m=B_mN(\sqrt{Q_m}) $, where $ Q_m $ is the section obtained via pullback from the $ \mo{\pro^2\setminus C}(2) $-section $ Q=T_1^2-4T_0T_2 $. Indeed, if $ t $ denotes the tautological section of $ \pi^*\mo{}(1) $, associated to $ \pi: \mo{}(1)\rightarrow B_mN $, then $ B_m\mb G_m\simeq\mc V(t^2-\pi^*Q_m)\sseq \pi^*\mo{}(1) $ is the zero locus of the section $ (t^2-\pi^*Q_m) $. The map on locally free sheaves $ \mo{}(-1)\rightarrow \mo{B_m\mb G_m} $ sends a local section $ y $ of $ \mo{}(-1) $ to $ yt $ restricted to $ B_m\mb G_m $. Let $ \varphi_m: B_m\mb G_m\rightarrow B_mN $ be the map between the finite approximations, then we get that $ (\varphi_m)_* \mo{B_m\mb G_m}=\mo{B_mN}\oplus \mo{}(-1) $. A local section of $ \varphi_m{}_*\mo{B_m\mb G_m} $ will then be of the for $ x+yt $, with $ x,y $ local sections of $ \mo{B_mN} $ and $ \mo{}(-1) $ respectively. Due to the fact that $ \varphi_m $ is \'etale, we have a well defined trace map $ \mr{Tr}_{\varphi_m}: (\varphi_m)_*\mo{{B_m\mb G_m}} \rightarrow \mo{B_mN} $. Then for $ x+yt $ local section of $ \varphi_m{}_*\mo{B_m\mb G_m}  $ we get that:
			\[ \mr{Tr}_{\varphi_m}((x+yt)^2)= \mr{Tr}_{\varphi_m}(x^2+Q_my^2+2xyt)=2x^2+2Q_my^2 \]
			In other words:
			\[ \mr{Tr}_{\varphi_m}(\langle 1\rangle)=\langle 2\rangle(1+q_{0,m}) \]
			\noindent  with $ q_{0,m} $ the pullback of $ q_0 $.\\
			

			Thanks to \cite[\S 8D]{Levine_Raksit_Gauss_Bonnet}, we can identify $ (\varphi_m)_*\langle 1\rangle $ with the quadratic form given by $ \mr{Tr}_{\varphi_m}(\langle 1\rangle) $, that is:
			\begin{equation}\label{ch2:_eq_finite_level_PF_<1>}
				(\varphi_m)_*\langle 1\rangle=\langle 2\rangle(1+q_{0,m})
			\end{equation}
			
			Now let $ m $ be an odd integer. Notice that the difference $ \varphi_*\langle 1\rangle-\langle 2\rangle(1+q_0)  $ is sent to $ (\varphi_m)_*\langle 1\rangle-\langle 2\rangle(1+q_{0,m})  $ in $ \mr{KW}^0(B_mN) $, thus by \eqref{ch2:_eq_finite_level_PF_<1>} we get:
			\begin{equation}\label{ch2:_eq_PF_zero_finiet_level}
				\nu_m^*\left( \varphi_*\langle 1\rangle-\langle 2\rangle(1+q_0) \right)=(\varphi_m)_*\langle 1\rangle-\langle 2\rangle(1+q_{0,m}) =0  
			\end{equation}
			By \cref{ch2:_Pushforward_pt0} (for $ \mr A=\mr{KW} $), this implies that:
			\[ \nu_m^*\left( \varphi_*\langle 1\rangle-\langle 2\rangle(1+q_0)\right)=\sum_{i>0}^{m-1}a_ie^{2i}=0 \]
			and therefore $ a_i=0 $ for all $ 0<i<m $. For bigger and bigger $ m $, this gives us that $ a_i=0 $ for all $ i>0 $. Hence, together with \eqref{ch2:_eq_PF_<1>},  we have that:
			\[ \varphi_*\langle 1\rangle=\sum_i a_ie^{2i}+b\cdot q_0= \langle 2\rangle(1+q_0) \]
			\noindent as claimed.

		\end{proof}

		\begin{pr}[{\cite[Lemma 5.4]{Motivic_Euler_Char}}]\label{ch2:_Lemma_5.4_MEC_field}
			Consider $ S=\spec(\sk) $. We have $ (1+q_0)\cdot p^*e=0 $ in $ \mr{KW}^{2}\left( \mc BN \right) $, $ (1+q_0)\cdot j^*e(\oocatname{T})=0 $ in $ \mr{KW}^{2}\left( \mc BN; \gamma_N \right) $ and $ (1+q_0)\cdot q_1=q_1\in q_1\cdot \mr A^0(\sk)  $.
		\end{pr}
		\begin{proof}
			We will freely use the notation employed in the proof of  \cref{ch2:_5.3_MEC}. 
			We already proved that $ (1+q_0)\cdot p^*e=0 $ in \cref{ch2:_Lemma_5.4_MEC_field_pt1}. Let us prove the second statement, that is, $ (1+q_0)\cdot j^*e(\oocatname{T})=0  $. By \cref{ch2:_rmk_mult_struct}, there exists a unique $ \lambda:=\lambda_1^{\mr{KW}} \in \mr{KW}^{0}\left(\mc  BSL_2\right)  $ such that $ (1+q_0)\cdot j^*e(\oocatname{T})=\lambda \cdot e(\oocatname{T}) $. But then multiplying both sides by $ p^*e $, since we already know $ (1+q_0)\cdot p^*e=0 $, we get $ 0=(1+\bar{q})\cdot p^*e \cdot j^*e(\oocatname{T})=\lambda \cdot p^*e \cdot j^*e(	\oocatname{T}) $. An element in $ \mr{KW}^{0}(\mc BSL_2) $ is a power series $ \sum_i a_ip^*e^{i} $ with $ a_i \in \mr{KW}^{-2i}(S) $, so $ \lambda \cdot p^*e \cdot j^*e(	\oocatname{T})=0 $ implies that $ \sum_i a_i p^*e^{i+1}=0 \in \mr{KW}^{2}(\mc BSL_2) $, that is, $ a_i=0 \ \forall \  \ i $ and hence $ \lambda=0 $. \\

			Let us now show that $ (1+q_0)\cdot q_1=q_1 $. 
			Recall that $ q_1 $ in \cref{ch2:_5.3_MEC_field} was chosen to be any element that was sent to $ 1 $ under $ \partial_1 $. Let $ \tilde q_1\in \mr{KW}^{0}(\mc BN; \gamma_N) $ be another element such that $ \partial_1(\tilde q_1)=1 $. Then $ q_1-\tilde{q}_1\in \ker(\partial_1) $, that is:
			\[ q_1-\tilde q_1=\alpha \cdot j^*e(\mc T) \]
			\noindent for $ \alpha\in \mr{KW}^{-2}(\mc BSL_2) $. Then we have:
			\begin{equation}\label{ch2:_eq_last_relation_mult_thm}
				(1+q_0)q_1=(1+q_0)(\alpha\cdot j^*e(\mc T)+ \tilde q_1)=(1+q_0)\tilde{q}_1 
			\end{equation}
			
			\noindent since $ (1+q_0)j^*e(\mc T)=0 $. Without loss of generality we can therefore replace $ q_1 $ with any other $ \tilde q_1 \in \mr{KW}^{0}(\mc BN; \gamma_N) $ such that $ \partial_1(\tilde q_1)=1 $. On $ \pro^2 $ with coordinates $ [T_0:T_1:T_2] $, consider the quadratic form:
			\[ -(T_0x^2+T_1xy+T_2y^2) \]
			\noindent and set $ \tilde q_1\in \mr{KW}^{0}(\mc BN) $ to be the corresponding element. Let us check that the boundary of $ \tilde q_1 $ is indeed $ 1 $. Let $ \pi_{\widetilde{\pro(F)}}: \widetilde{\pro(F)}\rightarrow \spec(\sk) $ be the structure map of $ \widetilde{\pro(F)} $, then we know that $ \mr{KW}^{\bullet}(\sk)$ is isomorphic to $ \mr{KW}^{\bullet}(\widetilde{\pro(F)}) $ via $ \pi_{\widetilde{\pro(F)}}^* $. Let $ U $ be any dense open subset of $ \pro^2 $ where $ \mo{}(2) $ gets trivialised, denote by $ V:=U\cap C $ the corresponding dense open in $ C $ and denote by $ \pi_V: V\rightarrow \spec(\sk) $ the structure map of $ V $. Since there always exists a rational $ \sk $-point in $ V $, the map:
			\[ \pi_V^*: \mr{KW}^{\bullet}(\sk)\rightarrow \mr{KW}^{\bullet}(V)  \]
			\noindent must be injective and this implies that the map $ \pi_V^*(\pi_{\widetilde{\pro(F)}}^*)^{-1} $ is injective too. Now take $ U=\set{T_0\neq 0} $, then we have $ \restrict{\tilde q_1}{U}=-(x^*+t_1xy, t_2y^") $, with $ t_1=\frac{T_1}{T_0} $ and $ t_2=\frac{T_2}{T_0} $. Diagonalising $  \restrict{\tilde q_1}{U} $, we get:
			\begin{align*}
				\restrict{\tilde q_1}{U}&=-\left[ (x+\frac{t_1}{2}y)^2+(t_2-\frac{t_1^2}{4})y^2 \right]=\\
				&=-\left[ (x+\frac{t_1}{2}y)^2+\restrict{q_0}{U}y^2 \right]
			\end{align*}
			So $ \restrict{\tilde q_1}{U}=-1+\restrict{q_0}{U} $ and hence $ \partial(\restrict{\tilde q_1}{U})=1 $ since $ \partial(q_0)=1 $ on $ U $. But this implies that $ \pi_V^*(\pi_{\widetilde{\pro(F)}}^*)^{-1}(\partial_1(\tilde q_1)-1)=\partial(\restrict{\tilde q_1}{U})-1=0 $, and, being $ \pi_V^*(\pi_{\widetilde{\pro(F)}}^*)^{-1} $  injective, we deduce that $ \partial_1(\tilde q_1)=1 $ as we wanted to show.\\
			Now that we know that $ \partial_1(\tilde q_1)=1 $, by \eqref{ch2:_eq_last_relation_mult_thm}, we can replace $ q_1 $ in   $ (1+q_0)q_1 $ with $ \tilde q_1 $. Since on $ U $ we have $ \restrict{\tilde q_1}{U}=-1+\restrict{q_0}{U} $, we have that $ (1+\restrict{q_0}{U})\restrict{\tilde q_1}{U}=0 $. But this implies:
			\[ \pi_{V}^*(\pi_{\widetilde{\pro(F)}}^*)^{-1}((1+q_0)\tilde q_1)= (1+\restrict{q_0}{U})\restrict{\tilde q_1}{U}=0   \]
			and since $ \pi_{V}^*(\pi_{\widetilde{\pro(F)}}^*)^{-1} $ is an injective map, we get that $ (1+q_0)\tilde q_1=0 $ and we are done.

		\end{proof}

					Putting together all the result we got so far, we get:
					\begin{co}\label{ch2:_MEC_5.5_field}
						Let $ \mr{KW}^{\bullet}(\sk)\llbracket x_0,x_2\rrbracket $ be the graded algebra over $ \mr{KW}^{\bullet}(\sk) $, freely generated by $ x_0,x_2 $ in degrees $ deg(x_i)=i $.  Sending $ x_0 $ to $ q_0:=q_0^{\mr{KW}} $ and $ x_2 $ to $ p^*e $ defines a $ \mr{KW}(\sk) $-algebra isomorphism:
						\[ \psi_{\sk}: \quot{\mr{KW}^{\bullet}(\sk)\llbracket x_0, x_2\rrbracket }{\left( x_0^2-1, (1+x_0)x_2 \right)} \longrightarrow \mr{KW}^{\bullet}\left(\mc BN \right) \] 
						Moreover $ \mr{KW}^{\bullet}\left( \mc BN; \gamma_N \right) $ is the quotient of the  free $ \mr{KW}^{\bullet}\left( \mc BN \right) $-module $ \mr{KW}^{\bullet-2}({\mc BSL_2})\cdot e(\oocatname{T})\oplus q_1\cdot\mr{KW}^{\bullet}(\sk) $ modulo the relations $ (1+q_0)j^*e(\oocatname{T})=0 $, $ (1+q_0)q_1=0$.
					\end{co}
					
					\begin{proof}
						It is just a consequence of  \cref{ch2:_5.3_MEC} and \cref{ch2:_Lemma_5.4_MEC_field}. The relation $ x_0^2=1 $ comes from $ q_0^2=1 $ that holds for any quadratic form $ \langle u \rangle $ where $ u $ is a unit.
					\end{proof}
					
					\begin{co}\label{ch2:_MEC_5.5}
						Let $ S $ be a smooth $ \sk $-scheme. 	Let $ \mr{KW}^{\bullet}(S)\llbracket x_0,x_2\rrbracket $ be the graded algebra over $ \mr{KW}^{\bullet}(S) $, freely generated by $ x_0,x_2 $ in degrees $ deg(x_i)=i $.  Sending $ x_0 $ to $ q_0:=q_0^{\mr{KW}} $ and $ x_2 $ to $ p^*e $ defines a $ \mr{KW}(S) $-algebra isomorphism:
						\[ \psi_S: \quot{\mr{KW}^{\bullet}(S)\llbracket x_0, x_2\rrbracket }{\left( x_0^2-1, (1+x_0)x_2 \right)} \longrightarrow \mr{KW}^{\bullet}\left( \mc BN \right) \] 
						Moreover $ \mr{KW}^{\bullet}\left( \mc BN; \gamma_N \right) $ is the quotient of the  free $ \mr{KW}^{\bullet}\left( \mc BN \right) $-module $ \mr{KW}^{\bullet-2}({\mc BSL_2})\cdot e(\oocatname{T})\oplus q_1\cdot\mr{KW}^{\bullet}(S) $ modulo the relations $ (1+q_0)j^*e(\oocatname{T})=0 $, $ (1+q_0)q_1=0 $. 
					\end{co}
					\begin{proof}
						We will prove the corollary reducing to the case over a field, hence we will use subscripts to indicate over which base are we working with. By \cref{ch2:_5.3_MEC} for $ \mr A=\mr{KW} $ (and by \cite[Theorem 10]{Ananyevskiy_PhD_Thesis}), additively we have the following isomorphism:
						\[ \mr{KW}^{\bullet}(\mc BN_S)\simeq \mr{KW}^{\bullet}(\mc BSL_{2,S}) \oplus \mr{KW}^{\bullet}(S)\cdot q_{0,S} \]
						Therefore we have a natural map:
						\[ \varphi_S: {\mr{KW}^{\bullet}(S)\llbracket x_{0,S}, x_{2,S}\rrbracket } \longrightarrow \mr{KW}^{\bullet}\left( \mc BN_S \right)  \]
						\noindent sending $ x_2 $ to $ p^*e_S\in \mr{KW}(\mc BSL_{2,S}) $, i.e. the Euler class of the tautological bundle of $ \mc BSL_{2,S} $ , and sending $ x_0 $ to $ q_{0,S} $. By \cref{ch2:_MEC_5.5_field}, the natural map:
						\[ \varphi_{\sk}: {\mr{KW}^{\bullet}(\sk)\llbracket x_{0,\sk}, x_{2,\sk}\rrbracket } \longrightarrow \mr{KW}^{\bullet}\left( \mc BN_{\sk}\right)  \]
						\noindent passes to the quotient giving us the isomorphism:
						\[ \psi_{\sk}: \quot{\mr{KW}^{\bullet}(\sk)\llbracket x_{0,\sk}, x_{2,\sk}\rrbracket }{\left( x_{0,\sk}^2-1, (1+x_{0,\sk})x_{2,\sk} \right)} \stackrel{\sim}{\longrightarrow} \mr{KW}^{\bullet}\left(\mc BN_{\sk} \right)  \]
						The structure map $ \pi: S \rightarrow \spec(\sk) $ induces, via the pullback $ \pi_S^* $, a map:
						\[ \theta_S: {\mr{KW}^{\bullet}(S)\llbracket x_{0,S}, x_{2,S}\rrbracket } \longrightarrow {\mr{KW}^{\bullet}(\sk)\llbracket x_{0,\sk}, x_{2,\sk}\rrbracket }   \] 
						\noindent sending $ x_{0,S} $ to $ x_{0,\sk} $ and $ x_{2,S} $ to $ x_{2,\sk} $. We also get a map:
						\[ \theta_{\mc BN}: \mr{KW}^{\bullet}(\mc BN_S) \longrightarrow \mr{KW}^{\bullet}(\mc BN_{\sk})    \]
						\noindent induced again by $ \pi_S^* $ and sending $ p^*e_S $ to $ p^*e_{\sk} $ and $ q_{0,S} $ to $ q_{0,\sk} $. This implies that $ \varphi_S $ passes to the quotient, giving us a map:
						\[ \psi_S: \quot{\mr{KW}^{\bullet}(S)\llbracket x_{0,S}, x_{2,S}\rrbracket }{\left( x_{0,S}^2-1, (1+x_{0,S})x_{2,S} \right)} \longrightarrow \mr{KW}^{\bullet}\left( \mc BN \right) \]
						\noindent  Since $ \psi_S $ is an isomorphism on the underlying modules, we get that $ \psi_S $ is also an isomorphism of $ \mr{KW}^{\bullet}(S) $-algebras.\\
						
						The twisted case $ \mr{KW}^{\bullet}(\mc BN_S; \gamma_N) $ is completely analogous and left to the reader.
						
						%
						%
						%
						
					\end{proof}



\section{Atiyah-Bott and Virtual Localisation Formulas}

\subsection{Atiyah-Bott Localization}
We are now ready to extend the Atiyah-Bott localization theorem to any $ SL_{\eta} $-oriented ring spectrum $ \mr A $, closely following \cite{Levine_Atiyah-Bott}. Once again, we remind the reader that we will freely interchange the quotient stack $ \mc BN $ with its ($ \A^1 $-equivalent) ind-scheme approximation $ BN $.

\begin{disclaimer}
	To avoid burdening the notation too much, in the following chapter we will drop the $ G $ from the upper- and lower-superscripts of the equivariant operations on Borel-Moore homology and cohomology. 
\end{disclaimer}

\begin{defn}
	Let $ K \supset \sk $ be a field, $ \chi: \mb G_m \rightarrow \mb  G_m $ a $ \mb G_m $-character such that $ \chi(-Id)=1 $, and take $ \lambda\in K^{\times} $. Define the subgroup scheme $ \Lambda(\chi,\lambda)\sseq N_K $ as:
	\[ \Lambda(\chi,\lambda):=\chi^{-1}(1)   \amalg \chi^{-1}(\lambda^{-1})\cdot \sigma \]
\end{defn}
\begin{defn}\label{ch4:_def_Homogeneous_Cases}
	As done in \cite{Levine_Atiyah-Bott}, consider again $ K\supset \sk $ a field, a character $ \chi: \mb G_m \rightarrow \mb G_m $ and the following types of $ N $-homogenous spaces $ X $ over $ K $ (where $ K_X $ will denote the $ N $-invariants of $ \mo{X}(X) $):
	\begin{enumerate}
		\item [(a)]  $ X=\left( \bigslant{N}{\chi^{-1}(1)} \right)_K $;
		\item [(b)] Suppose $ \chi(-Id)=1 $ and let $ X=\bigslant{N_K}{\Lambda(\chi,\lambda)} $ for some $ \lambda \in K^{\times} $;
		\item [(c$ \pm $)] Let $ K'\supset K $ a degree two extension:
		\begin{enumerate}
			\item [(c+)] Suppose $ \chi(-Id)=1 $ and take some $ \lambda \in K^{\times} $. Let $ \tau $ be the conjugation of $ K' $ over $ K $. We can choose an isomorphism of $ \mb G_m $-homogeneous spaces $ \bigslant{\mb G_m}{\chi^{-1}(1)_{K'}}\simeq \mb G_{m,K'} $ inducing the isomorphism:
			\[ \left( \bigslant{N}{\chi^{-1}(1)} \right)_{K'}\simeq \mb G_{m,K'}\amalg  \sigma \cdot\mb G_{m,K'} \]
			\noindent Let $ \rho(\sigma^{i}\cdot x):=\sigma\cdot\sigma^{i} \cdot \frac{\lambda}{x} $, for $ i=0,1 $, be the $ K' $-automorphism of $ \left(\bigslant{N}{\chi^{-1}(1)} \right)_{K'}$. Composing with $ \tau $, we get the $ F $-automorphism, $ \tau\circ \rho $ which gives us a $ \bigslant{\Z}{2\Z} $-action on $\left( \bigslant{N}{\chi^{-1}(1)} \right)_{K'}$. The left $ N $-action on $ N $ descends to a left action on the quotient and hence we get an $ N $-homogeneous space:
			\[ X:=\bigslant{\left( \bigslant{N}{\chi^{-1}(1)} \right)_{K'}}{\langle \tau \circ \rho \rangle} \]
			\item [(c-)]  If $ \chi(-Id)=-1 $, we can take some $ \lambda_0\in K^{\times} $, then choose a generator $ \sqrt{a} $ for $ K' $ over $ K $ (with $ a \in K^{\times} $) and take $ \lambda:=\lambda_0\sqrt{a} $. With this $ \lambda $ we can define as before $\rho(\sigma^{i}\cdot x):=\sigma\cdot\sigma^{i}\cdot \frac{\lambda}{x}  $, for $ i=0,1 $, and consider the $ N $-homogeneous space:
			\[ X:=\bigslant{\left( \bigslant{N}{\chi^{-1}(1)} \right)_{K'}}{\langle \tau \circ \rho \rangle} \]
		\end{enumerate}
	\end{enumerate}
\end{defn}

These will be all the $ N $-homogeneous spaces we have to consider, using a special case of \cite[Lemma 7.4]{Levine_Atiyah-Bott}:

\begin{lemma}[{\cite[Lemma 7.4]{Levine_Atiyah-Bott}}]\label{ch4:_ABL_Lemma_7.4}
	Let $ X $ be an $ N $-homogeneous space. Suppose $ \mb G_m $ acts non-trivially on $ X $ and that $ X $ is  smooth over $ k_X:=H^0(X, \mo{X})^{N}$. Let $ Y:=\bigslant{X}{\mb G_m} $ with its induced structure as a $ \bigslant{\Z}{2\Z} $-homogeneous space over $ k_X $. Then:
	\begin{enumerate}
		\item As a $ Y $-scheme $ X \simeq \mb G_{m,Y} $;
		\item Fix a closed point $ y_0 \in Y $, let $ A\sseq\bigslant{\Z}{2\Z} $ be the isotropy group of $ y_0 $, let $ B $ be the kernel of the action map $ A \rightarrow \mr{Aut}_{k_X}(k_X(y_0)) $. Then there are three cases:
		\begin{enumerate}
			\item [\textit{(Case a)}] $ B=A=\set{0} $;
			\item [\textit{(Case b)}] $ A=B=\bigslant{\Z}{2\Z} $;
			\item [\textit{(Case c)}] $ A=\bigslant{\Z}{2\Z} $ and $ B=\set{0} $.
		\end{enumerate}
		\item In \textit{(Case a)}  $ X $ is a homogeneous space of type $ (a) $; in case $ (b) $ then $ X $ is of type $ (b) $; in case $ (c) $ then $ X $ is of type $ (c\pm) $, depending on whether $ \chi(-Id)=\pm 1 $, and $ \kappa(y_0) $ is the degree two extension of $ k_{X'} $ referred to in \cref{ch4:_def_Homogeneous_Cases}.
	\end{enumerate}
\end{lemma}

Via the $ N $-action, the map $ \pi_X: X \rightarrow \spec(k_X) \rightarrow \spec(\sk) $ induces a map:
\[ \pi_X^N: \mc X=\left[\quot{X}{N}\right] \rightarrow \mc BN  \]
For any $ \mr A \in \mr{SH}(\sk) $, this induces the pullback maps:
\[ (\pi_X^N)^*: \mr{A}^{\bullet}(BN) \longrightarrow \mr{A}^{\bullet}_N(X) \]
\[ (\pi_X^N)^*: \mr{A}^{\bullet}(BN; \gamma_N) \longrightarrow \mr{A}^{\bullet}_N(X; \pi_X^*\gamma_N) \]
\noindent where $ \gamma_N  $ is the generator of $ \mr{Pic}(BN) $.

\begin{notation}
	We will denote by $ e $ the Euler class of $ (\pi_X^N)^*\widetilde{\mo{}}^{+}(1)  $ living in $ \mr{A}^2_{N}(X)\simeq\mr{A}\left( X \times^{N} EN \right) $. We will omit the appropriate pullbacks  from the twists in our equivariant theories, if it is clear from the context what twist should actually be used; the same will happen for pullbacks on bundles and associated Euler classes.
\end{notation}

\begin{lemma}[{\cite[Lemma 7.6]{Levine_Atiyah-Bott}}]\label{ch4:_ABL_Lemma_7.6}
	Let $ X,\chi $ be as in \cref{ch4:_ABL_Lemma_7.4}. Suppose $ \chi(t)=t^{m} $ for some integer $ m\geq 1 $. Let $ \mr A \in \mr{SH}(\sk) $ be an $ SL_{\eta} $-oriented ring spectrum. Then we have:
	\begin{enumerate}
		\item [(a)] In case $ (a) $, $ e $ goes to zero in $ \mr{A}^{\bullet}_{N}(X) $.
		\item [(b)] In case $ (b) $, $ m $ is even and $ e(\widetilde{\mo{}}(m)) $ goes to zero in $ \mr{A}^{\bullet}_{N}(X; \pi_X^*\gamma_N ) $.
		\item [(c+)] In case $ (c+) $, i.e. $ m $ even, $ e(\widetilde{\mo{}}(m)) $ goes to zero in $ \mr{A}^{\bullet}_{N}(X;\pi_X^*\gamma_N) $.
		\item [(c-)] In case $ (c-) $, i.e. $ m $ odd, $ e(\widetilde{\mo{}}(2m)) $ goes to zero in $ \mr{A}^{\bullet}_{N}(X;\pi_X^*\gamma_N) $.
	\end{enumerate}
\end{lemma}
\begin{rmk}
	In the cases $ (b) $ and $ (c+) $, the pullback of  $ \gamma_N $ gets trivialised. Hence the Euler classes, considered in the previous lemma, live in $ \mr{A}^{\bullet}_{N}(X)  $.
\end{rmk}
\begin{proof}
	The same proof used in \cite[Lemma 7.6]{Levine_Atiyah-Bott} works verbatim for a generic $ \mr{A} $, since  only  the vanishing of Euler classes for odd rank vector bundles and multiplicativity of Euler classes was used there, and these properties still hold true for any $ SL_{\eta} $-oriented spectrum.
\end{proof}

\begin{notation}
	From now on, we will denote the Euler classes of the $ \widetilde{\mo{}}^{\pm}(m) $ as $ \tilde{e}^{\pm}(m):= e \left( \widetilde{\mo{}}^{\pm}(m)\right) $.
\end{notation}

\begin{rmk}
	Consider $ X \in \catname{Sch}_{\bigslant{}{\sk}}^{N} $, $ L \in \mr{Pic}(X) $ an $ N $-linearised line bundle, and let $ \mr A \in \mr{SH}(\sk) $ be an $ SL_{\eta} $-oriented ring spectrum.. Then products in  $ \mr A $-cohomology induce a structure of graded commutative ring on $ \mr{A}^{\bullet}_{N}\left( X \right) $ and a structure of $ \mr{A}^{\bullet}_{N}(X) $-module on $ \mr{A}^{\bullet}_{N}(X;L) $. Identifying $ \mr{A}^{\bullet}_{N}(X;L^{\otimes 2})\simeq \mr{A}^{\bullet}_{N}(X) $, we get a commutative graded product:
	\[ \mr{A}^{\bullet}_{N}(X;L)\otimes \mr{A}^{\bullet}_{N}(X;L)  \longrightarrow \mr{A}^{\bullet}_{N}(X)  \]
	For $ s \in \mr{A}^{2\bullet}_{N}(X) $, we say that an element $ x \in \mr{A}^{2\bullet}_{N}(X;L)[s^{-1}] $ is invertible if there exists $ y \in \mr{A}^{2\bullet}_{N}(X;L)[s^{-1}]  $ such that $ xy=1 \in \mr{A}^{2\bullet}_{N}(X)[s^{-1}]   $. A homogeneous element $ x \in \mr{A}^{2\bullet}_{N}(X;L) $ is invertible if and only if $ x^2\in \mr{A}^{2\bullet}_{N}(X) $ is invertible in the usual sense as an element of the commutative graded ring.
\end{rmk}

\begin{lemma}[Key Lemma]\label{ch4:_key_lemma_Bachmann_Hopkins}
	Let $ \mr A \in \mr{SH}(\sk) $ be an $ SL_{\eta} $-oriented ring spectrum. For any odd integer $ m $, the Euler class $ \tilde{e}^{+}(m) $ is invertible in:
	\[ \mr A^{\bullet}(\mc BN)\left[ (m\cdot e)^{-1} \right] \]
	For any even integer $ m=2n $, the Euler class $ \tilde{e}^{+}(m) $ is invertible in:
	\[ \mr A^{\bullet}(\mc BN; \gamma_N)\left[ (m\cdot e)^{-1} \right] \]
\end{lemma}
\begin{proof}
	For this proof we will make use of the ind-scheme $ BN $ approximating $ \mc BN $. To prove our claim, it is enough to consider the universal case $ \mr{MSL}_{\eta} $: the general case will follow considering the $ SL $-orientation map $ \varphi_{\eta}^{A}: \mr{MSL}_{\eta} \rightarrow A  $ associated to the $ SL_{\eta} $-oriented ring spectrum.\\
	Since we are working in $ \mr{SH}(\sk) $, by \cite[Theorem 8.8]{Bachmann_Hopkins}, we have:
	\[ \mr{MSL}_{\eta}(\sk)\simeq W(\sk)[y_2,y_4,y_6,\ldots] \]
	\noindent where $ y_{2i} $ has degree $ -4i $. But then by \cite[Theorem 10]{Ananyevskiy_PhD_Thesis} and our \cref{ch2:_5.3_MEC}, we get:
	\[ \mr{MSL}_{\eta}^{\bullet}(BN)\simeq \left(W(\sk)[y_2,y_4,\ldots]\right)\llbracket e \rrbracket \oplus q_0^{\mr{MSL}} \cdot W(\sk)[y_2,y_4,\ldots]   \]
	Let us suppose $ m $ is odd. Since $ \tilde{e}^{+}(m) $ has degree two, we have:
	\[ \tilde{e}^{+}(m) =\sum_{i=0}^{\infty} a_i \cdot e^{2i+1} \]
	\noindent where $ a_i =a_i(y_2,y_4,\ldots) \in \left(W(\sk)[y_2,y_4,\ldots]\right)^{-4i}$. Notice that $ a_0 \in W(\sk) $, since it has degree zero. The map $ \varphi_{\eta}^{H\mc W}: \mr{MSL}_{\eta} \rightarrow H\mc W $, induced by the $ SL $-orientation, is a map of ring spectra given by sending each $ y_i \mapsto 0 $. Thus we have:
	\[ \varphi_{\eta}^{H\mc W}(a_0\cdot e)=\overline{a_0}\cdot\varphi_{\eta}^{H\mc W}(e^{\mr{MSL}_{\eta}}) \]
	\noindent where $ \overline{a_0}:=\varphi_{\eta}^{H\mc W}(a_0)=a_0(0,0,\ldots) $. But comparing the Euler classes in $ \mr{MSL}_{\eta} $ and $ H\mc W $, we get:
	\[ \varphi_{\eta}^{H\mc W}(\tilde{e}^{+}(m))=\overline{a_0}\cdot e^{H\mc W}=\pm m \cdot e^{H\mc W} \]
	\noindent thus $ a_0(0,0,\ldots)=\overline{a_0}=\pm m $ by the computations made in \cite[Theorem 7.1]{Motivic_Euler_Char}. But $ a_0 $, as we said before, as degree zero, hence $ a_0=\pm m $. But this implies that, in $ \mr{MSL}_{\eta}^{\bullet}(BN)\left[  m^{-1} \right] $, we can write:
	\[ \tilde{e}^{+}(m)=\pm m \cdot e(1+\sum_{j=1} b_j \cdot e^{j}) \]
	\noindent and  $ (1+\sum_{j=1} b_j \cdot e^{j}) $ is already invertible in $ \mr{MSL}_{\eta}^{\bullet}(BN)\left[  m^{-1} \right]  $. Hence $ \tilde{e}^{+}(m) $ is invertible in $ \mr{MSL}_{\eta}^{\bullet}(BN)\left[ \left( m\cdot e \right)^{-1} \right]  $. The even case $ m=2n $ is completely analogous. We still have that:
	\[ \mr{MSL}_{\eta}^{\bullet}(BN; \gamma_N)\simeq e(\oocatname{T}) \cdot \left(W(\sk)[y_2,y_4,\ldots]\right)\llbracket e \rrbracket \oplus q_1^{\mr{MSL}} \cdot W(\sk)[y_2,y_4,\ldots]  \]
	\noindent  by \cref{ch2:_5.3_MEC} (and \cite[Theorem 10]{Ananyevskiy_PhD_Thesis}), where $ \oocatname{T} $ is the pullback of the tangent bundle of $ \pro(\sym^2(F))\times^{SL_2}ESL_2 $ relative to $ BSL_2 $. From  \cite[Lemma 6.1]{Motivic_Euler_Char}, we know $ \tilde{e}:=\tilde{e}^+(2)=e(\oocatname{T}) $.\\
	
	As before we have:
	\[ \tilde{e}^{+}(2n)= \tilde{e}\cdot \left( \sum_{i=0}^{\infty} c_i \cdot e^{2i} \right) \]
	\noindent where $ c_i=c_i(y_2, y_4,\ldots) \in\left( W(\sk)[y_2, y_4,\ldots]\right)^{-4i} $. Again we can make a comparison between the Euler classes of $ \mr{MSL}_{\eta} $ and $ H\mc W $, and we get that $ \varphi^{H\mc W}_{\eta}(c_0)=c_0(0,0,\ldots)=\pm n $ by \cite[Theorem 7.1]{Motivic_Euler_Char}. But then $ c_0=\pm n $, since it is a degree zero element. Thus, in $ \mr{MSL}_{\eta}^{\bullet}(BN)\left[  n^{-1} \right]  $,we can write:
	\[  \tilde{e}^{+}(2n)= \tilde{e}\cdot (\pm n \cdot e)\left( 1+\sum_{i=1}^{\infty} d_i \cdot e^{2i-1} \right) \]
	\noindent where the element $ \left( 1+\sum_{i=1}^{\infty} d_i \cdot e^{2i-1} \right)  $ is already invertible in $ \mr{MSL}_{\eta}^{\bullet}(BN)\left[  n^{-1} \right] $. If we consider $ \left( \tilde{e}^{+}(2n) \right)^2 \in \mr{MSL}_{\eta}^{4}(BN) $, we have:
	\[ \left(\tilde{e}^{+}(2n)\right)^2= \tilde{e}^2\cdot (\pm n \cdot e)^2\left( 1+\sum_{i=1}^{\infty} d_i \cdot e^{2i-1} \right)^2 \]
	Using the fact that $ (\tilde{e}^{H\mc W})^2=-4e \in H\mc W^{4}(BN) $ (cf. \cite[Theorem 7.1]{Motivic_Euler_Char}), again by looking at the formal power series of $ \tilde{e}^2 $ compared to $ (\tilde{e}^{H\mc W})^2 $, we get that:
	\[ \left(\tilde{e}^{+}(m)\right)^2= {e}^2\cdot (\pm m \cdot e)^2\left( 1+\ldots \right) \in \mr{MSL}^{4}_{\eta}(BN) \]
	So in the end it is enough to invert $ m\cdot e $ and we also get that $ \tilde{e}^{+}(m) $ is invertible in $ \mr{MSL}^{\bullet}_{\eta}(BN; \gamma_N)\left[ \left( m\cdot e \right)^{-1} \right] $.
\end{proof}

\begin{pr}\label{ch4:_ABL_7.7}
	Let $ X $ be an $ N $-homogeneous space, and let $ \mr A $ be an $ SL_{\eta} $-oriented ring spectrum. Then there exists an integer $ M>0 $ such that:
	\[ \mr{A}^{\bullet}_{N}(X)[(Me)^{-1}]\simeq 0 \]
	
\end{pr}
\begin{proof}
	Using \cref{ch4:_ABL_Lemma_7.6}, if $ X $ is of type $ (a) $, then we just invert $ e $ and $ M=1 $. For case $ (b) $ and $ (c\pm) $, again  by \cref{ch4:_ABL_Lemma_7.6} we have that $ \tilde{e}^{+}(m)  $ will go to zero for an appropriate $ m $ depending on the type of $ X $.  Then we conclude by \cref{ch4:_key_lemma_Bachmann_Hopkins}, noticing that if  the Euler class going to zero is some $ \tilde{e}^{+}(2k) $ living in the twisted theory, then its square $ \tilde{e}^{+}(2k)^2 \in \mr A^{\bullet}_N(X) $ will also go to zero.
	
	%
\end{proof}

\begin{rmk}
	Let $ T $ be the standard torus inside $ N $. Since $ \bigslant{N}{T}\simeq \bigslant{\Z}{2\Z} $, given $ X $ an $ N $-equivariant scheme over $ \sk $, we have that its $ T $-fixed points $ X^T\sseq X $ are the union of the $ 0 $-dimensional $ N $-orbits.
\end{rmk}
\begin{thm}[{Torus Localization, cf. \cite[Theorem 7.10]{Motivic_Euler_Char}}]\label{ch4:_ABL_7.10}
	Let $ X \in \catname{Sch}_{\bigslant{}{\sk}}^{N} $ and let $ \iota: X^T \into X $ the the associated closed immersion. Then for any $ SL_{\eta} $-oriented ring spectrum $ \mr A $, and for any $ L \in \mr{Pic}(X) $ with an $ N $-linearization, there exists an integer $ M>0 $ such that:
	\[ \iota_*: \mr{A}^{\mr{BM}}_{\bullet, N}\left( \bigslant{X^T}{\sk}; \iota^*L \right)\left[ \left( M\cdot e \right)^{-1} \right] \stackrel{}{\longrightarrow}  \mr{A}^{\mr{BM}}_{\bullet, N}\left( \bigslant{X}{\sk}; L \right) \left[ \left( M\cdot e \right)^{-1} \right]   \]
	\noindent is an isomorphism.
	
\end{thm}
\begin{proof}
	Remember that $ \mr{char}(\sk)>2 $.  If we invert the exponential characteristic $ p $ of $ \sk $, then taking the perfect closure of $ \sk^{perf} $ of $ \sk $ induces an isomorphism $ \mr{A}^{\mr{BM}}_{\bullet, N}\left( \bigslant{\cdot}{\sk}; \cdot \right) \stackrel{\sim}{\rightarrow} \mr{A}^{\mr{BM}}_{\bullet, N}\left( \bigslant{\cdot}{\sk^{perf}}; \cdot \right)  $ by \cite[2.1.5]{Elmanto_Khan_Perfection_MHT}, so we can assume $ p\ |\  M $ and $ \sk $ to be perfect.\\
	Using the localization sequence for equivariant Borel-Moore homology, we reduce ourselves to show that if $ X $ has no zero-dimensional orbits, that is, $ X^{T}=\emptyset $, then $ \mr{A}^{\mr{BM}}_{\bullet, N}\left( \bigslant{X}{\sk}; L \right)\simeq 0  $. By \cite[Proposition 1.1]{Levine_Atiyah-Bott}, there exists a finite stratification $ X=\amalg_{\alpha} X_{\alpha} $ such that  for each $ \alpha $ we have that $ \bigslant{X_{\alpha}}{N} $ exists as a quasi-projective $ \sk $-scheme and we can take $ \pi_{\alpha}: X_{\alpha}\rightarrow \bigslant{X_{\alpha}}{N} $ to be smooth. So applying again a localization sequence argument, we can reduce to the case of $ Z:=\bigslant{X}{N} $ is a quasi-projective $ \sk $-scheme and $ \pi: X \longrightarrow Z $ is smooth. Again by localization sequence argument, we can assume $ Z $ is integral and $ \pi $ is equi-dimensional. We have to show that if $ \pi $ has strictly positive relative dimension, then $ \mr{A}^{\mr{BM}}_{\bullet, N}\left( \bigslant{X}{\sk}; L \right)\simeq 0  $, so we can assume $ \pi $ has strictly positive dimension.\\
	Let us consider the finite level approximation for $ N $-equivariant Borel-Moore homology:
	\[ \mr{A}^{\mr{BM}}_{\bullet, N^{(m)}}\left( \bigslant{X}{\sk}; L \right):=\mr{A}^{\mr{BM}}_{\bullet}\left( \bigslant{X\times^{N} E_mN}{\sk}; L \right)   \]
	\noindent For each of these we have a Leray homological spectral sequence (cf. \cite[4.1.2]{Asok-Deglise-Nagel}):
	\[ E^{1}_{p,q}=\bigoplus_{z \in Z_{(q)}} \mr{A}^{\mr{BM}}_{p+q, N^{(m)}}\left( \bigslant{X_z}{\sk}; L_z \right) \Rightarrow \mr{A}^{\mr{BM}}_{p+q, N^{(m)}}\left( \bigslant{X}{\sk}; L \right)   \]
	\noindent But by purity for $ \sk \rightarrow \kappa(z) $, since $ z \in Z_{(q)} $, we have:
	\[ \mr{A}^{\mr{BM}}_{p, N^{(m)}}\left( \bigslant{X_z}{\kappa(z)}; L_z \right) \simeq \mr{A}^{\mr{BM}}_{p+q, N^{(m)}}\left( \bigslant{X_z}{\sk}; L \right)  \]
	\noindent and hence the spectral sequence becomes:
	\[ E^{1}_{p,q}=\bigoplus_{z \in Z_{(q)}} \mr{A}^{\mr{BM}}_{p, N^{(m)}}\left( \bigslant{X_z}{\kappa(z)}; L_z \right) \Rightarrow \mr{A}^{\mr{BM}}_{p+q, N^{(m)}}\left( \bigslant{X}{\sk}; L \right)   \]
	If $ \xi $ is a generic point of $ Z $, then there exists an open neighbourhood $ U_{\xi} $ of $ \xi $ such that each fiber $ X_z $ of $ \pi:X \rightarrow Z $ has the same same type, as homogeneous space for $ N $ over $ \kappa(z) $, as does $ X_{\xi} $. So after taking a further stratification, we can assume that all the fibers $ X_z $, for $ z \in Z $, have the same type. So by \cref{ch4:_ABL_7.7} there exists an $ M $ such that $ \mr{A}^{\mr{BM}}_{p, N^{(m)}}\left( \bigslant{X_z}{\kappa(z)}\right)\left[ \left( Me \right)^{-1} \right]\simeq 0 $, since $ \mr{A}^{\mr{BM}}_{p, N^{(m)}}\left( \bigslant{X_z}{\kappa(z)}\right) $ is a module over cohomology $ \mr{A}^{\bullet}_{N}\left(  X_z \right) \simeq \mr{A}^{\mr{BM}}_{\bullet, N}\left( \bigslant{X_z}{\kappa(z)}\right) $. The module map is obtained through the pullback induced by the natural map $ E_mN\rightarrow \spec(\sk) $, after identifying Borel-Moore Witt homology and Witt theory (up to some re-indexing) via purity, relative to $ X_z \rightarrow z $  (which is a smooth map, since it is a fiber of $ X \rightarrow Z $).  Notice that the integer $ M $ only depends on the type of $ X_z $, so it is the same integer for any $ z \in Z $. But again, since $ \mr{A}^{\mr{BM}}_{p, N^{(m)}}\left( \bigslant{X_z}{\kappa(z)}; L_z\right) $ is a also module over cohomology $ \mr{A}^{\bullet}_{N}\left(  X_z \right)  $, we have $ \forall \ z \in Z $:
	\[ \mr{A}^{\bullet}_{N}\left(  X_z \right) \left[ \left( Me\right)^{-1} \right]\simeq 0  \Rightarrow \mr{A}^{\mr{BM}}_{p, N^{(m)}}\left( \bigslant{X_z}{\kappa(z)}; L_z\right)\left[ \left( Me \right)^{-1} \right]\simeq 0  \]
	\noindent and hence the spectral sequence $ E_{p,q}^1 $ tells us that:
	\[ \mr{A}^{\mr{BM}}_{p+q, N^{(m)}}\left( \bigslant{X}{\sk}; L \right))\left[ \left( Me \right)^{-1} \right]\simeq 0 \ \forall \ p,q \ \forall \ m  \]
	In particular, this means that:
	\[ \set{\mr{A}^{\mr{BM}}_{p+q, N^{(m)}}\left( \bigslant{X}{\sk}; L \right))\left[ \left( Me \right)^{-1} \right]}_{m \in \N} \]
	\noindent satisfies the Mittag-Leffler condition. Therefore the limit of those groups actually computes $ \mr{A}^{\mr{BM}}_{p+q, N}\left( \bigslant{X}{\sk}; L \right))\left[ \left( Me \right)^{-1} \right] $. Hence we get:
	\[ \mr{A}^{\mr{BM}}_{p+q, N}\left( \bigslant{X}{\sk}; L \right))\left[ \left( Me \right)^{-1} \right] \simeq 0 \]
	\noindent as we wanted to prove.
	
\end{proof}

Before we finally get the Atiyah-Bott localization for an $ N $-action, we need to recall some more  theorems and definitions from \cite{Levine_Atiyah-Bott}:

\begin{defn}[{\cite[Def. 8.2]{Levine_Atiyah-Bott}}] \label{ch4:_ABL_def_8.2}
	Let $ \bar \sigma $ be the image of $ \sigma \in N $ in $ \bigslant{N}{T}\simeq \bigslant{\Z}{2\Z}=\langle \bar \sigma \rangle $.\\\begin{enumerate}
		\item Let us denote $ \abs{X}^{N} $ the union of the irreducible components $ Z \sseq X^{T} $ such that the generic point $ \xi_Z $ is fixed by $ \bar \sigma $, and let us denote $ {X}_{ind}^{T} $ the union of irreducible components $ W \sseq X^{T} $ such that $ \bar \sigma \cdot W \cap W=\emptyset $.
		\item We say that the $ N $-action is \textit{semi-strict} if $ X^{T}_{red}=\abs{X}^{N}\cup {X}^{T}_{ind} $.
		\item A semi-strict $ N $-action is said to be \textit{strict} if $ \abs{X}^{N}\cap {X}^{T}_{ind}=\emptyset  $ and we can decompose $ \abs{X}^{N} $ as disjoint union of two $ N $-stable closed subschemes:
		\[ \abs{X}^{N}=X^{N}\amalg X^{T}_{fr} \]
		\noindent where the $ \bigslant{N}{T} $-action on $ X^{N}_{fr} $ is free.
	\end{enumerate}
	
\end{defn}

\begin{notation}
	For any group $ G $, a character $ \chi $ will correspond to a line bundle over $ BG $ that we will denote as $ L_{\chi} $.
\end{notation}

\begin{pr}\label{ch3:_finite_Kunneth_BN}
	Let $ X \in \catname{Sch}_{\bigslant{}{\sk}}^{N} $ with a trivial $ N $-action, and let $ L $ be a line bundle on $ X $ also with trivial $ N $-action. Let $ \chi: N \rightarrow \mb G_m $ be a character with associated line bundle $ L_{\chi} $. Let $ \mr A $ be an $ SL_{\eta} $-oriented ring spectrum. Then, up to inverting the exponential characteristic of $ \sk $, we have an isomorphism:
	\[ \mr A^{-\bullet}(\mc BN; L_{\chi})\otimes \mr A^{\mr{BM}}_{\bullet}\left( \bigslant{X}{\sk}; L \right)\stackrel{\sim}{\longrightarrow} \mr A^{\mr{BM}}_{N,\bullet}(\bigslant{X}{\sk}; L_{\chi}\otimes L) \]
\end{pr}
\begin{proof}
	We can assume that $ \sk $ is perfect, if it is not we will invert its exponential characteristic and use \cite[Corollary 2.1.7]{Elmanto_Khan_Perfection_MHT}. Notice that since $ \mr A $ is $ SL $-oriented, if we twist or untwist $ \mr A^{\mr{BM}}_{N,\bullet}(\bigslant{X}{\sk}; L_{\chi}\otimes L)  $ by the determinant of the co-Lie algebra $ \mf n^{\vee} $ of $ N $, nothing changes (up to isomorphism). 
	To simplify the notation, we give the proof in the untwisted case (i.e. for trivial $ L $ and $ L_{\chi} $); the proof in the twisted case is the same and is left to the reader.
	Notice that $ \left[ \quot{X}{N} \right]\simeq X \times \mc BN $. Consider:
	\begin{center}
		\begin{tikzpicture}[baseline={(0,0)}, scale=2]
			\node (a) at (0,1) {$ X\times \mc BN$};
			\node (b) at (1, 1) {$ \mc BN$};
			\node (c)  at (0,0) {$ X $};
			\node (d) at (1,0) {$ \spec(\sk) $};
			\node (e) at (0.2,0.8) {$  \ulcorner $};
			\node (f) at (0.5,0.5) {$ \Delta $};

			\path[font=\scriptsize,>= angle 90]
			
			(a) edge [->] node [above ] {$  $} (b)
			(a) edge [->] node [left] {$  $} (c)
			(b) edge[->] node [right] {$  $} (d)
			(c) edge [->] node [below] {$  $} (d);
		\end{tikzpicture}
	\end{center}
	
	Using the composition product, under the identification $  \mr A^{-\bullet}(\mc BN)\simeq\mr A^{\mr{BM}}_{\bullet}\left( \bigslant{\mc BN}{\mc BN}\right) $, together with the base change $ \Delta^*: \mr A^{\mr{BM}}_{\bullet}\left( \bigslant{X}{\sk}, v \right)\rightarrow \mr A^{\mr{BM}}_{\bullet}\left( \bigslant{X\times \mc BN}{\mc BN} \right)  $, we get a map:
	\[ \psi:= (-\odot-)\circ (\Delta^*\times Id): \mr A^{\mr{BM}}_{\bullet}(\bigslant{X}{\sk})\otimes \mr A^{-\bullet}(\mc BN) \longrightarrow \mr A^{\mr{BM}}_{\bullet}(\bigslant{X\times \mc BN}{\mc BN})    \]
	Using purity (cf. \cite[Corollary 4.25]{ChoDA24}), the target of $ \psi $ is just $ \mr A^{\mr{BM}}_{N,\bullet}(\bigslant{X}{\sk}, -\mf n^{\vee}) $. \\


	Using the Leray spectral sequence (\cite[Theorem 4.1.2]{Asok-Deglise-Nagel}), we can show that $ \psi $ is an equivalence. Indeed, we have:
	
	\[ E_1^{p,q}=\bigoplus_{x \in X^{(p)}} \mr A_{p+q}^{\mr{BM}}(\bigslant{\kappa(x)}{\sk}) \allora \mr A_{p+q}^{\mr{BM}}(\bigslant{X}{\sk}) \]
	
	\[ 'E_1^{p,q}=\bigoplus_{x \in X^{(p)}} \mr A_{p+q}^{\mr{BM}}(\bigslant{\mc BN_{\kappa(x)}}{\mc BN}) \allora \mr A_{p+q}^{\mr{BM}}(\bigslant{X\times \mc BN}{\mc BN}) \]
	But $ \mr A^{\bullet}(\mc BN_{\sk}) $ is a free $ \mr A^{\bullet}(\sk) $-module by \cref{ch2:_MEC_5.5_field} and hence from $ E_1 $ we get:
	
	\[ ''E_1^{p,q}=\bigoplus_{x \in X^{(p)}} \mr A_{p+q}^{\mr{BM}}(\bigslant{\kappa(x)}{\sk})\otimes \mr A^{\bullet}(\mc BN) \allora \mr A_{p+q}^{\mr{BM}}(\bigslant{X}{\sk})\otimes \mr A^{\bullet}(\mc BN) \]
	
	But the map $ \psi $ at the level of spectral sequences induces a map $ ''E_1^{*,*}\rightarrow 'E_1^{*,*} $ that becomes an isomorphism. Let us go into more details on why that is true. Since we assumed that $ \sk $ is perfect, for any $ x\in X^{(p)} $ the extension $ \quot{\kappa(x)}{\sk} $ is separable, that is, $ \spec(\kappa(x)) $ is smooth over $ \sk $. Then we can use purity to see that $ \mr A^{\mr{BM}}(\bigslant{\kappa(x)}{\sk})\simeq \mr A(\kappa(x)) $ and $ \mr A^{\mr{BM}}(\bigslant{\mc BN_{\kappa(x)}}{\mc BN})\simeq \mr A(\mc BN_{\kappa(x)}) $. But by \cref{ch2:_5.3_MEC} for $ S=\spec(\kappa(x)) $ with $ x\in X^{(p)} $, we have that:
	\begin{equation}\label{ch4:_eq_Finite_Kunneth_BN}
		\mr A^{\bullet}(\kappa(x))\otimes_{\mr A^{\bullet}(\sk)} \mr A^{\bullet}(\mc BN)\simeq \mr A^{\bullet}(\mc BN_{\kappa(x)})
	\end{equation}
	
	To see why this equivalence it is true, \cref{ch2:_5.3_MEC} tells us that $ \mr A^{\bullet}(\mc BN)\simeq \mr A^{\bullet}(\mc BSL_2)\oplus \mr A^{\bullet}(\sk) $. Then to check that \eqref{ch4:_eq_Finite_Kunneth_BN} holds, it is equivalent to check that the same equality holds for $ \mc BSL_2 $ instead of $ \mc BN $. By \cite[Theorem 9]{Ananyevskiy_PhD_Thesis}, we know that we can compute $ \mr A^{\bullet}(\mc BSL_{2,k}) $ using the finite approximations $ B_mSL_{2,k} $ for any field $ k $. Since filtered limits commute with finite colimits, we have:
	\[ \mr A^{\bullet}(\kappa(x))\otimes_{\mr A^{\bullet}(\sk)} \left( \lim_m \mr A^{\bullet}(B_mSL_2) \right)\simeq \lim_m \mr A^{\bullet}(\kappa(x))\otimes_{\mr A^{\bullet}(\sk)}  \mr A^{\bullet}(B_mSL_2) \]
	\noindent Clearly the right hand side is just $ \mr A^{\bullet}(\mc BSL_{2,\kappa(x)}) $ and we have our claim.\\
	
	Therefore $ \psi $ induces an isomorphism on spectral sequences and thus it is an isomorphism itself and we are done.

\end{proof}

\begin{pr}\label{ch4:_ABL_8.5}
	Let $ X \in \catname{Sch}_{\bigslant{}{\sk}} $ with an $ N $-action, let $ L \in \mr{Pic}(X) $ an $ N $-linearised line bundle, and  let $ \mr A\in \mr{SH}(\sk) $ be an $ SL_{\eta} $-oriented ring spectrum.
	\begin{enumerate}
		\item Let $ \iota_0: X^{N} \into X $ the closed immersion. For each connected component $ X_i^{N} $ of $ X^{N} $ there is a line bundle $ L_i $ on $ X_i^{N} $ with trivial $ N $-action and a character $ \chi_i $ such that $ \restrict{\iota_0^* L}{X_i^{N}}\simeq L_i \otimes L_{\chi_i} $ as $ N $-linearised line bundle. Moreover:
		\[ \mr{A}^{\mr{BM}}_{N,k}\left( X^{N}; L \right)  \simeq \bigoplus_{p+q=k} \mr{A}^{\mr{BM}}_{p}(X_i^{N}; L_i) \otimes_{W(\sk)} \mr{A}^{-q}\left( BN; L_{\chi_i} \right)  \]
		\item Let $ \iota_{ind}: X^{T}_{ind} \into X $ be the inclusion, then:
		\[ \mr{A}^{\mr{BM}}_{N}\left( \bigslant{X_{ind}^{T}}{\sk}; \iota_{ind}^*L \right) \left[ e^{-1} \right] =0 \]
		\noindent where $ e $ is the Euler class associated to (the pullback of) $ \widetilde{\mo{}}^{+}(1) $.
	\end{enumerate}
\end{pr}

\begin{proof}
	For the first assertion we may assume $ X^{N} $ is connected. The line bundle $ \iota_0^*L $ lives over a space with trivial $ N $-action, so $ N $ will act with a character $ \chi $ on $ \iota_0^*L $ (corresponding to a line bundle $ L_{\chi} $) so that we will have an isomorphism of $ N $-linearised line bundles $ \iota_0^*L\simeq L_0 \otimes L_{\chi} $ for some line bundle $ L_0 $ having the trivial $ N $-action. Then the isomorphism of the first assertion in the proposition will follow from \cref{ch3:_finite_Kunneth_BN}.\\
	For the second assertion, consider $ C_1,\ldots,C_{2r} $ the irreducible components of $ X_{ind}^{T} $ with $ \bar \sigma C_{2i-1}=C_{2i} $. We will proceed by induction on $ r $. Let us start with $ r=1 $, as an $ N $-scheme $ X^{T}_{ind}=C_1 \amalg C_2\simeq \left( \bigslant{N}{T} \right) \times C_1 $ with the trivial action on $ C_1 $. Hence:
	\[ \left[ \quot{X^{T}_{ind}\ \!}{\ \!N} \right]\simeq C_1\times \mc BT \]
	\noindent and we get:
	\[ \mr{A}^{\mr{BM}}_{\bullet, N}\left( \bigslant{X^{T}_{ind}}{\sk}; \iota_{ind}^*L \right)\simeq \mr{A}^{\mr{BM}}_{\bullet}\left( \bigslant{C_1}{\sk}; L_1 \right) {\otimes}_{W(\sk)} \mr{A}^{-\bullet}\left( BT; L_{\chi_{C_1}} \right)  \]
	\noindent where $ L_1 $ is a line bundle over $ C_1 $ and $ L_{\chi_{C_1}} $ is the bundle associated to a character $ \chi_{C_1} $ of $ T $. Let  $ q: B\mb G_m \rightarrow BN  $, we have $ q^*\widetilde{\mo{}}^{+}(1)\simeq \mo{}(1)\oplus \mo{}(-1)  $, and so $ e=q^*e(\widetilde{\mo{}}^{+}(1))\simeq e\left( q^*\widetilde{\mo{}}^{+}(1)\right)=e\left( \mo{}(1)\oplus \mo{}(-1) \right)=0 $. But $\mr{A}^{-\bullet}\left( BT; L_{\chi_{C_1}} \right)   $ is a $ \mr{A}^{-\bullet}\left( BT\right)   $-module and $ \mr{A}^{-\bullet}\left( BT\right) \left[ e^{-1} \right]  =0 $, hence $ \mr{A}^{-\bullet}\left( BT; L_{\chi_{C_1}} \right)[e^{-1}]=0  $ and we are done for the initial inductive step. 
	
	For $ r>1 $, consider $ X^{T}_{ind}=C \cup C' $ where $ C:= C_1 \amalg C_2 $ and $ C':=C_3 \cup \ldots \cup C_{2r} $. By induction:
	\[ \mr{A}^{\mr{BM}}_{N}\left( \bigslant{C}{\sk}; \iota_{ind}^*L \right)\left[ e^{-1} \right] =\mr{A}^{\mr{BM}}_{N}\left( \bigslant{C'}{\sk}; \iota_{ind}^*L \right)\left[ e^{-1} \right] =0 \]
	Moreover $ C\cap C'=\left( C_1\cap C'  \right)\amalg \left( C_2 \cap C' \right)\simeq \left( C_1\cap C' \right)\amalg \bar \sigma \cdot \left( C_1 \cap C' \right) $, hence $ \left( C \cap C' \right)\times^{N} EN \simeq \left( C_1\cap C' \right)\times BT $ similar to the computation made for $ r=1 $. But using the same argument as before, we get in this way that $ \mr{A}^{\mr{BM}}_{N}\left( \bigslant{C\cap C'}{\sk}; \iota_{ind}^*L \right)\left[ e^{-1} \right] =0  $. Then using the localization sequence, it easy to see that:
	\[ \mr{A}^{\mr{BM}}_{N}\left( \bigslant{X^{T}_{ind}}{\sk}; \iota_{ind}^*L \right)\left[ e^{-1} \right] =0 \]
\end{proof}

\begin{thm}[Atiyah-Bott Localization for N-action]\label{ch4:_ABL_8.6}
	Let $ X \in \catname{Sch}_{\bigslant{}{\sk}}^{N} $ be a scheme with an $ N $-action and let $ \mr A \in \mr{SH}(\sk) $ be an $ SL_{\eta} $-oriented spectrum. Let $ \iota: \abs{X} ^{N} \into X  $ be the closed immersion. Let $ L \in \mr{Pic}(X) $ be an $ N $-linearised line bundle. Suppose the $ N $-action  is semi-strict. Then there is a non-zero integer $ M $ such that:
	%
	\[ \iota_*: \mr{A}^{\mr{BM}}_{\bullet, N}\left( \abs{X}^{N}; \iota^*L \right)\left[ \left( M \cdot e \right)^{-1} \right] \stackrel{}{\longrightarrow} \mr{A}^{\mr{BM}}_{\bullet, N}\left( X ; L \right)\left[ \left( M \cdot e \right)^{-1} \right]   \]
	\noindent is an isomorphism. 
\end{thm}
\begin{proof}
	Since the action is semi-strict, we can consider the closed immersion $ X^{T}\simeq X^{T}_{ind}\cup \abs{X}^{N} \into X $, that has $ X \setminus X^{T} $ as open complement. By \cref{ch4:_ABL_8.5}, we have $ \mr{A}^{\mr{BM}}_{N}\left( \bigslant{X^{T}_{ind}}{\sk}; \iota_{ind}^*L \right)\left[ e^{-1} \right] =0 $. Applying again \cref{ch4:_ABL_8.5}, this time to the scheme $ \abs{X}^{N} $, we get:
	\[ \mr{A}^{\mr{BM}}_{N}\left( \bigslant{X^{T}_{ind}\cap \abs{X}^{N}}{\sk}; \iota_{ind}^*L \right)\left[ e^{-1} \right] =0 \]
	\noindent since $ \left( \abs{X}^{N} \right)^{T}_{ind}\simeq X^{T}_{ind}\cap \abs{X}^{N} $. But using a localization sequence, this implies that:
	\[ \mr{A}^{\mr{BM}}_{N}\left( \bigslant{X^{T}_{ind} \setminus \abs{X}^{N}}{\sk}; \iota_{ind}^*L \right)\left[ e^{-1} \right] =0 \]
	\noindent as well. So we have that the inclusion $ \abs{X}^{N}\into X^{T} $ induces an isomorphism on $ N $-equivariant Borel-Moore homologies, and so we get the final claim using \cref{ch4:_ABL_7.10}.
\end{proof}

\subsection{Bott Residue Formula}

Consider now $ \iota: Y \into X $ a regular embedding in $ \catname{Sch}_{\bigslant{}{\sk}}^{N} $. Let $ N_{\bigslant{Y}{X}} $  be the normal bundle associated to $ \iota $, it has a natural $ N $-linearisation, so we can consider $ e_N(N_{\bigslant{Y}{X}}) \in \mr{A}^{\mr{BM}}_{r,N}\left( \bigslant{Y}{\sk}; \det\!{}^{-1}\!\!\,\left( N_{\bigslant{Y}{X}} \right) \right) $, where $ r $ is the rank of $ N_{\bigslant{Y}{X}} $.
\begin{lemma}\label{ch4:_ABL_9.1}
	Let $ L \in \mr{Pic}(X) $ and  let $ \mr A\in \mr{SH}(\sk) $ be an $ SL_{\eta} $-oriented ring spectrum. Then:
	\[ \iota^!\iota_*:  \mr{A}^{\mr{BM}}_{\bullet,N}\left( \bigslant{Y}{\sk};  \iota^* L \right)  \longrightarrow \mr{A}^{\mr{BM}}_{\bullet-r,N}\left( \bigslant{Y}{\sk}; \iota^*L \otimes\det\!{}^{-1}\!\!\,\left( N_{\bigslant{Y}{X}} \right) \right)   \]
	\noindent is the cup product with  $ e_N\left(N_{\bigslant{Y}{X}}\right) $.
\end{lemma}
\begin{proof}
	The claim follows using the same arguments as in \cite[Corollary 4.2.3]{DJK}.
\end{proof}

Following \cite[Lemma 9.3, Construction 2.7]{Levine_Atiyah-Bott}, we have that the Euler class of an $ N $-linearised vector bundle $ V $ on some connected $ Y\in \catname{Sch}_{\bigslant{}{\sk}}^{N} $ gets inverted once we invert some Euler classes coming from representations of $ N $ over $ \sk $. Let us briefly recall such construction  (for details refer to \textit{loc. cit.}). Let $ \mc V $ be an $ N $-linearised locally free sheaf on some connected scheme $ Y\in \catname{Sch}_{\bigslant{}{\sk}}^{N} $.

\begin{enumerate}
	\item [(\textit{Case 1})] Suppose $ N $ acts trivially on $ Y $. Then for each point $ y \in Y $, we can find $ N $-stable open neighbourhoods $ j_{U_y}:U_y\into Y $ of $ y $ such that our locally free sheaf trivialises as $ \psi_y: j_{U_y}^*\mc V \stackrel{\sim}{\rightarrow} \mo{U_y }\otimes_{\sk} V(f) $ for some $ \sk $-representation $ f $ of $ N $. Taking the decomposition of $ V(f) $ into isotypical components, we get the corresponding decomposition of $ j_{U_y}^*\mc V $. We can actually find a global decomposition of $ \mc V $ into isotypical components, indexed by $ \sk $-irreducible representations of $ N $:
	\[ \mc V=\bigoplus_{\phi} \mc V_{\phi} \]
	\noindent in such a way that on each trivialising open subset $ j_U: U \into Y $ we have:
	\[ j_U^*\mc V_{\phi}= \mo{U}\otimes V(\phi)^{n_{\phi}} \]
	\noindent where $ V(\phi) $ is a $ \sk $-irreducible representation of $ N $. The $ \bigoplus_{\phi} V(\phi)^{n_{\phi}} $ turns out to be completely determined by $ \mc V $, up to isomorphism. We denote its isomorphism class as $ [\mc V^{gen}] $.
	
	\item [\textit{(Case 2)}] Suppose $ T=\mb G_m\sseq N $ acts trivially on $ Y $. We can decompose $ \mc V $ into weight spaces $ \mc V=\bigoplus_{m} \mc V_m $, and let $ \mc V^{\mf m}:=\bigoplus_{m\neq 0} \mc V_m $. Then for any $ N $-trivialised open $ j_U: U \into Y $, we have:
	\[ j_U^*\mc V\simeq \mc V_0 \oplus \bigoplus_{m>0} \mo{U}\otimes^{\sigma,\tau} V(\rho_m)^{n_m} \]
	\noindent where $ \rho_i $'s are the rank two $ N $-representations introduced in Chapter 3, and $ \cdot \otimes^{\sigma,\tau} V(\rho_i) $ denotes the $ \mo{U} $-semi-linear extension of the representations in the sense of \cite[Def.2.5]{Levine_Atiyah-Bott}. Then isomorphism class of the representation $ \bigoplus_{m>0} V(\rho_m)^{n_m} $ is uniquely determined by $ \mc V $ and hence we denote this class by $ [\mc V^{gen}] $.
	
	\item [\textit{(Case 3)}]  Suppose $ \mb G_m=T\sseq N $ acts trivially on $ Y $ and that $ q: Y \rightarrow \bigslant{Y}{N}\simeq \bigslant{Y}{\langle \bar \sigma \rangle} $ is a degree 2 étale cover, where $ \bar \sigma $ is the image of $ \sigma $ in $ \bigslant{N}{T}\simeq \bigslant{\Z}{2\Z}\simeq \langle \bar{\sigma} \rangle $. Then for any $ N $-trivialised open $ j_U: U \into Y $, we have:
	\[ j_U^*\mc V\simeq \mo{U}\otimes^{\sigma,\tau} V(\rho_0)^{n_0} \oplus \bigoplus_{m>0} \mo{U}\otimes^{\sigma,\tau} V(\rho_m)^{n_m} \]
	\noindent where the notation is the same as in \textit{Case 2}. This time the isomorphism class of $ V(\rho_0)^{n_0} \oplus \bigoplus_{m>0} V(\rho_m)^{n_m} $ is uniquely determined by $ \mc V $, and we denote this class by $ [\mc V^{gen}] $.
\end{enumerate}

\begin{defn}[{\cite[Def. 4.7]{Levine_Atiyah-Bott}}]
	Let $ V $ an $ N $-linearised vector bundle on a connected $ Y\in \catname{Sch}_{\bigslant{}{\sk}}^{N} $. Suppose we are in one of the three cases when $ [\mc V^{gen}] $ is defined. Choose a representative $ V^{gen} \in [\mc V^{gen}] $ and suppose it has even rank $ 2r $ (this is always true in \textit{Case 2}).  Let $ \mr A\in \mr{SH}(\sk) $ be an $ SL_{\eta} $-oriented ring spectrum. Then we have:
	\[ e_N(V^{gen}) \in \mr{A}^{2r}\left( BN; \det\!{}^{-1}\!\!\,(V^{gen}) \right) \]
	\[ e_N(V^{gen})^2 \in \mr{A}^{4r}\left( BN \right) \]
	We then define the \textit{generic Euler class} as the subset:
	\[ [e_N(V^{gen})]:=\set{u\cdot e_N(V^{gen})^2\st{0}{3}  \ u \in\left( \mr{A}^0(\sk) \right)^{\times} }\sseq \mr{A}^{4r}(BN) \]
	\noindent and this depends by construction only on the isomorphism class of $ V $ as an $ N $-linearised bundle. The localization:
	\[ \mr{A}^{\bullet}_N\left( Y \right)\left[  [e_N(V^{gen})]^{-1}  \right] \]
	\noindent will denote the localization with respect to any element $ y \in [e_N(V^{gen})] $ seen as an element in $ \mr{A}_N(Y) $ through the $ \mr{A}(BN) $-module map. If the representative $ V^{gen} $ has a trivialization of its determinant, then it will give us an actual class $ e_N(V^{gen}) \in \mr{A}^{2r}(BN) $ and then it becomes a localization by $ e_N(V^{gen}) $ in the usual sense.
\end{defn}

\begin{lemma}[{\cite[Lemma 9.3]{Levine_Atiyah-Bott}}]\label{ch4:_ABL_9.3}
	Let $ V $ an $ N $-linearised vector bundle on a connected scheme $ Y \in \catname{Sch}_{\bigslant{}{\sk}}^{N} $ of rank $ 2r $. Let $ \mr A\in \mr{SH}(\sk) $ be an $ SL_{\eta} $-oriented ring spectrum. Let us suppose that assumptions of \cite[Construction 2.7]{Levine_Atiyah-Bott} are satisfied, so we are in Case 1,2,3 in \textit{loc.cit.}  and hence we get a generic Euler class $ [e_{N}^{gen}(V) ]\sseq \mr{A}^{4r}(BN) $. Choose an element $ e_N(V^{gen}) \in [e_N(V^{gen})] $. Then $ e_{N}(V) \in \mr{A}^{2r}_{N}(Y; \det\!{}^{-1}\!\!\,\,(V)) $ is invertible in $ \mr{A}^{2\bullet}_{N}(Y;\det\!{}^{-1}\!\!\,\,(V))\left[ \left( e_{N}^{gen}(V) \right)^{-1} \right] $.
\end{lemma}
\begin{proof}
	We need to show that $ e_N(V) $ is invertible in $ \mr{A}^{2\bullet}_{N}(Y;\det\!{}^{-1}\!\!\,\,(V))\left[ \left( e_{N}^{gen}(V) \right)^{-1} \right]  $,  but this is equivalent to show that multiplication by $ e_N(V) $:
	\[ \cdot e_V(N): \mr{A}^{2\bullet}_{N}(Y;\det\!{}^{-1}\!\!\,\,(V))\left[ \left( e_{N}^{gen}(V) \right)^{-1} \right]  \longrightarrow \mr{A}^{2\bullet+2r}_{N}(Y;\det\!{}^{-1}\!\!\,\,(V))\left[ \left( e_{N}^{gen}(V) \right)^{-1} \right]  \]
	\noindent is an isomorphism. By assumptions, we can find an open cover made by $ N $-stable opens $ U_i \sseq Y $ such that $ V $ trivialises over each $ U_i $. Then by a Mayer-Vietoris argument, we can reduce ourselves to prove our claim on all the intersections of our $ U_i $'s. But there the claim is obvious, since on each $ U_i $ we have an isomorphism of vector bundles (with an $ N $-action) $ \restrict{V}{U_i}\simeq \pi_Y^*V^{gen} $, where $ \pi_Y: Y\rightarrow \spec(\sk) $ is the structure map.

\end{proof}

\begin{thm}[Bott Residue Formula]\label{ch4:_ABL_9.5}
	Let $ X \in \catname{Sch}_{\bigslant{}{\sk}}^{N} $, $ L \in \mr{Pic}(X) $ an $ N $-linearised line bundle, and $ \mr A \in \mr{SH}(\sk) $ an $ SL_{\eta} $-oriented ring spectrum. Let us suppose that each connected component $ \iota_j: \abs{X}^{N}_{j} \into X $ of $ \abs{X}^{N} $  is a regular embedding. Moreover for each normal bundle $ N_j $ associated to $ \iota_j $, assume that the hypothesis of Case 1,2,3 in \cite[Construction 2.7]{Levine_Atiyah-Bott} are satisfied for $ V=N_j $. Denote by $ e_N^{gen}(N_{fix}) $ the products of $ e_N^{gen}(N_j) $'s. Finally assume that the $ N $-action on $ X $ is semi-strict.\\
	Denote $ P:=p\cdot M\cdot e$, where $ p $ is the exponential characteristic of the ground field $ \sk $ and $ M $ is the same integer as in \cref{ch4:_ABL_7.10}.\\
	Under the identification induced by the decomposition $ \abs{X}^{N}=\bigcup_j \abs{X}^{N}_{j} $ in its connected components:
	\[ \mr{A}^{\mr{BM}}_{N}\left( \bigslant{\abs{X}^{N}}{\sk}; \iota^*L \right)\simeq \prod_{j}  \mr{A}^{\mr{BM}}_{N}\left( \bigslant{\abs{X}^{N}_{j}}{\sk}; \iota^*_jL \right)  \]
	\noindent the inverse of the isomorphism:
	\[ \iota_*:  \mr{A}^{\mr{BM}, \bullet}_{N}\left( \bigslant{\abs{X}^{N}}{\sk}; \iota^*L \right)\left[ \left( P\cdot e_N^{gen}(N_{fix}) \right)^{-1} \right] \longrightarrow  \mr{A}^{\mr{BM}, \bullet}_{N}\left( \bigslant{X}{\sk}; L \right) \left[ \left( P\cdot e_N^{gen}(N_{fix}) \right)^{-1} \right]  \]
	\noindent is given by:
	\[ x \mapsto \prod_j \iota^!_j(x)\cap e_N(N_j)^{-1} \]
\end{thm}
\begin{proof}
	After inverting $ P $, we can use \cref{ch4:_ABL_8.6} and hence see $ e_N^{gen}(N_{fix}) $ and $ e_N^{gen}(N_j) $ as  elements in $ \mr{A}^{\mr{BM}}_{N}\left( \bigslant{X}{\sk}; L \right)  $. Inverting said elements will invert also $ e_N(N_j) $ by \cref{ch4:_ABL_9.3}. By \cref{ch4:_ABL_8.6} we can find elements:
	\[ y_j\in  \mr{A}^{\mr{BM}}_{N}\left( \bigslant{\abs{X}^{N}_{j}}{\sk}; \iota^*_jL \right)\left[ \left( P\cdot e_N^{gen}(N_{j}) \right)^{-1} \right]  \]
	\noindent such that:
	\[  x=\sum_j \iota_j{}_* y_j \]
	\noindent for any $ x \in   \mr{A}^{\mr{BM}}_{N}\left( \bigslant{X}{\sk}; L \right) \left[ \left( P\cdot e_N^{gen}(N_{fix}) \right)^{-1} \right]  $.  Since by \cref{ch4:_ABL_9.1} we have:
	\[ \iota_j^!(x)=\iota_j^!\iota_j{}_*y_j=y_j\cap e_N(N_j) \]
	\noindent the map we defined sending $ x $ to $\prod_j  \iota_j^!(x)\cap e_N(N_j)^{-1} $ will be indeed the inverse to $ \iota_* $.
\end{proof}
\begin{rmk}\label{ch4:_ABL_9.6}
	If $ X \in \catname{Sch}_{\bigslant{}{\sk}}^{N} $ is smooth, then $ X^{T} $ is smooth too, the action will be semi-strict, and each $ \iota_j: \abs{X}^{N}_j \into X $ will be a regular embedding. Moreover the respective normal bundles will be of the form $ N_{j}=N_j^{\mf m} $, so we are in Case 2 of \cite[Construction 2.7]{Levine_Atiyah-Bott} and we can indeed apply the previous theorem.
\end{rmk}

\subsection{Virtual Localization Formula}
We now have all the formal properties we need to prove the virtual localization formula of virtual fundamental classes of $ N $-equivariant schemes following \cite{VLF_Levine} and hence \cite{Graber-Pandharipande}.\\
For this section, let us fix $ X \in \catname{Sch}\bigslant{}{\sk}^{N} $ with a closed immersion $ \iota: X \into Y $ in $ \catname{Sch}_{\bigslant{}{\sk}}^{N} $, where $ Y $ is a smooth $ \sk $-scheme. Let us also suppose $ X $ is equipped with an $ N $-equivariant perfect obstruction theory represented by a two term complex $ \mc E_{\bullet}:=(\mc E_1 \rightarrow \mc E_0) $, of $ N $-linearised locally free sheaves, together with an $ N $-equivariant map $ \varphi_{\bullet}: \mc E_{\bullet} \longrightarrow \mb L_{\bigslant{X}{\sk}} $.  \\
Now consider the maximal subtorus $ T:=\mb G_m \sseq N $. We have the fixed $ T $-schemes $ X^{T}\sseq Y^{T} $, where we give $ X^{T} $ the scheme structure $ X^{T}:=X \cap Y^{T} $. Notice that $ Y^{T} $ is smooth (this is  true in much greater generality, for example see \cite[Theorem 4.3.6]{romagny2022algebraicity}); consider its connected  components $ Y_1,\ldots, Y_s $ and inclusion maps $ \iota_i^{Y}: Y_i \into Y $. Let us denote $ \iota_j: X_j:=Y_j \cap X \into X $ so that $ X^{T}=\coprod_j X_j $.\\
Let $ \mc F $ be a $ T $-linearised coherent sheaf on $ X_j $. The $ T $-action on the $ X_j $ is trivial, so we can decompose $ \mc F $ into its weight spaces for the $ T $-action:
\[ \mc F \simeq \bigoplus_{m\in \Z} \mc F_m \]
\noindent If $ \mc F $ is locally free, then so are the $ \mc F_m $'s. 
\begin{notation}
	Let $ \mc F $ be a $ T $-linearised coherent sheaf on some scheme $ Z $ with a trivial $ T $-action. Let $ \mc F\simeq \bigoplus_{m \in \Z}\mc F_m $ be its decomposition in weight spaces, then we will denote:
	\[ \mc F^{\mf m}:=\bigoplus_{m\in \Z\setminus \set{0}} \mc F_m \]
	\noindent the \textit{ moving part}, and by:
	\[ \mc F^{\mf f}:= \mc F_0 \]
	\noindent the \textit{fixed part} of $ \mc F $.
\end{notation}

In the situation described above where we have $ \iota_j: X_j \into X $ and a perfect obstruction theory  $ \varphi_{\bullet}: \mc E \rightarrow \mb L_{\bigslant{X}{\sk}} $ on $ X $, then $ \varphi_{\bullet} $ induces maps:
\[ \varphi_{\bullet}^{(j)}: \iota_j^*\mc E_{\bullet}^{\mf f} \longrightarrow \mb L_{\bigslant{X_j}{\sk}} \]
\noindent that by \cite[Proposition 1]{Graber-Pandharipande} are perfect obstruction theories for the $ X_j $'s.\\

\begin{defn}
	The virtual conormal sheaf of each $ X_j $ is defined to be the perfect complex $ \mc N_j^{vir}:=\iota_j^*\mc E_{\bullet}^{\mf m} $.
\end{defn}

By \cite[Lemma 6.2]{VLF_Levine} we have:
\begin{lemma}[{\cite[Lemma 6.2]{VLF_Levine}}] \label{ch4:_VLF_6.2}
	For each $ j $, the perfect obstruction theory $ \varphi_{\bullet}^{(j)}: \iota_j^{*}\mc E_{\bullet}^{\mf f} \rightarrow \mb L_{\bigslant{X_j}{\sk}} $ and $\mc N_j^{vir}$ have a natural $ N $-linearisation.
\end{lemma}

\begin{rmk}
	If $ Y $ is smooth then the $ N $-action is semi-strict \cite[Remark 9.6]{Levine_Atiyah-Bott}. By \cite[Remark 6.4]{VLF_Levine}, if the $ N $-action on $ Y $ is strict, then the $ N $ action on $ X $ will also be strict.
\end{rmk}

We will assume the $ N $-action on $ X $ is strict.

\begin{defn}
	By conventions set before $ X_j:=X\cap Y_j $ with $ Y_j $ connected components of $ Y^{T} $. Let us denote:
	\[ \abs{X}^{N}_j:=\abs{X}^{N}\cap X_j \]
	\noindent and thus we have:
	\[ \abs{X}^{N}\simeq \coprod_j \abs{X}^{N}_j \]
	\noindent where each $ \abs{X}^{N}_j $ has its own perfect obstruction theory $ \varphi_{\bullet}^{(j)} $ with associated virtual normal cone given by $ \mc  N_j^{vir} $. Each $ \abs{X}^{N}_j $ will decompose in connected component that we will denote as $ \abs{X}^{N}_{j,k} $
\end{defn}

\begin{rmk}
	In the case of a strict $ N $-action, consider an $ N $-linearised locally free  sheaf $ \mc V $ on some connected components of $ X^{T} $. Suppose that $ \mc V=\mc V^{\mf m} $. Then $ \mc V $ admits a \textit{generic representation type} $ [\mc V^{gen}] $ in the sense of \cite[Construction 2.7]{Levine_Atiyah-Bott}, which is an isomorphism class of $ N $-representations over $ \sk $. 
\end{rmk}

\begin{lemma}\label{ch4:_VLF_6.5}
	Let $ \mr A\in \mr{SH}(\sk) $ be an $ SL_{\eta} $-oriented ring spectrum. Consider $ X \in \catname{Sch}_{\bigslant{}{\sk}}^{N} $. Let $ \mc V $ an $ N $-linearised sheaf on $ X $ and suppose the $ N $-action on $ X $ is strict and we have $ \mc V=\mc V^{\mf m} $ on $ \abs{X}^{N} $. The restriction of $ \mc V $ on the connected components $ \abs{X}_{j,k}^{N} $ will be denoted by $ \mc V_{j,k} $. For fixed $ j,k $ we have:
	\begin{enumerate}
		\item there exists integers $ M,n $ such that $ e_N(\mc V^{gen}_{j,k}) $ is invertible in  $ \mr{A}^{\bullet}\left( BN \right)\left[ \left( M\cdot e^n \right)^{-1} \right]  $.
		\item the class $ e_N\left(  V_{j,k} \right) $ is invertible in $ \mr{A}^{\bullet}_N\left( \abs{X}^{N}_{j,k}; \det\!{}^{-1}\!\!\,\left(  V_{j,k} \right) \right)\left[ \left( M\cdot e^n \right)^{-1} \right] $
	\end{enumerate}
\end{lemma}
\begin{proof}
	The second claim is a consequence of  \cref{ch4:_ABL_9.3}. For the first claim it is enough to notice that $ \mc V^{gen}_{j,k} $ is given by irreducible $ \sk $-representations of $ N $, and  we know that we can recover any irreducible representation of $ N $ by tensor products of the representations $ \rho_m^{+}, \rho_0^{-} $. Then we conclude by applying multiple times \cref{ch4:_key_lemma_Bachmann_Hopkins}. 
\end{proof}

Let us recall again our setting: we are working with $ \iota: X \into Y $ an $ N $-equivariant closed immersion with $ Y $ smooth over $ \sk $; we are given an $ N $-linearised perfect obstruction theory $ \varphi_{\bullet}: \mc E_{\bullet} \longrightarrow \mb L_{\bigslant{X}{\sk}} $. We suppose that the $ N $-action on $ X $ is strict.

\begin{defn}\label{ch4:_VLF_6.6}
	Let $ \mr A\in \mr{SH}(\sk) $ be an $ SL_{\eta} $-oriented ring spectrum. Then we have:
	\begin{enumerate}
		\item [$ (i) $] For each connected component $ \iota_{j}^{Y}: {Y}_{j} \into Y $ of $ Y^{T} $, by \cref{ch4:_VLF_6.5}, there exists an integer $ M_{j}^{Y} $ such that 
		$ e_N(\iota_j^{Y}{}^*T_Y^{\mf m}) $ will be invertible in:
		\[ \mr{A}_N\left( Y_j; \det\!{}^{-1}\!\!\,\left( \iota_j^Y{}^*T_Y^{\mf m} \right) \right)\left[ \left( M_j^{Y}\cdot e \right)^{-1} \right]  \]
		
		\item [$ (ii) $] For each component $ \iota_{j,k}: \abs{X}_{j,k}^{N} \into X $ of $ \abs{X}^{N} $, by \cref{ch4:_VLF_6.5}, there exists an integer $ M_{j,k}^{X} $ such that 
		$e_N\left(\left(\iota_{j,k}^*E_1^{\mf m}\right)^{gen}\right)$ is invertible in:
		\[ \mr{A}_N\left( \abs{X}_{j,k}^{N}; \det\!{}^{-1}\!\!\,\left( \iota_{j,k}^*E_1^{\mf m}\right) \right)\left[ \left( M_{j,k}^{X}\cdot e \right) ^{-1}\right]  \]
		Hence we can define:
		\[ e_N\left( N_{\iota_{j,k}}^{vir} \right):= e_N(\left(\iota_{j,k}^*E_0^{\mf m}\right))\cdot e_N(\left(\iota_{j,k}^*E_1^{\mf m}\right))^{-1}  \]
		\noindent living in $ \mr{A}_N\left( \abs{X}_{j,k}^{N}; \det\!{}^{-1}\!\!\,\left( N_{\iota_{j,k}}^{vir} \right) \right)\left[ \left( M_{j,k}^{X}\cdot e \right) ^{-1}\right]  $.\\
		Denoting $ M^{X}_{j}:=\prod_k M^{X}_{j,k} $ we can also define:
		\[ e_N\left( N_{\iota_j}^{vir} \right):=\set{e_N\left( N_{\iota_{j,k}}^{vir} \right)}_{k} \in \mr{A}_N\left( \abs{X}_j^{N}; \det\!{}^{-1}\!\!\,\left( N_{\iota_j} ^{vir}\right) \right)\left[ \left(M^{X}_j\cdot e\right)^{-1} \right] \]
		\noindent where we use the identification:
		\begin{align*}
			\mr{A}_N\left( \abs{X}_j^{N}; \det\!{}^{-1}\!\!\,\left( N_{\iota_j} ^{vir}\right) \right)&\left[ \left(M^{X}_j\cdot e\right)^{-1} \right]\simeq\\
			\simeq & \prod_k \mr{A}_N\left( \abs{X}_{j,k}^{N}; \det\!{}^{-1}\!\!\,\left( N_{\iota_{j,k}} ^{vir}\right) \right)\left[ \left(M^{X}_{j,k}\cdot e\right)^{-1} \right] 
		\end{align*}
		
	\end{enumerate}
	
\end{defn}

\begin{rmk}
	From \cref{ch4:_ABL_8.6}, we also have an integer used in the Atiyah-Bott localization theorem for $ X $ that we will denote as $ M_0 $.
\end{rmk}

\begin{thm}[Virtual Localization Formula]\label{ch4:_VLF_for_general_A}
	Let $ \mr A\in \mr{SH}(\sk) $ be an $ SL_{\eta} $-oriented ring spectrum. Let $ \iota: X \into Y $ be a closed immersion in $ \catname{Sch}_{\bigslant{}{\sk}}^{N} $, with $ Y $ a smooth $ N $-scheme. Let $ \varphi_{\bullet}: \mc E_{\bullet} \rightarrow \mb L_{\bigslant{X}{\sk}} $ be an $ N $-linearised perfect obstruction theory. Suppose the $ N $-action on $ X $ is strict. With the notation introduced in \cref{ch4:_VLF_6.6}, we set:
	\[ M:=M_0 \cdot \prod_{i,j} M_i^{X}\cdot M_j^{Y} \] 
	Let $ \left[ \abs{X}_j^{N}, \varphi_{\bullet}^{(j)} \right]_N^{vir} \in \mr{A}^{\mr{BM}}_{\bullet, N}\left( \bigslant{\abs{X}_{j}^{N}}{\sk}, \iota^*_j \mc E_{\bullet}^{\mf f} \right)  $ the $ N $-equivariant virtual fundamental class for the $ N $-linearised perfect obstruction theory $ \varphi_{\bullet}^{(j)} $ on $ \abs{X}_j^{N} $. Then we have:
	\[ \left[X, \varphi \right]_N^{vir}=\sum_{j=1}^{s} \iota_j{}_*\left( \left[ \abs{X}_{j}^{N}, \varphi^{(j)} \right]_N^{vir} \cap e_N\left( N_{\iota_j}^{vir} \right)^{-1} \right) \in \mr{A}^{\mr{BM}}_N\left( X, \mc E_{\bullet} \right)\left[ \left(M \cdot e\right)^{-1} \right] \]
\end{thm}

\begin{proof}
	We have developed along the way all the necessary tools used in the proof \cite[Theorem 6.7]{VLF_Levine}, upgrading them to the case of an $ SL_{\eta} $-oriented ring spectrum. Then the very same strategy used in \textit{loc. cit.} works also in our case. For completeness, we will sketch now the proof for a general $ SL_{\eta} $-oriented ring spectrum, but no claim of originality is made here.\\
	
	Now consider $ X_j=Y_j \cap X^{T} $. By \cref{ch4:_ABL_8.5}, the localised Borel-Moore homology of $ X^{T}_{ind} $ vanishes. By assumption, our action is strict and thus we have $ X_j=\abs{X}_j^{N}\amalg X_j \cap X^{T}_{ind} $. By a localization sequence argument, we can replace $ X $ with $ X \setminus X^{T}_{ind} $ and $ Y $ with $ Y\setminus X^{T}_{ind} $, so without loss of generality we can assume $ X^{T}_{ind}=\emptyset $ and $ X_j=\abs{X}^{N}_j $. Let us denote the following inclusions: \\
	\begin{minipage}{0.5\textwidth}
		\begin{center}
			$ \iota: X\into Y $\\
			$ \iota_j: X_j \into Y_j $
		\end{center}
	\end{minipage}
	\hfill
	\begin{minipage}{0.45\textwidth}
		\begin{center}
			$ i_j: X_j \into X $\\
			$ i_j^{Y}: Y_j\into Y $
		\end{center}
	\end{minipage}\\
	The $ N $-equivariant fundamental class $ [Y]_N $ of $ Y $  lives in $ \mr A^{\mr{BM}}_{N}\left(\bigslant{Y}{S}, \Omega_{\bigslant{Y}{\sk}}\right) $. The fixed locus $ Y^{T} $ is smooth over $ \sk $, with connected components $ Y_1,\ldots, Y_s $. Let $ [Y_j]_N \in \mr A^{\mr{BM}}_N(\bigslant{Y_j}{\sk}, \Omega_{\bigslant{Y_j}{\sk}}) $ be the fundamental classes of the connected components. Since the normal bundle associated to $ \iota_j^{Y}: Y_j \into Y $ is given by $ (\iota_j^{Y})^*T_{\bigslant{Y}{\sk}}^{\mf m} $, by \cref{ch4:_ABL_9.5} and \cref{ch4:_ABL_9.6} we have that:
	\[ [Y]_N=\sum_{j=1}^{s} (\iota_j^{Y})_*\left([Y_j]_N \cap e_N\left( (\iota_j^{Y})^*T_{\bigslant{Y}{\sk}}^{\mf m} \right)^{-1} \right)\in \mr A^{\mr{BM}}_{N}\left( \bigslant{Y}{\sk}, \Omega_{\bigslant{Y}{\sk}} \right)\left[ \left( M_Ye \right)^{-1} \right] \]
	\noindent where we used the functoriality of lci fundamental classes (cf. \cite[Theorem 4.2.1]{DJK}) to identify $ (\iota_j^{Y})^{!}[Y]_N=[Y_j]_N $. Consider the following cartesian square:
	\begin{center}
		\begin{tikzpicture}[baseline={(0,0)}, scale=1.5]
			\node (a) at (0,1) {$  X $};
			\node (b) at (1, 1) {$ Y $};
			\node (c)  at (0,0) {$ Y_j  $};
			\node (d) at (1,0) {$ Y $};
			\node (e) at (0.2,0.8) {$ \ulcorner $};
			\node (f) at (0.5,0.5) {$ \Delta $};

			\path[font=\scriptsize,>= angle 90]
			
			(a) edge [closed] node [above ] {$ \iota $} (b)
			(a) edge [->] node [left] {$  $} (c)
			(b) edge[->] node [right] {$ Id_Y $} (d)
			(c) edge [closed] node [below] {$ \iota_j^{Y} $} (d);
		\end{tikzpicture}
	\end{center}
	If we take the refined intersection product with respect to $ \iota: X \into Y $ and $ Id_Y: Y\rightarrow Y $, by compatibility with proper pushforwards (\cite[3.16]{Motivic_Vistoli}), from the equation above we get:
	\[ [X,\varphi]_N^{vir}=[X,\varphi]_N^{vir}*_{\iota, Id_Y} [Y]_N=\sum_{j=1}^{s} (\iota_j^{Y})_* \left([X, \varphi]_N^{vir}*_{\iota, \iota_j^{Y}} [Y_j]_N\cap e_N\left( (\iota_j^{Y})^* T_{\bigslant{Y}{\sk}}^{\mf m} \right)^{-1}  \right)  \]
	We will then only need to show that:
	\begin{claim}\label{ch4:_VLF_final_claim}
		\[ [X, \varphi]_N^{vir}*_{\iota, \iota_j^{Y}} [Y_j]_N\cap e_N\left( (\iota_j^{Y})^* T_{\bigslant{Y}{\sk}}^{\mf m} \right)^{-1} = [X_j, \varphi^{(j)}]_N^{vit}\cap e_N\left( N_{\iota_j}^{vir} \right)^{-1} \]
	\end{claim}

	\begin{notation}
		To make the rest of the proof easier to read, we will change our notation a bit. Given a cartesian square:
		\begin{center}
			\begin{tikzpicture}[baseline={(0,0)}, scale=1.5]
				\node (a) at (0,1) {$  X $};
				\node (b) at (1, 1) {$ Y $};
				\node (c)  at (0,0) {$ Z $};
				\node (d) at (1,0) {$ W $};
				\node (e) at (0.2,0.8) {$ \ulcorner $};
				\node (f) at (0.5,0.5) {$ \Delta $};

				\path[font=\scriptsize,>= angle 90]
				
				(a) edge [closed] node [above ] {$ g $} (b)
				(a) edge [->] node [left] {$  $} (c)
				(b) edge[->] node [right] {$  $} (d)
				(c) edge [closed] node [below] {$ f $} (d);
			\end{tikzpicture}
		\end{center}
		\noindent where $ f $ is lci, we will denote the refined Gysin map as $ f^!:=g_{\Delta}^! $. And for any given vector bundle $ E \rightarrow W $, we will denote its zero section as $ s_{E}: W \into E $.
	\end{notation}

	Let us briefly recall our construction of the equivariant virtual fundamental class from \cref{ch1:_EGT_VFC_Construction}. We have the cone $ D:=\mf C_{\bigslant{X}{Y}}\times \mb V(\mc E_0)=C_{\bigslant{X}{Y}}\times E_0 $, with quotient $ D^{vir}:=\bigslant{D}{\iota^*T_{\bigslant{Y}{\sk}}} $. The virtual cone $ D^{vir} $ has a closed immersion $ \iota_{\varphi}: D^{vir} \into E_1:=\mb V(\mc E_1) $ in $ \catname{Sch}_{\bigslant{}{\sk}}^{N} $. Then the virtual class was defined as:
	\[ [X, \varphi]_N^{vir}:=s_{E_1}^!\left( (\iota_{\varphi})_* [D^{vir}]_N \right)\in \mr A^{\mr{BM}}_{N}\left( \bigslant{X}{\sk}, E_{\bullet} \right) \]
	For each $ X_j $ the $ N $-linearised obstruction theory is given by $ \varphi^{(j)}: i_j^*\mc E_{\bullet}^{\mf f}\rightarrow \tau_{\leq 1} \mb L_{\bigslant{X_j}{\sk} }$. Denoting by $ D_j:=\mf C_{\bigslant{X_j}{Y_j}} \times \iota_j^*E_0^{\mf f} $ and $ D_j^{vir}:=\bigslant{D_j^{vir}}{\iota_j^*T_{Y_j}} $ the corresponding cones, we have:
	\[ [X_j, \varphi^{(j)}]_N^{vir}:=s_{i_j^*E_1^{\mf f}}^{!}\left( (\iota_{\varphi^{(j)}})_* [D_j^{vir}]_N \right) \in \mr A^{\mr{BM}}_{N}\left( \bigslant{X_j}{\sk}, E_{\bullet}^{\mf f} \right)  \]
	
	Let us start with the proof of the following:
	\begin{claim}\label{ch4:_VLF_claim_1}
		Consider the cartesian squares:\\
		\begin{minipage}{0.23\textwidth}
			\begin{center}
				\begin{tikzpicture}[baseline={(0,0)}, scale=2]
					\node (a) at (0,1) {$ \iota^*T_Y $};
					\node (b) at (1.5, 1) {$ D $};
					\node (c)  at (0,0) {$  X  $};
					\node (d) at (1.5,0) {$ E_1 $};
					\node (e) at (0.2,0.8) {$ \ulcorner $};
					\node (f) at (0.75,0.5) {$ \Delta_{1} $};

					\path[font=\scriptsize,>= angle 90]
					
					(a) edge [->] node [above ] {$ t_Y $} (b)
					(a) edge [->] node [left] {$  $} (c)
					(b) edge[->] node [right] {$  $} (d)
					(c) edge [->] node [below] {$ $} (d)
					(c) edge [bend left=30, ->] node [left] {$ {s}_{\iota^*T_Y} $} (a);
				\end{tikzpicture}
			\end{center}
		\end{minipage}
		\hfill
		\begin{minipage}{0.13\textwidth}
			\begin{center}
				\begin{tikzpicture}[baseline={(0,0)}, scale=2]
					\node (a) at (0,1) {$ \iota_j^*T_{Y_j}$};
					\node (b) at (1.5, 1) {$  D_j  $};
					\node (c)  at (0,0) {$  X_j  $};
					\node (d) at (1.5,0) {$ i_j^*E_1^{\mf f} $};
					\node (e) at (0.2,0.8) {$ \ulcorner $};
					\node (f) at (0.75,0.5) {$ \Delta_{2} $};

					\path[font=\scriptsize,>= angle 90]
					
					(a) edge [->] node [above ] {$ t_{Y_j} $} (b)
					(a) edge [->] node [left] {$ $} (c)
					(b) edge[->] node [right] {$  $} (d)
					(c) edge [->] node [below] {$  $} (d);
				\end{tikzpicture}
			\end{center}
		\end{minipage}
		\hfill
		\begin{minipage}{0.33\textwidth}
			\begin{center}
				\begin{tikzpicture}[baseline={(0,0)}, scale=2]
					\node (a) at (0,1) {$ i_j^*\iota^*T_{Y} $};
					\node (b) at (1.5, 1) {$  D_j \times i_j^*E_0^{\mf m} $};
					\node (c)  at (0,0) {$  X_j $};
					\node (d) at (1.5,0) {$ i_j^*E_1$};
					\node (e) at (0.2,0.8) {$ \ulcorner $};
					\node (f) at (0.75,0.5) {$ \Delta_{Y_j} $};

					\path[font=\scriptsize,>= angle 90]
					
					(a) edge [->] node [above ] {$  $} (b)
					(a) edge [->] node [left] {$ $} (c)
					(b) edge[->] node [right] {$  $} (d)
					(c) edge [->] node [below] {$  $} (d)
					(c) edge [bend left=30, ->] node [left] {$  $} (a);
				\end{tikzpicture}
			\end{center}
		\end{minipage}\\
		The refined Gysin pullbacks associated to the squares above will give us:
		\[ [X, \varphi]^{vir}_N=s_{\iota^*T_Y}^! s_{E_1}^![D]_N \]
		\[ [X_j, \varphi^{(j)}]^{vir}_N=s_{\iota_j^*T_{Y_j}}^!s_{i_j^*E_1^{\mf f}}^![D_j]_N \]
		\[ [X, \varphi]_N^{vir}*_{\iota, i_j^{Y}} [Y_j]_N= s_{i_j^*\iota^*T_Y}^!s_{i_j^*E_1}^!\left[ D_j \times i_j^*E_0^{\mf m} \right]  \]
		
	\end{claim}
	The first two equations in \cref{ch4:_VLF_claim_1} follow from the squares $ \Delta_1, \Delta_2 $ and the functoriality of (refined) Gysin pullbacks (cf. \cite[Remark 3.13]{Motivic_Vistoli}). Let $ \beta_j^{D}: D_j \times i_j^*E_0^{\mf m} \into D $ be the closed immersion of cones, then the equivariant Vistoli's lemma \cite[Proposition 4.16]{Motivic_Vistoli} tells us that:
	\begin{equation}\label{ch4:_Vistoli's_eq}
		(\beta_j^{D})_*\left[ D_j \times i_j^*E_0^{\mf m} \right]_N=\pi_{Y_j}^![D]_N
	\end{equation}
	\noindent where $ \pi_{Y_j}^! $ denotes the refined Gysin pullback with respect to $ \pi_{Y_j}: Y_j \rightarrow S $. 
	Then using \cref{ch4:_Vistoli's_eq} and the commutativity of refined Gysin pullbacks, we get the third equation in \cref{ch4:_VLF_claim_1} (see \cite[Theorem 6.7, Proof: Step 4]{VLF_Levine} for more details). \\
	\begin{flushright}
		$ \stackrel{\text{(Claim 2)}}{\blacksquare} $
	\end{flushright}
	
	Notice that $ (i_j^{Y})^*T_{Y}^{\mf f}\simeq T_{Y_j}  $, hence $ (i_j^{Y})^*T_Y\simeq T_{Y_j}\oplus (i_j^{Y})^*T_{Y}^{\mf m}  $. Then we have:
	\[ s_{i_j^*\iota^*T_Y}^! s_{i_j^*E_1}^!\left[ D_j \times i_j^*E_0^{\mf m} \right] = {s}_{i_j^*\iota^*T_{Y}^{\mf m}}^{!}s_{i_j^*E_1}^!\left[ D_j^{vir} \times i_j^*E_0^{\mf m} \right]   \]
	\begin{claim}\label{ch4:_VLF_Claim_2}
		We have that:
		\[ [X, \varphi]_N^{vir}*_{\iota, i_j^Y}[Y_j]_N \cap e_N\left( i_j^*E_0^{\mf m} \right)=s_{i_j^*E_0^{\mf m}}^!\left( s_{i_j^*E_1}^![D_j^{vir}\times i_j^*E_0^{\mf m}]_N \right)\cap e_N\left( i_j^*\iota^*T_{Y}^{\mf m} \right) \]
	\end{claim}
	
	We have a closed immersion $ D_j^{vir}\times_{X_j} i_j^*E_0^{\mf m}\into \bigslant{i_j^*D}{\iota^*_jT_{Y_j}} $ and composing this map with the natural map $ \bigslant{i_j^*D}{\iota^*_jT_{Y_j}} \rightarrow i_j^*E_1  $, we get a map $ \sigma: D_j^{vir}\times_{X_j} i_j^*E_0^{\mf m} \rightarrow i_j^*E_1 $. Denote by $ \mc Z_{\sigma}\left( D_j^{vir}\times_{X_j} i_j^*E_0^{\mf m} \right) $ the scheme-theoretic pullback of $ D_j^{vir}\times_{X_j} i_j^*E_0^{\mf m} $ along the zero section of $ i_j^*E_1 $. There exists a commutative diagram in $ \catname{Sch}_{\bigslant{}{\sk}}^{N} $:
	\begin{center}
		\begin{tikzpicture}[baseline={(0,0)}, scale=1.75]
			\node (a) at (0,1) {$ \mc Z_{\sigma}\left( D_j^{vir}\times_{X_j} i_j^*E_0^{\mf m} \right) $};
			\node (b) at (2, 1) {$  i_j^*E_0^{\mf m}$};
			\node (c)  at (0,0) {$  i_j^*\iota^*T_{Y}^{\mf m} $};
			\node (d) at (2,0) {$ X_j $};
			\node (e) at (0.2,0.68) {$  $};
			\node (f) at (1,0.5) {$  $};

			\path[font=\scriptsize,>= angle 90]
			
			(a) edge [closed] node [above ] {$ f $} (b)
			(a) edge [closed] node [left] {$ g $} (c)
			(b) edge[->] node [right] {$  $} (d)
			(c) edge [->] node [below] {$  $} (d);
		\end{tikzpicture}
	\end{center}
	Let $ \alpha:=s_{i_j^*E_1}^!\left[ D_j^{vir}\times_{X_j} i_j^*E_0^{\mf m} \right]  $.
	By the (equivariant) excess intersection formula (cf. \cite[Prop. 3.3.4]{DJK}), we have:
	\[ s_{i_j^*\iota^*T_{Y}^{\mf m}}\left( f_*(\alpha)\cap e_N\left( i_j^*E_0^{\mf m} \right) \right)=s_{i_j^*E_0^{\mf m}}\left( g_*(\alpha)\cap e_N\left( i_j^*\iota^*T_{Y}^{\mf m} \right) \right) \]
	\noindent and this gives us the formula of \cref{ch4:_VLF_Claim_2}. \begin{flushright}
		$ \stackrel{\text{(Claim 3)}}{\blacksquare} $
	\end{flushright}
	Now we can finally prove \cref{ch4:_VLF_final_claim}. We have a natural map $ (\iota_{\varphi^{(j)}}, d^{\mf m}): D_j^{vir}\times i_j^*E_{0}^{\mf m} \rightarrow i_j^*E_1^{\mf m} $, induced by the inclusion $ \iota_{\varphi^{(j)}} $ and by the "moving" differential $ d^{\mf m} $ on $ \mc E_{\bullet}^{\mf m} $. By $ \A^1 $-homotopy invariance, we can suppose that $  (\iota_{\varphi^{(j)}}, d^{\mf m}) $  factors via the first coordinate projection $ D_j^{vir}\times i_j^*E_{0}^{\mf m}\stackrel{p_1}{\rightarrow} D_j^{vir} $, and the natural inclusion map $ D_j^{vir}\into i_j^*E_1^{\mf m}\sseq i_j^*E_1 $. By the equivariant excess intersection formula, we have:
	\begin{align*}
		s_{i_j^*E_0^{\mf m}}^!s_{i_j^*E_1^{\mf m}}^!\left[D_j^{vir}\times i_j^*E_0^{\mf m}\right]_N&=\left( s_{i_j^*E_1^{\mf m}}^!\left[ D_j^{vir}\times i_j^*E_0^{\mf m} \right] \right)\cap e_N\left( i_j^*E_0^{\mf m} \right)=\\
		&=\left[X_j, \varphi^{(j)}\right]_N^{vir}\cap e_N\left( i_j^*E_1^{\mf m} \right)
	\end{align*}
	Putting together all the statements we proved, we got:
	\begin{align*}
		[X, \varphi]_N^{vir}*_{\iota, \iota_j^{Y}} [Y_j]_N &\cap e_N\left( (\iota_j^{Y})^* T_{\bigslant{Y}{\sk}}^{\mf m} \right)^{-1}=\\
		&\stackrel{}{=} s_{i_j^*\iota^*T_Y}^!s_{i_j^*E_1}^!\left[ D_j \times i_j^*E_0^{\mf m} \right]\cap e_N\left( (\iota_j^{Y})^* T_{\bigslant{Y}{\sk}}^{\mf m} \right)^{-1}=\\
		&={s}_{i_j^*\iota^*T_{Y}^{\mf m}}^{!}s_{i_j^*E_1}^!\left[ D_j^{vir} \times i_j^*E_0^{\mf m} \right]\cap e_N\left( (\iota_j^{Y})^* T_{\bigslant{Y}{\sk}}^{\mf m} \right)^{-1}  =\\
		&=s_{i_j^*E_0^{\mf m}}^!\left( s_{i_j^*E_1}^![D_j^{vir}\times i_j^*E_0^{\mf m}]_N \right)\cap \left( e_N\left( i_j^*E_0^{\mf m} \right)  \right)^{-1}=\\
		&=\left[X_j, \varphi^{(j)}\right]_N^{vir}\cap e_N\left( i_j^*E_1^{\mf m} \right)\cap \left( e_N\left( i_j^*E_0^{\mf m} \right)  \right)^{-1}
	\end{align*}
	\noindent that is exactly our \cref{ch4:_VLF_final_claim}. Thus we have just  proved our virtual localization theorem.
	
\end{proof}

\begin{co}[Virtual Localization For Witt Theory]\label{ch4:_VLF_for_KW}
	In the same situation as in \cref{ch4:_VLF_for_general_A}, for $ \mr A=\mr{KW} $ we get:
	\[ \left[X, \varphi \right]_N^{vir}=\sum_{j=1}^{s} \iota_j{}_*\left( \left[ \abs{X}_{j}^{N}, \varphi^{(j)} \right]_N^{vir} \cap e_N\left( N_{\iota_j}^{vir} \right)^{-1} \right) \in \mr{KW}^{\mr{BM}}_N\left( X, E_{\bullet} \right)\left[ \left(M \cdot e\right)^{-1} \right] \]
\end{co}

	\cleardoublepage
	\addcontentsline{toc}{section}{References}

	\printbibliography	\thispagestyle{empty}
	
\end{document}